\documentclass{amsart} 

\usepackage{amsmath,amssymb,amsthm} 
\usepackage{graphicx} 
\usepackage{subfig}
\usepackage[arrow, matrix, curve]{xy} 
\usepackage{color} 
\usepackage{mathtools}
\usepackage{booktabs} 
\usepackage{rotating}
\usepackage{pinlabel}
\usepackage{hyperref} 
\usepackage[export]{adjustbox}

\DeclareMathOperator{\tb}{tb}
\DeclareMathOperator{\rot}{rot}

\DeclareMathOperator{\de}{d}
\DeclareMathOperator{\e}{e}

\DeclareMathOperator{\cs}{cs}

\newcommand{\Z}{\mathbb{Z}}
\newcommand{\Q}{\mathbb{Q}}
\newcommand{\N}{\mathbb{N}}

\newcommand{\xist}{\xi_{\mathrm{st}}}

\newcommand{\Pushoff}{{\def\svgwidth{1,6ex}\,\,
\begingroup%
  \makeatletter%
  \providecommand\color[2][]{%
    \errmessage{(Inkscape) Color is used for the text in Inkscape, but the package 'color.sty' is not loaded}%
    \renewcommand\color[2][]{}%
  }%
  \providecommand\transparent[1]{%
    \errmessage{(Inkscape) Transparency is used (non-zero) for the text in Inkscape, but the package 'transparent.sty' is not loaded}%
    \renewcommand\transparent[1]{}%
  }%
  \providecommand\rotatebox[2]{#2}%
  \ifx\svgwidth\undefined%
    \setlength{\unitlength}{49.54732089bp}%
    \ifx\svgscale\undefined%
      \relax%
    \else%
      \setlength{\unitlength}{\unitlength * \real{\svgscale}}%
    \fi%
  \else%
    \setlength{\unitlength}{\svgwidth}%
  \fi%
  \global\let\svgwidth\undefined%
  \global\let\svgscale\undefined%
  \makeatother%
  \begin{picture}(1,1.00059637)%
    \put(0,0){\includegraphics[width=\unitlength,page=1]{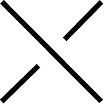}}%
  \end{picture}%
\endgroup%
\,\,}} 

\usepackage{amsthm}
\usepackage{thmtools}
\declaretheoremstyle[notefont=\bfseries,notebraces={}{},%
    headpunct={},postheadspace=1em]{mystyle}
\declaretheorem[style=mystyle,numbered=no,name=Theorem]{thm-hand}

\newtheoremstyle{thm}{}{}{\itshape}{}{\bfseries}{}{ }{} 
\newtheoremstyle{definition}{}{}{}{}{\bfseries}{}{ }{} 

\theoremstyle{thm}
\newtheorem{Theorem}{Theorem}[section]
\newtheorem{theorem}[Theorem]{Theorem}
\newtheorem{lemma}[Theorem]{Lemma}

\newtheorem{corollary}[Theorem]{Corollary}

\newtheorem*{theorem*}{Theorem}

\theoremstyle{definition}

\begingroup\expandafter\expandafter\expandafter\endgroup
\expandafter\ifx\csname pdfsuppresswarningpagegroup\endcsname\relax
\else
\fi

\begin{document}


\title[Contact surgery numbers of E(2,3,11) and L(4m+3,4)]{Contact surgery numbers of $\Sigma(2,3,11)$ and $L(4m+3,4)$}

\author{Rima Chatterjee}
\address{Mathematisches Institut, Universit\"at zu K\"oln,
	Weyertal 86--90, 50931 K\"oln, Germany}
\email{rchatt@math.uni-koeln.de}
\author{Marc Kegel}
\address{Humboldt Universit\"at zu Berlin, Rudower Chaussee 25, 12489 Berlin, Germany}
\address{Ruhr-Universit\"at Bochum, Universit\"atsstra{\ss}e 150, 44780 Bochum, Germany}
\email{kegemarc@hu-berlin.de, kegelmarc87@gmail.com}


\date{\today} 


\begin{abstract}
We classify all contact structures with contact surgery number one on the Brieskorn sphere $\Sigma(2,3,11)$ with both orientations. We conclude that there exist infinitely many non-isotopic contact structures on each of the above manifolds which cannot be obtained by a single rational contact surgery from the standard tight contact 3-sphere. We further prove similar results for some lens spaces: We classify all contact structures with contact surgery number one on lens spaces of the form $L(4m+3,4)$. Along the way, we present an algorithm and a formula for computing the Euler class of a contact structure from a general rational contact surgery description and classify which rational surgeries along Legendrian unknots are tight and which ones are overtwisted.
\end{abstract}

\makeatletter
\@namedef{subjclassname@2020}{%
  \textup{2020} Mathematics Subject Classification}
\makeatother


\maketitle

\section{Introduction}
A fundamental result in $3$-dimensional contact topology due to Ding--Geiges says that any connected orientable compact contact $3$-manifold with co-orientable contact structure can be obtained by contact surgery on a Legendrian link in the standard tight contact $3$-sphere $(S^3,\xist)$~\cite{Ding_Geiges_Surgery}. Moreover, one can assume all contact surgery coefficients to be $\pm1$. Thus a natural complexity notion for a given contact $3$-manifold is the minimal number of components required for the surgery link. 

The \textit{contact surgery number} $\cs(M, \xi)$ of a given contact $3$-manifold $(M,\xi)$ is defined to be the minimal number of components of a Legendrian link $L$ in $(S^3, \xist)$ needed to describe $(M,\xi)$ as a rational contact surgery along $L$ (with non-vanishing contact surgery coefficients). Similarly, we define other versions of contact surgery numbers, i.e.~$\cs_\Z(M, \xi)$, $\cs_{1/\Z}(M, \xi)$, and $\cs_{\pm1}(M, \xi)$ by requiring in addition that all contact surgery coefficients of $L$ are integers, reciprocal integers, or $\pm1$, respectively.

In~\cite{EKS_contact_surgery_numbers} the study of contact surgery numbers was initiated. In particular, explicit calculations on some simple manifolds were performed and general upper bounds on contact surgery numbers were obtained. For example it was shown that the contact surgery number of a contact manifold $(M,\xi)$ is at most $3$ larger than the topological surgery number of the underlying smooth manifold $M$. On the other hand, it remained unclear if this bound is sharp. And especially for tight contact manifolds certain properties of contact surgery numbers remained unclear.

In this paper, we will add new computations of contact surgery numbers on certain Brieskorn spheres and lens spaces. In particular, these results will add some previously unknown phenomena about the behavior of contact surgery numbers of tight contact structures.

\subsection{The Brieskorn sphere \texorpdfstring{$\Sigma(2,3,11)$}{E(2,3,11)}}
Let $\Sigma(2,3,11)$ be the Brieskorn sphere with Brieskorn invariants $(2,3,11)$, i.e.\ the Seifert fibered space presented as the surgery diagram in Figure~\ref{fig:topological_surgery_diagrams}. Up to contactomorphism there exists a unique Stein fillable (hence tight) contact structure on $\Sigma(2,3,11)$ which we denote by $\xist$~\cite[Theorem~4.4]{Ghiggini_Schoeneberger}. Since $\Sigma(2,3,11)$ is a homology sphere every overtwisted contact structure is determined by its $\de_3$-invariant, which can take all integer values. Our first result completely classifies all contact structures on this manifold with contact surgery number $1$.

\begin{theorem}\label{thm:Brieskorn}\hfill
\begin{enumerate}
    \item A contact structure on the Brieskorn sphere $\Sigma(2,3,11)$ has integer contact surgery number $\cs_\Z=1$ if and only if its $\de_3$-invariant is of the form
\begin{equation*}
m(3-m)-1 \,\,\,\textrm{ for }\,\,\, m\geq 4.    
\end{equation*}
\item A contact structure on the Brieskorn sphere $\Sigma(2,3,11)$ has rational contact surgery number $\cs=1$ if and only if its $\de_3$-invariant is of the form
\begin{align*}
m(3-m)-1 \,\,\,&\textrm{ for }\,\,\, m\geq 4,\\
-2\big(m(m+2)+2\big)\,\,\,&\textrm{ for }\,\,\, m\geq -3, \,\textrm{ or}\\
-2\big(m(m+3)+3\big)\,\,\,&\textrm{ for }\,\,\, m\leq -3.
\end{align*}
\end{enumerate}
\end{theorem}

In the proof of Theorem~\ref{thm:Brieskorn} we will also describe all possible contact surgery diagrams of these contact structures along a single Legendrian knot. From that, we will deduce the following corollary.

\begin{corollary}\label{cor:Brieskorn1}
Let $\xi$ be any overtwisted contact structure on $\Sigma(2,3,11)$, then 
\begin{align*}
    2\leq\cs_{1/\Z}(\xi)\leq\cs_{\pm1}(\xi)\leq3.
\end{align*}
\end{corollary}

We also have similar results for the tight contact structure.

\begin{corollary}\label{cor:Brieskorn_tight}
Let $\xist$ be the standard tight contact structure on $\Sigma(2,3,11)$, then 
\begin{align*}
\cs(\xist)=2, \,\,\,2\leq \cs_{1/\Z}(\xist)\leq 3, \,\,\,\textrm{ and }\,\,\,2\leq \cs_{\Z}(\xist)\leq\cs_{\pm1}(\xist)\leq 4.
\end{align*}
A rational contact surgery diagram of $(\Sigma(2,3,11),\xist)$ along a $2$-component Legendrian link is shown in Figure~\ref{fig:contact_Brieskorn}.
\end{corollary}

We remark that this is the first tight contact structure for which we can explicitly compute the rational contact surgery number and where it does not agree with the topological surgery number of the underlying smooth manifold.

\subsection{The mirror of the Brieskorn sphere \texorpdfstring{$\Sigma(2,3,11)$}{E(2,3,11)}}
Again up to contactomorphism there exists a unique Stein fillable (hence tight) contact structure on $-\Sigma(2,3,11)$ which we denote by $\xist$~~\cite[Theorem~4.9]{Ghiggini_Schoeneberger}, cf.~\cite{Ghiggini_vanHornMorris}. Again we can completely classify the contact structures with contact surgery number $1$.

\begin{theorem}\label{thm:-Brieskorn}\hfill
\begin{enumerate}
    \item A contact structure on $-\Sigma(2,3,11)$ has integer contact surgery number $\cs_\Z=1$ if and only if its $\de_3$-invariant is of the form
\begin{equation*}
m(m-1) \,\,\,\textrm{ for }\,\,\, m\geq 0.    
\end{equation*}
\item An overtwisted contact structure on $-\Sigma(2,3,11)$ has rational contact surgery number $\cs=1$ if and only if its $\de_3$-invariant is of the form
\begin{align*}
m(m-1) \,\,\,&\textrm{ for }\,\,\, m\geq 0,\\
2m(m+1)\,\,\,&\textrm{ for }\,\,\, m\leq -1,\textrm{ or }\\
2(m+1)^2\,\,\,&\textrm{ for }\,\,\, m\leq -1.
\end{align*}
\end{enumerate}
\end{theorem}

From Theorem~\ref{thm:-Brieskorn} we will deduce that some contact structures have unique contact surgery diagrams along Legendrian knots.

\begin{corollary}\label{cor:-Brieskorn1}
If $K$ is a Legendrian knot in $(S^3,\xist)$ such that contact $(+1)$-surgery or contact $(-1)$-surgery along $K$ yields a contact structure $\xi$ on $-\Sigma(2,3,11)$. Then $K$ is isotopic to the unique Legendrian realization of the twist knot $-K5a1$ with $\tb=0$ and $\rot=0$, the contact surgery coefficient is $+1$, and $\xi$ is isotopic to the overtwisted contact structure with $\de_3=0$. This contact surgery diagram is shown in Figure~\ref{fig:contact_-Brieskorn}.
\end{corollary}

If we allow reciprocal integer contact surgery coefficients, then this contact structure has exactly one more surgery diagram: contact $(1/2)$-surgery along the Legendrian realization of the right-handed trefoil with $\tb=0$ and $\rot=1$, shown in Figure~\ref{fig:contact_-Brieskorn}.

For the tight contact structure, we can compute all contact surgery numbers.

\begin{corollary}\label{cor:-Brieskorn_tight}
If $K$ is a Legendrian knot in $(S^3,\xist)$ such that for some $r\in\Q\setminus\{0\}$, contact $r$-surgery $K(r)$ is contactomorphic to $(-\Sigma(2,3,11),\xist)$, then $K$ is isotopic to the right-handed trefoil with $\tb=1$ and $\rot=0$ and $r=-1/2$.

In particular, we compute its contact surgery numbers as
\begin{align*}
\cs(\xist)=\cs_{1/\Z}(\xist)=1, \,\,\,\textrm{ and }\,\,\,\cs_{\pm1}(\xist)=\cs_{\Z}(\xist)=2.
\end{align*}
\end{corollary}

We remark that this is the first tight contact structure for which we have explicitly computed all different versions of the contact surgery numbers and where $\cs_{\pm1}$ and $\cs_{1/\Z}$ differ.

\subsection{Integral surgery numbers of lens spaces}
For $m\geq1$, we denote by $L(4m+3,4)$ the lens space obtained by topological $(-m-\frac{3}{4})$-surgery on the unknot $U$. We call this surgery diagram the \textit{standard surgery diagram} of $L(4m+3,4)$. 
If we denote by $\mu_U$ the meridian of this unknot we know that $H_1(L(4m+3,4))$ is isomorphic to $\Z_{4m+3}$ generated by $\mu_U$.
The Euler class of a $2$-plane field on $L(4m+3,4)$ is Poincar\'e dual to a first homology class which we always present in this identification with $\Z_{4m+3}$. In other words if we write that a contact structure $\xi$ on a lens space of the form $L(4m+3,4)$ has Euler number $e$ for an integer $e$, then we consider $e$ modulo $4m+3$ and see it as an element in $\Z_{4m+3}$ which we identify with $H_1(L(4m+3,4))$ under the above identification, i.e. we identify $e$ with the element $e\mu \in H_1(L(4m+3,4))$. 

The tight contact structures on $L(4m+3,4)$ are classified by Honda~\cite{Honda_lens}. In particular, any tight contact structure can be obtained by rational contact surgery on a Legendrian unknot (and thus has $\cs=1$), two tight contact structures on a lens space are distinguished by their Euler classes.

For the tight contact structures, we can compute all contact surgery numbers. In particular, there are exactly $m$ non-contactomorphic tight contact structures on $L(4m+3,4)$ that have integer contact surgery number one.

\begin{theorem}\label{thm:lens_tight}
	Let $\xi_k$ be the tight contact structure $\xi$ on the lens space $L(4m+3,4)$ with Euler class $k\in \Z_{4m+3}=H_1(L(4m+3,3))$. Then their contact surgery numbers are as follows:
	\begin{align*}
		\cs_{\pm1}(\xi_k)=\cs_{\Z}(\xi_k)=1,& \,\,\textrm{  if }\,\, k\equiv2m(l+1)+2+4l\,\operatorname{mod}(4m+3), \\
  &\,\,\textrm{  for } l\in\{0,1,2,\ldots,m-1\},\\
		\cs_{\pm1}(\xi_k)=\cs_{\Z}(\xi_k)=2, &\,\,\textrm{  otherwise.}
	\end{align*} 
\end{theorem}

We conclude that several tight contact structures on these lens spaces have unique contact surgery diagrams along a single Legendrian knot.

\begin{corollary}\label{thm:cor_lens}
	Let $l\in\{0,1,2,\ldots,m-1\}$ and let $K$ be a Legendrian knot in $(S^3,\xist)$ such that for $n\in\Z$ contact $n$-surgery along $K$ yields $(L(4m+3,4),\xi_{2m(l+1)+2+4l})$. Then $K$ is the unique Legendrian realization of $T_{(2,-(2m+1))}$ with $\tb=-4m-2$ and $\rot=1+2l$ and the contact surgery coefficient $n$ is $-1$.
\end{corollary}

For the overtwisted contact structures we can completely classify the contact structures with $\cs_{\pm1}=1$ and with $\cs_{\Z}=1$. In general it follows from~\cite{EKS_contact_surgery_numbers} that any overtwisted contact structure on $L(4m+3,3)$ has $\cs_{\pm1}\leq3$. 
Since the first homology $H_1(L(4m+3,4))$ has no $2$-torsion overtwisted contact structures are completely determined by their Euler class and their $\de_3$-invariants~\cite{Gompf_Stein,Ding_Geiges_Stipsicz}. 

\begin{theorem}\label{thm:lens_overtwisted}\hfill
	\begin{enumerate}
		\item An overtwisted contact structure on $L(4m+3,4)$ has contact surgery number $\cs_{\pm1}=1$ if and only if its tuple $(\e,\de_3)$ of Euler class $\e$ and $\de_3$-invariant is of the form
		\begin{equation*}
			\left(2m+2+k,\frac{3m+2-k(k+1)}{4m+3} +\frac{1}{2}       \right)\,\,\textrm{ for }\,\,k\in\{-1,0,1,2,\ldots,m\}.
		\end{equation*}
		\item An overtwisted contact structure on $L(4m+3,4)$ has contact surgery number $\cs_{\Z}=1$ if and only if it appears in the list of $(1)$ or if its tuple $(\e,\de_3)$ of Euler class $\e$ and $\de_3$-invariant is of the form
		\begin{equation*}
			\left(\pm(k+1),-\frac{(k+1)^2}{4m+3} -k-m-\frac{1}{2}\right)  \,\,\textrm{ for }\,\,k\leq-m-2.
		\end{equation*}
	\end{enumerate}
\end{theorem}

In particular, the above result implies the following.

\begin{corollary}\label{cor:integerlenscor}\hfill
\begin{enumerate}
		\item For any given $e\in\{2m+1,2m+2,\ldots,3m+2\}$ there exists exactly one overtwisted contact structure $\xi$ on $L(4m+3,4)$ with Euler class $\e(\xi)=e$ that has $\cs_{\pm1}=1$. 
        \item For a given $e\in\Z_{4m+3}\setminus\{2(m+2)\}$ there exists infinitely many pairwise non-contactomorphic contact structures on $L(4m+3,4)$ with Euler classes $e$ that have $\cs_{\Z}=1$. And there exist also infinitely many pairwise non-contactomorphic contact structures on $L(4m+3,4)$ with Euler classes $e$ that have $\cs_{\Z}>1$. 
    \end{enumerate}
\end{corollary}

\subsection{Rational surgery numbers of lens spaces}
We also have similar results for rational contact surgery numbers. All tight contact structures on lens spaces have $\cs=1$. We also classify the overtwisted contact structures on $L(4m+3,4)$ that have $\cs=1$ in Theorem~\ref{thm:lens_rational}. But here the result is much more evolved, so in this introduction, we only state one of its corollaries.

\begin{corollary}\label{cor:lens_rational}
For every given $m\geq1$ there exists infinitely many pairwise non-contactomorphic contact structures on $L(4m+3,4)$ that have $\cs>1$. 
\end{corollary}

\subsection{Outline of the arguments}
We briefly outline our main arguments. Let $M$ be either the Brieskorn sphere $\pm\Sigma(2,3,11)$ or a lens space which is of the form $L(4m+3,4)$. Then there are results of Culler--Gordon--Luecke--Shalen~\cite{cyclic_dehn_surgery}, Baldwin--Sivek~\cite{baldwin2022characterizing}, and Rasmussen~\cite{rasmussen2007lens} that together completely classify all surgery diagrams along a single knot yielding $M$. It turns out that these are all surgeries along unknots, certain torus knots, and twist knots for which the classification of Legendrian realizations is known~\cite{Eliashberg_Fraser,Etnyre_Honda_torus_knots,Etnyre_Ng_vertesi_twist,Leg_knot_atlas}. Moreover, the classification of tight contact structures on $M$ is known by~\cite{Honda_lens,Ghiggini_Schoeneberger} and the overtwisted contact structures are determined by their homotopical invariants~\cite{Eliashberg_OT}. Since $M$ is a rational homology sphere that has no $2$-torsion in its first homology, these homotopical invariants are given by the Euler class, and the $\de_3$-invariant~\cite{Gompf_Stein,Ding_Geiges_Stipsicz}. Thus we can enumerate all possible contact surgery diagrams along a single Legendrian knot in $(S^3,\xist)$ yielding a contact structure on $M$. Then we check which of these surgeries yield tight and which yield overtwisted manifolds. We use the formulas from~\cite{Durst_Kegel_rot_surgery} to compute their homotopical invariants. This yields the classification of contact structures on $M$ that have contact surgery number $1$ from which we can deduce the claimed results. Opposed to the results obtained in~\cite{EKS_contact_surgery_numbers}, here we also need to work with Legendrian non-simple knots (the Legendrian twist knots) and since we do not just work on homology spheres, we have more complicated surgery coefficients and need to take the Euler classes into account. For that we discuss in Theorem~\ref{thm:Euler_algorithm} how to compute the Euler classes of general rational contact surgery diagrams which might be of independent interest. In Section~\ref{sec:prelim} we recall the needed background on contact surgery and in Section~\ref{sec:proof} we present the details of the above arguments.

\subsection*{Conventions}
All contact structures are assumed to be positive and coorientable. We present Legendrian knots always in their front projection. Since a contact surgery diagram determines a contact manifold only up to contactomorphism, we consider contact manifolds only up to contactomorphism (and not up to isotopy). Moreover, the contactomorphism type of a contact surgery is independent of an orientation of the Legendrian surgery link. Thus we mainly consider unoriented Legendrian links in $(S^3,\xist)$ up to isotopy of unoriented links. Then the rotation number of an unoriented Legendrian knot is only defined up to sign. For some calculations, it will be helpful to choose orientations on Legendrian links in which case the rotation numbers are always understood with signs. We normalize the $\de_3$ invariant such that $\de_3(S^3,\xist)=0$ as done for example in~\cite{casals2021stein,EKS_contact_surgery_numbers,surgery_graph}. With that normalization, the $\de_3$-invariant is additive under connected sum and takes integer values on homology spheres.

\subsection*{Acknowledgments} The authors are supported by \textit{SFB/TRR 191 Symplectic Structures in Geometry, Algebra and Dynamics}, funded by the Deutsche Forschungsgemeinschaff (Project-ID 281071066-TRR 191).

We would like to thank John Etnyre for many helpful suggestions and discussions.

This work started in February 2023 at the conference \textit{Winterbraids XII} in Tours and was continued at the workshop and summer school on \textit{Low-Dimensional Topology} at IISER Pune. We thank the organizers of both events for their support. The authors wish to thank Peter Albers, for his hospitality when R.C. visited the University of Heidelberg. Part of the writing was done when R.C.\ was visiting the Max-Plank Institute of Mathematics in Bonn. She expresses her gratitude for their hospitality. 

\section{Preliminaries}\label{sec:prelim}
In this section, we briefly recall the necessary background on contact surgery and how to compute the homotopical invariants of a contact structure from one of its contact surgery diagrams. For background on contact geometry, we refer to~\cite{Geiges_book}. For more on contact surgery and the homotopical invariants of contact structures, we refer to~\cite{Gompf_Stein,Ding_Geiges_Surgery,Ding_Geiges_Stipsicz,Ozbagci_Stipsicz_book,Durst_Kegel_rot_surgery,phdthesis,Legendrian_knot_complement,casals2021stein,EKS_contact_surgery_numbers}.

\subsection{Contact surgery}
Let $K$ be a Legendrian knot in $(S^3,\xist)$. We perform Dehn surgery on $K$ with contact surgery coefficient $r\in\Q\setminus\{0\}$ (i.e.\ measured with respect to the contact longitude of $K$, obtained by pushing $K$ into the Reeb-direction). Then there exist (up to contactomorphism) finitely many tight contact structures on the newly glued-in solid torus that fit together with the old contact structure to give a global contact structure on the surgered manifold~\cite{Honda_lens}. We write $K(r)$ for one of these contact manifolds obtained by \textit{contact $r$-surgery} along $K$. 
If in addition, the contact surgery coefficient is of the form $1/n$, for $n \in \Z$, then the contact structure on the surgered manifold is unique up to contactomorphism.

It will be convenient to introduce the following notation. We write $K{\def\svgwidth{1,6ex}\,\,
\begingroup%
  \makeatletter%
  \providecommand\color[2][]{%
    \errmessage{(Inkscape) Color is used for the text in Inkscape, but the package 'color.sty' is not loaded}%
    \renewcommand\color[2][]{}%
  }%
  \providecommand\transparent[1]{%
    \errmessage{(Inkscape) Transparency is used (non-zero) for the text in Inkscape, but the package 'transparent.sty' is not loaded}%
    \renewcommand\transparent[1]{}%
  }%
  \providecommand\rotatebox[2]{#2}%
  \ifx\svgwidth\undefined%
    \setlength{\unitlength}{49.54732089bp}%
    \ifx\svgscale\undefined%
      \relax%
    \else%
      \setlength{\unitlength}{\unitlength * \real{\svgscale}}%
    \fi%
  \else%
    \setlength{\unitlength}{\svgwidth}%
  \fi%
  \global\let\svgwidth\undefined%
  \global\let\svgscale\undefined%
  \makeatother%
  \begin{picture}(1,1.00059637)%
    \put(0,0){\includegraphics[width=\unitlength,page=1]{PushOff.pdf}}%
  \end{picture}%
\endgroup%
\,\,} K$ for the Legendrian link consisting of $K$ together with a Legendrian push-off of $K$. We denote a copy of $K$ with $n$ extra stabilizations by $K_n$. Here the $n$-extra stabilizations are not specified but fixed. If the knot $K_n$ is again stabilized $m$ times this is denoted by $K_{n,m}$. 

\begin{lemma}[Ding--Geiges~\cite{Ding_Geiges_torus_bundles,Ding_Geiges_Surgery}]\label{lem:Kirby}
Let $K$ be a Legendrian knot in $(S^3,\xist)$.
\begin{enumerate}
    \item \textbf{Cancellation lemma:} For every $n\in\Z\setminus\{0\}$, we have
    \begin{equation*}
	K\left(\frac{1}{n}\right){\def\svgwidth{1,6ex}\,\,
\begingroup%
  \makeatletter%
  \providecommand\color[2][]{%
    \errmessage{(Inkscape) Color is used for the text in Inkscape, but the package 'color.sty' is not loaded}%
    \renewcommand\color[2][]{}%
  }%
  \providecommand\transparent[1]{%
    \errmessage{(Inkscape) Transparency is used (non-zero) for the text in Inkscape, but the package 'transparent.sty' is not loaded}%
    \renewcommand\transparent[1]{}%
  }%
  \providecommand\rotatebox[2]{#2}%
  \ifx\svgwidth\undefined%
    \setlength{\unitlength}{49.54732089bp}%
    \ifx\svgscale\undefined%
      \relax%
    \else%
      \setlength{\unitlength}{\unitlength * \real{\svgscale}}%
    \fi%
  \else%
    \setlength{\unitlength}{\svgwidth}%
  \fi%
  \global\let\svgwidth\undefined%
  \global\let\svgscale\undefined%
  \makeatother%
  \begin{picture}(1,1.00059637)%
    \put(0,0){\includegraphics[width=\unitlength,page=1]{PushOff.pdf}}%
  \end{picture}%
\endgroup%
\,\,} \,K\left(-\frac{1}{n}\right)\cong(S^3,\xist).
	\end{equation*}
    \item \textbf{Replacement lemma:} For every $n\in\Z\setminus\{0\}$, we have
    \begin{equation*}
	K\left(\pm\frac{1}{n}\right)\cong K(\pm1){\def\svgwidth{1,6ex}\,\,
\begingroup%
  \makeatletter%
  \providecommand\color[2][]{%
    \errmessage{(Inkscape) Color is used for the text in Inkscape, but the package 'color.sty' is not loaded}%
    \renewcommand\color[2][]{}%
  }%
  \providecommand\transparent[1]{%
    \errmessage{(Inkscape) Transparency is used (non-zero) for the text in Inkscape, but the package 'transparent.sty' is not loaded}%
    \renewcommand\transparent[1]{}%
  }%
  \providecommand\rotatebox[2]{#2}%
  \ifx\svgwidth\undefined%
    \setlength{\unitlength}{49.54732089bp}%
    \ifx\svgscale\undefined%
      \relax%
    \else%
      \setlength{\unitlength}{\unitlength * \real{\svgscale}}%
    \fi%
  \else%
    \setlength{\unitlength}{\svgwidth}%
  \fi%
  \global\let\svgwidth\undefined%
  \global\let\svgscale\undefined%
  \makeatother%
  \begin{picture}(1,1.00059637)%
    \put(0,0){\includegraphics[width=\unitlength,page=1]{PushOff.pdf}}%
  \end{picture}%
\endgroup%
\,\,}\cdots{\def\svgwidth{1,6ex}\,\,
\begingroup%
  \makeatletter%
  \providecommand\color[2][]{%
    \errmessage{(Inkscape) Color is used for the text in Inkscape, but the package 'color.sty' is not loaded}%
    \renewcommand\color[2][]{}%
  }%
  \providecommand\transparent[1]{%
    \errmessage{(Inkscape) Transparency is used (non-zero) for the text in Inkscape, but the package 'transparent.sty' is not loaded}%
    \renewcommand\transparent[1]{}%
  }%
  \providecommand\rotatebox[2]{#2}%
  \ifx\svgwidth\undefined%
    \setlength{\unitlength}{49.54732089bp}%
    \ifx\svgscale\undefined%
      \relax%
    \else%
      \setlength{\unitlength}{\unitlength * \real{\svgscale}}%
    \fi%
  \else%
    \setlength{\unitlength}{\svgwidth}%
  \fi%
  \global\let\svgwidth\undefined%
  \global\let\svgscale\undefined%
  \makeatother%
  \begin{picture}(1,1.00059637)%
    \put(0,0){\includegraphics[width=\unitlength,page=1]{PushOff.pdf}}%
  \end{picture}%
\endgroup%
\,\,} K(\pm1).
	\end{equation*}
    \item \textbf{Transformation lemma:} For every $r\in\Q\setminus\{0\}$ and every integer $k$, we have
    \begin{equation*}
		K(r)\cong K\left(\frac{1}{k}\right){\def\svgwidth{1,6ex}\,\,
\begingroup%
  \makeatletter%
  \providecommand\color[2][]{%
    \errmessage{(Inkscape) Color is used for the text in Inkscape, but the package 'color.sty' is not loaded}%
    \renewcommand\color[2][]{}%
  }%
  \providecommand\transparent[1]{%
    \errmessage{(Inkscape) Transparency is used (non-zero) for the text in Inkscape, but the package 'transparent.sty' is not loaded}%
    \renewcommand\transparent[1]{}%
  }%
  \providecommand\rotatebox[2]{#2}%
  \ifx\svgwidth\undefined%
    \setlength{\unitlength}{49.54732089bp}%
    \ifx\svgscale\undefined%
      \relax%
    \else%
      \setlength{\unitlength}{\unitlength * \real{\svgscale}}%
    \fi%
  \else%
    \setlength{\unitlength}{\svgwidth}%
  \fi%
  \global\let\svgwidth\undefined%
  \global\let\svgscale\undefined%
  \makeatother%
  \begin{picture}(1,1.00059637)%
    \put(0,0){\includegraphics[width=\unitlength,page=1]{PushOff.pdf}}%
  \end{picture}%
\endgroup%
\,\,} \,K\left(\frac{1}{\frac{1}{r}-k}\right).
	\end{equation*}
 If the contact surgery coefficient $r$ is negative, one can write $r$ uniquely as 
		\begin{equation*}\label{expand}
		r=[r_1+1,r_2,\ldots,r_n]:=r_1+1-\cfrac{1}{r_2-\cfrac{1}{\dotsb -\cfrac{1}{r_n} }}
		\end{equation*}
		with integers $r_1,\ldots ,r_n\leq -2$ and we have
		\begin{equation*}
		K(r)\cong K_{|2+r_1|}(-1){\def\svgwidth{1,6ex}\,\,
\begingroup%
  \makeatletter%
  \providecommand\color[2][]{%
    \errmessage{(Inkscape) Color is used for the text in Inkscape, but the package 'color.sty' is not loaded}%
    \renewcommand\color[2][]{}%
  }%
  \providecommand\transparent[1]{%
    \errmessage{(Inkscape) Transparency is used (non-zero) for the text in Inkscape, but the package 'transparent.sty' is not loaded}%
    \renewcommand\transparent[1]{}%
  }%
  \providecommand\rotatebox[2]{#2}%
  \ifx\svgwidth\undefined%
    \setlength{\unitlength}{49.54732089bp}%
    \ifx\svgscale\undefined%
      \relax%
    \else%
      \setlength{\unitlength}{\unitlength * \real{\svgscale}}%
    \fi%
  \else%
    \setlength{\unitlength}{\svgwidth}%
  \fi%
  \global\let\svgwidth\undefined%
  \global\let\svgscale\undefined%
  \makeatother%
  \begin{picture}(1,1.00059637)%
    \put(0,0){\includegraphics[width=\unitlength,page=1]{PushOff.pdf}}%
  \end{picture}%
\endgroup%
\,\,} K_{|2+r_1|,|2+r_2|}(-1){\def\svgwidth{1,6ex}\,\,
\begingroup%
  \makeatletter%
  \providecommand\color[2][]{%
    \errmessage{(Inkscape) Color is used for the text in Inkscape, but the package 'color.sty' is not loaded}%
    \renewcommand\color[2][]{}%
  }%
  \providecommand\transparent[1]{%
    \errmessage{(Inkscape) Transparency is used (non-zero) for the text in Inkscape, but the package 'transparent.sty' is not loaded}%
    \renewcommand\transparent[1]{}%
  }%
  \providecommand\rotatebox[2]{#2}%
  \ifx\svgwidth\undefined%
    \setlength{\unitlength}{49.54732089bp}%
    \ifx\svgscale\undefined%
      \relax%
    \else%
      \setlength{\unitlength}{\unitlength * \real{\svgscale}}%
    \fi%
  \else%
    \setlength{\unitlength}{\svgwidth}%
  \fi%
  \global\let\svgwidth\undefined%
  \global\let\svgscale\undefined%
  \makeatother%
  \begin{picture}(1,1.00059637)%
    \put(0,0){\includegraphics[width=\unitlength,page=1]{PushOff.pdf}}%
  \end{picture}%
\endgroup%
\,\,}\cdots{\def\svgwidth{1,6ex}\,\,
\begingroup%
  \makeatletter%
  \providecommand\color[2][]{%
    \errmessage{(Inkscape) Color is used for the text in Inkscape, but the package 'color.sty' is not loaded}%
    \renewcommand\color[2][]{}%
  }%
  \providecommand\transparent[1]{%
    \errmessage{(Inkscape) Transparency is used (non-zero) for the text in Inkscape, but the package 'transparent.sty' is not loaded}%
    \renewcommand\transparent[1]{}%
  }%
  \providecommand\rotatebox[2]{#2}%
  \ifx\svgwidth\undefined%
    \setlength{\unitlength}{49.54732089bp}%
    \ifx\svgscale\undefined%
      \relax%
    \else%
      \setlength{\unitlength}{\unitlength * \real{\svgscale}}%
    \fi%
  \else%
    \setlength{\unitlength}{\svgwidth}%
  \fi%
  \global\let\svgwidth\undefined%
  \global\let\svgscale\undefined%
  \makeatother%
  \begin{picture}(1,1.00059637)%
    \put(0,0){\includegraphics[width=\unitlength,page=1]{PushOff.pdf}}%
  \end{picture}%
\endgroup%
\,\,} K_{|2+r_1|,\ldots,|2+r_n|}(-1).
		\end{equation*}
\end{enumerate}
In addition, all these results hold true in a tubular neighborhood of $K$. In particular, they can be applied to knots in larger contact surgery diagrams.
\end{lemma}

\subsection{The homotopical invariants}
We will also need to compute the algebraic invariants of the underlying tangential $2$-plane field of a contact structure. It is known that a tangential $2$-plane field $\xi$ on a rational homology sphere $M$ with no $2$-torsion in $H_1(M)$, is completely determined by its Euler class and its $\de_3$-invariant~\cite{Gompf_Stein,Ding_Geiges_Stipsicz}. Roughly speaking, the Euler class determines $\xi$ on the $2$-skeleton of $M$, while the $\de_3$-invariant gives $\xi$ on the $3$-cell of $M$. For contact $(\pm1/n)$-surgeries these invariants can be computed with the following result from~\cite{Durst_Kegel_rot_surgery}. 

\begin{lemma}\label{lem:d3}
	Let $L = L_1 \sqcup \ldots \sqcup L_k$ be a Legendrian link in
	$(S^3, \xist)$ and denote by $(M, \xi)$ the contact manifold obtained
	from $S^3$ by contact $(\pm {1}/{n_i})$-surgeries along $L$, with $n_i \in \N$. After choosing an orientation on $L$ we write $t_i$, $r_i$ for the Thurston--Bennequin invariant and the rotation number of $L_i$, and $l_{ij}$ for the linking number of $L_i$ and $L_j$. We denote the topological surgery coefficient of $L_i$ by $p_i/q_i=\pm1/n_i+ t_i$ and define the generalized linking matrix as
	\begin{align*}
	Q:=\begin{pmatrix}
	p_1&q_2 l_{12} &\cdots&q_n l_{1n}\\
	q_1 l_{21} & p_2&&\\
	\vdots&&\ddots\\
	q_1 l_{n1}&&& p_n
	\end{pmatrix}.
	\end{align*}	
	
	\begin{itemize}
		\item [(1)] Then $Q$ is a presentation matrix for the first homology group
		$H_1(M)$ of $M$, i.e. $H_1(M)$ is generated by the meridians $\mu_i$
		and the relations are 
		$Q{\mathbf\mu}=0$ where $\mathbf\mu$ is the vector with entries $\mu_i$. 
		\item [(2)] If all $n_i=1$, the Euler class $\textrm{e}(\xi)$ is Poincar\'{e} dual to 
		\begin{equation*}
		\operatorname{PD}\big(\textrm{e}(\xi)\big)=\sum_{i=1}^k r_i\mu_i\in H_1(M). 
		\end{equation*}
		
		\item [(3)] The Euler class $\textrm{e}(\xi)$ is torsion if and only if there exists
		a rational solution $\mathbf b\in\Q^k$ of $Q\mathbf b=\mathbf{r}$, where $\mathbf{r}$ is the vector of the rotation numbers $r_i$.
		In this case, the $\de_3$-invariant is a well-defined rational number and computes as
		\begin{equation*}
		\de_3(M,\xi) = \frac{1}{4} \left(\sum_{i=1}^k n_i b_i r_i  +  (3-n_i)\operatorname{sign}_i\right)   - \frac{3}{4} \sigma (Q) 
		\end{equation*}
		where $\operatorname{sign}_i$ denotes the sign of the contact surgery coefficient of $L_i$ and $\sigma(Q)$ denotes the signature of $Q$. (In the proof of Theorem~5.1. in~\cite{Durst_Kegel_rot_surgery} it is shown that the eigenvalues of $Q$ are all real and thus the signature of $Q$ is well-defined although $Q$ is in general a non-symmetric matrix.)
	\end{itemize}
\end{lemma}

Next, we present a formula and an algorithm how to compute the Euler class of a contact structure from a general rational contact surgery diagram. Here the main idea is to use the transformation lemma to transform a general contact surgery diagram into one with only $(\pm1)$-contact surgery coefficients. We will perform that process by doing explicit Kirby moves  which will allow us to keep track of the meridians. Then we can use Lemma~\ref{lem:d3} to compute the Euler class.

\begin{theorem}\label{thm:Euler_algorithm}
   Let $(M,\xi)$ be a contact manifold that is obtained by rational contact surgery on an oriented Legendrian link $L$ in $(S^3,\xist)$. Let $K$ be a component of $L$ with contact surgery coefficient $r\in\Q\setminus\{0\}$. 
   Then the Euler class of $\xi$ can be represented in the basis of $H_1(M)$ given by the meridians of $L$ by
   \begin{equation*}
       \e(\xi)=e_K\mu + e_L,
   \end{equation*}
   where $e_K$ is an integer, $\mu$ is the meridian of $K$, and $e_L$ is a linear combination of the meridians of $L\setminus K$.
   
   On the other hand, the transformation lemma tells us that there exist integers $l,m,s_{l+1},\ldots,s_{m}$ such that
   \begin{equation}
\label{eq:pushoff}
K(r)\cong\underbrace{K(+1){\def\svgwidth{1,6ex}\,\,
\begingroup%
  \makeatletter%
  \providecommand\color[2][]{%
    \errmessage{(Inkscape) Color is used for the text in Inkscape, but the package 'color.sty' is not loaded}%
    \renewcommand\color[2][]{}%
  }%
  \providecommand\transparent[1]{%
    \errmessage{(Inkscape) Transparency is used (non-zero) for the text in Inkscape, but the package 'transparent.sty' is not loaded}%
    \renewcommand\transparent[1]{}%
  }%
  \providecommand\rotatebox[2]{#2}%
  \ifx\svgwidth\undefined%
    \setlength{\unitlength}{49.54732089bp}%
    \ifx\svgscale\undefined%
      \relax%
    \else%
      \setlength{\unitlength}{\unitlength * \real{\svgscale}}%
    \fi%
  \else%
    \setlength{\unitlength}{\svgwidth}%
  \fi%
  \global\let\svgwidth\undefined%
  \global\let\svgscale\undefined%
  \makeatother%
  \begin{picture}(1,1.00059637)%
    \put(0,0){\includegraphics[width=\unitlength,page=1]{PushOff.pdf}}%
  \end{picture}%
\endgroup%
\,\,}\cdots K(+1)}_l{\def\svgwidth{1,6ex}\,\,
\begingroup%
  \makeatletter%
  \providecommand\color[2][]{%
    \errmessage{(Inkscape) Color is used for the text in Inkscape, but the package 'color.sty' is not loaded}%
    \renewcommand\color[2][]{}%
  }%
  \providecommand\transparent[1]{%
    \errmessage{(Inkscape) Transparency is used (non-zero) for the text in Inkscape, but the package 'transparent.sty' is not loaded}%
    \renewcommand\transparent[1]{}%
  }%
  \providecommand\rotatebox[2]{#2}%
  \ifx\svgwidth\undefined%
    \setlength{\unitlength}{49.54732089bp}%
    \ifx\svgscale\undefined%
      \relax%
    \else%
      \setlength{\unitlength}{\unitlength * \real{\svgscale}}%
    \fi%
  \else%
    \setlength{\unitlength}{\svgwidth}%
  \fi%
  \global\let\svgwidth\undefined%
  \global\let\svgscale\undefined%
  \makeatother%
  \begin{picture}(1,1.00059637)%
    \put(0,0){\includegraphics[width=\unitlength,page=1]{PushOff.pdf}}%
  \end{picture}%
\endgroup%
\,\,} K_{s_{l+1}}(-1){\def\svgwidth{1,6ex}\,\,
\begingroup%
  \makeatletter%
  \providecommand\color[2][]{%
    \errmessage{(Inkscape) Color is used for the text in Inkscape, but the package 'color.sty' is not loaded}%
    \renewcommand\color[2][]{}%
  }%
  \providecommand\transparent[1]{%
    \errmessage{(Inkscape) Transparency is used (non-zero) for the text in Inkscape, but the package 'transparent.sty' is not loaded}%
    \renewcommand\transparent[1]{}%
  }%
  \providecommand\rotatebox[2]{#2}%
  \ifx\svgwidth\undefined%
    \setlength{\unitlength}{49.54732089bp}%
    \ifx\svgscale\undefined%
      \relax%
    \else%
      \setlength{\unitlength}{\unitlength * \real{\svgscale}}%
    \fi%
  \else%
    \setlength{\unitlength}{\svgwidth}%
  \fi%
  \global\let\svgwidth\undefined%
  \global\let\svgscale\undefined%
  \makeatother%
  \begin{picture}(1,1.00059637)%
    \put(0,0){\includegraphics[width=\unitlength,page=1]{PushOff.pdf}}%
  \end{picture}%
\endgroup%
\,\,}\cdots {\def\svgwidth{1,6ex}\,\,
\begingroup%
  \makeatletter%
  \providecommand\color[2][]{%
    \errmessage{(Inkscape) Color is used for the text in Inkscape, but the package 'color.sty' is not loaded}%
    \renewcommand\color[2][]{}%
  }%
  \providecommand\transparent[1]{%
    \errmessage{(Inkscape) Transparency is used (non-zero) for the text in Inkscape, but the package 'transparent.sty' is not loaded}%
    \renewcommand\transparent[1]{}%
  }%
  \providecommand\rotatebox[2]{#2}%
  \ifx\svgwidth\undefined%
    \setlength{\unitlength}{49.54732089bp}%
    \ifx\svgscale\undefined%
      \relax%
    \else%
      \setlength{\unitlength}{\unitlength * \real{\svgscale}}%
    \fi%
  \else%
    \setlength{\unitlength}{\svgwidth}%
  \fi%
  \global\let\svgwidth\undefined%
  \global\let\svgscale\undefined%
  \makeatother%
  \begin{picture}(1,1.00059637)%
    \put(0,0){\includegraphics[width=\unitlength,page=1]{PushOff.pdf}}%
  \end{picture}%
\endgroup%
\,\,} K_{s_{l+1},\cdots ,s_m}(-1).
\end{equation}
We denote the components of the above surgery description as $K^i$ for $i=1,2,\cdots, m$ and write $t$ and $r$ for the Thurston--Bennequin invariant and rotation number of $K$, and $r_i$ for the rotation number and $\mu_i$ for the meridian of $K^i$.

We have the following cases:
\begin{itemize}
\item[(1)] For $m=2$, the surgery diagram reduces to $K(r)=K(\pm 1){\def\svgwidth{1,6ex}\,\,
\begingroup%
  \makeatletter%
  \providecommand\color[2][]{%
    \errmessage{(Inkscape) Color is used for the text in Inkscape, but the package 'color.sty' is not loaded}%
    \renewcommand\color[2][]{}%
  }%
  \providecommand\transparent[1]{%
    \errmessage{(Inkscape) Transparency is used (non-zero) for the text in Inkscape, but the package 'transparent.sty' is not loaded}%
    \renewcommand\transparent[1]{}%
  }%
  \providecommand\rotatebox[2]{#2}%
  \ifx\svgwidth\undefined%
    \setlength{\unitlength}{49.54732089bp}%
    \ifx\svgscale\undefined%
      \relax%
    \else%
      \setlength{\unitlength}{\unitlength * \real{\svgscale}}%
    \fi%
  \else%
    \setlength{\unitlength}{\svgwidth}%
  \fi%
  \global\let\svgwidth\undefined%
  \global\let\svgscale\undefined%
  \makeatother%
  \begin{picture}(1,1.00059637)%
    \put(0,0){\includegraphics[width=\unitlength,page=1]{PushOff.pdf}}%
  \end{picture}%
\endgroup%
\,\,} K_{s_2}(-1).$ In the basis of $H_1(M)$ given by this surgery description the Euler class is given by 
{\small\begin{align*}
    \e(\xi)=\big(\mp t(r_1-r_2)+r_2\big)\mu +\e_L
\end{align*}
}
\item[(2)] For $m=l$, i.e for contact $(\frac{1}{m})$-surgery and for $l=0$ and $s_{i}=0$ for $i=3,\cdots m$, i.e for contact $(-\frac{1}{m})$ surgery the Euler classes are given by {\small\begin{align*}
    \e(\xi)&=r\mu +\e_L
\end{align*}
}

\item[(3)] For $m>2$ and $l=0$, in the basis of $H_1(M)$ given by the surgery description  from Equation~\ref{eq:pushoff} the Euler class is given by
{\tiny
\begin{align*}
    \e(\xi)=\left(rt+(1-t)\left[r_m\prod_{k=2}^{m-1} (1-t+s_k)+r_2(t-s_2)+\sum_{i=3}^{m-1} r_i(t-s_i)\prod_{k=2}^{i-1}(1-t+s_k)\right]\right)\mu+e_L
\end{align*}
} This case corresponds to negative contact surgery.
\item[(4)] For $m>2$, $l>0$ and $l\neq m$, in the basis of $H_1(M)$ given by the surgery description  from Equation~\ref{eq:pushoff} the Euler class is given by
{\tiny
\begin{align*}
    \e(\xi)=\left(r+(1+t)^l\left[-r+r_m\prod_{k=l+1}^{m-1} (1-t+s_k)+r_{l+1}(t-s_2)+\sum_{i=l+2}^{m-1} r_i(t-s_i)\prod_{k=l+1}^{i-1}(1-t+s_k)\right]\right)\mu+e_L
 \end{align*}
}
\end{itemize}
\end{theorem}

The main ingredients in the proof of the above theorem are the following two lemmas. 

    \begin{lemma}
    \label{lem:meridian}
        Under the diffeomorphism from Equation~\ref{eq:pushoff}, for $m>2$ and $m\neq l$  the homology classes of the meridians $\mu_i$  in $H_1(M)$ are mapped as follows:
        \begin{align*}
         \mu_k&\longmapsto -t(1+t)^{k-1}\mu\,\textrm{ for }\,k=1,2,\cdots, l\\
        \mu_{l+1}&\longmapsto (t-s_{l+1})(1+t)^{l}\mu\\
        \mu_{l+2}&\longmapsto (t-s_{l+2})(1-t+s_{l+1})(1+t)^{l}\mu\\
        &\vdots\\
        \mu_{m-1}&\longmapsto (t-s_{m-1})(1-t+s_{m-2})\cdots (1-t+s_{l+1})(1+t)^l \mu\\
        \mu_m&\longmapsto (1-t+s_{m-1})(1-t+s_{m-2})\cdots (1-t+s_{l+1})(1+t)^l\mu
        \end{align*}   
      For $m=2$, the meridians are mapped as follows:
       \begin{align*}
           \mu_1&\longmapsto -t\mu\\
           \mu_2&\longmapsto(1+ t)\mu
       \end{align*}
     For $m=l$, we have the following:
       \begin{align*}
           \mu_k&\longmapsto -t(1+t)^{k-1}\mu\,\textrm{ For }\,k=1,2,\cdots, m\\
           \mu_m&\longmapsto(1+ t)^{m-1}\mu
       \end{align*}
     \end{lemma}
\begin{proof} We will consider several cases of increasing difficulty.

\noindent
\textbf{Case 1:} First we consider the case that $r=1/m$ for a positive integer $m>2$. This is equivalent of saying that $l=m$. We will show by induction on $m$ that
\begin{align*}
    \mu_k&\longmapsto-t(1+t)^{k-1}\mu\, \textrm{ for all } k=1,2,\ldots, m-1,\\
    \mu_m&\longmapsto(1+t)^{m-1}\mu.
\end{align*}

The induction step follows from the Kirby moves in Figure~\ref{fig:kirbymove1}. We start in $(1)$ with the Kirby diagram 
$$\underbrace{K(1){\def\svgwidth{1,6ex}\,\,
\begingroup%
  \makeatletter%
  \providecommand\color[2][]{%
    \errmessage{(Inkscape) Color is used for the text in Inkscape, but the package 'color.sty' is not loaded}%
    \renewcommand\color[2][]{}%
  }%
  \providecommand\transparent[1]{%
    \errmessage{(Inkscape) Transparency is used (non-zero) for the text in Inkscape, but the package 'transparent.sty' is not loaded}%
    \renewcommand\transparent[1]{}%
  }%
  \providecommand\rotatebox[2]{#2}%
  \ifx\svgwidth\undefined%
    \setlength{\unitlength}{49.54732089bp}%
    \ifx\svgscale\undefined%
      \relax%
    \else%
      \setlength{\unitlength}{\unitlength * \real{\svgscale}}%
    \fi%
  \else%
    \setlength{\unitlength}{\svgwidth}%
  \fi%
  \global\let\svgwidth\undefined%
  \global\let\svgscale\undefined%
  \makeatother%
  \begin{picture}(1,1.00059637)%
    \put(0,0){\includegraphics[width=\unitlength,page=1]{PushOff.pdf}}%
  \end{picture}%
\endgroup%
\,\,}\cdots {\def\svgwidth{1,6ex}\,\,
\begingroup%
  \makeatletter%
  \providecommand\color[2][]{%
    \errmessage{(Inkscape) Color is used for the text in Inkscape, but the package 'color.sty' is not loaded}%
    \renewcommand\color[2][]{}%
  }%
  \providecommand\transparent[1]{%
    \errmessage{(Inkscape) Transparency is used (non-zero) for the text in Inkscape, but the package 'transparent.sty' is not loaded}%
    \renewcommand\transparent[1]{}%
  }%
  \providecommand\rotatebox[2]{#2}%
  \ifx\svgwidth\undefined%
    \setlength{\unitlength}{49.54732089bp}%
    \ifx\svgscale\undefined%
      \relax%
    \else%
      \setlength{\unitlength}{\unitlength * \real{\svgscale}}%
    \fi%
  \else%
    \setlength{\unitlength}{\svgwidth}%
  \fi%
  \global\let\svgwidth\undefined%
  \global\let\svgscale\undefined%
  \makeatother%
  \begin{picture}(1,1.00059637)%
    \put(0,0){\includegraphics[width=\unitlength,page=1]{PushOff.pdf}}%
  \end{picture}%
\endgroup%
\,\,} K(1)}_{m-1}{\def\svgwidth{1,6ex}\,\,
\begingroup%
  \makeatletter%
  \providecommand\color[2][]{%
    \errmessage{(Inkscape) Color is used for the text in Inkscape, but the package 'color.sty' is not loaded}%
    \renewcommand\color[2][]{}%
  }%
  \providecommand\transparent[1]{%
    \errmessage{(Inkscape) Transparency is used (non-zero) for the text in Inkscape, but the package 'transparent.sty' is not loaded}%
    \renewcommand\transparent[1]{}%
  }%
  \providecommand\rotatebox[2]{#2}%
  \ifx\svgwidth\undefined%
    \setlength{\unitlength}{49.54732089bp}%
    \ifx\svgscale\undefined%
      \relax%
    \else%
      \setlength{\unitlength}{\unitlength * \real{\svgscale}}%
    \fi%
  \else%
    \setlength{\unitlength}{\svgwidth}%
  \fi%
  \global\let\svgwidth\undefined%
  \global\let\svgscale\undefined%
  \makeatother%
  \begin{picture}(1,1.00059637)%
    \put(0,0){\includegraphics[width=\unitlength,page=1]{PushOff.pdf}}%
  \end{picture}%
\endgroup%
\,\,} K(1)$$ and their meridians $\mu_i$ and end up with their images in 
$$\underbrace{K(1){\def\svgwidth{1,6ex}\,\,
\begingroup%
  \makeatletter%
  \providecommand\color[2][]{%
    \errmessage{(Inkscape) Color is used for the text in Inkscape, but the package 'color.sty' is not loaded}%
    \renewcommand\color[2][]{}%
  }%
  \providecommand\transparent[1]{%
    \errmessage{(Inkscape) Transparency is used (non-zero) for the text in Inkscape, but the package 'transparent.sty' is not loaded}%
    \renewcommand\transparent[1]{}%
  }%
  \providecommand\rotatebox[2]{#2}%
  \ifx\svgwidth\undefined%
    \setlength{\unitlength}{49.54732089bp}%
    \ifx\svgscale\undefined%
      \relax%
    \else%
      \setlength{\unitlength}{\unitlength * \real{\svgscale}}%
    \fi%
  \else%
    \setlength{\unitlength}{\svgwidth}%
  \fi%
  \global\let\svgwidth\undefined%
  \global\let\svgscale\undefined%
  \makeatother%
  \begin{picture}(1,1.00059637)%
    \put(0,0){\includegraphics[width=\unitlength,page=1]{PushOff.pdf}}%
  \end{picture}%
\endgroup%
\,\,}\cdots {\def\svgwidth{1,6ex}\,\,
\begingroup%
  \makeatletter%
  \providecommand\color[2][]{%
    \errmessage{(Inkscape) Color is used for the text in Inkscape, but the package 'color.sty' is not loaded}%
    \renewcommand\color[2][]{}%
  }%
  \providecommand\transparent[1]{%
    \errmessage{(Inkscape) Transparency is used (non-zero) for the text in Inkscape, but the package 'transparent.sty' is not loaded}%
    \renewcommand\transparent[1]{}%
  }%
  \providecommand\rotatebox[2]{#2}%
  \ifx\svgwidth\undefined%
    \setlength{\unitlength}{49.54732089bp}%
    \ifx\svgscale\undefined%
      \relax%
    \else%
      \setlength{\unitlength}{\unitlength * \real{\svgscale}}%
    \fi%
  \else%
    \setlength{\unitlength}{\svgwidth}%
  \fi%
  \global\let\svgwidth\undefined%
  \global\let\svgscale\undefined%
  \makeatother%
  \begin{picture}(1,1.00059637)%
    \put(0,0){\includegraphics[width=\unitlength,page=1]{PushOff.pdf}}%
  \end{picture}%
\endgroup%
\,\,} K(1)}_{m-2}{\def\svgwidth{1,6ex}\,\,
\begingroup%
  \makeatletter%
  \providecommand\color[2][]{%
    \errmessage{(Inkscape) Color is used for the text in Inkscape, but the package 'color.sty' is not loaded}%
    \renewcommand\color[2][]{}%
  }%
  \providecommand\transparent[1]{%
    \errmessage{(Inkscape) Transparency is used (non-zero) for the text in Inkscape, but the package 'transparent.sty' is not loaded}%
    \renewcommand\transparent[1]{}%
  }%
  \providecommand\rotatebox[2]{#2}%
  \ifx\svgwidth\undefined%
    \setlength{\unitlength}{49.54732089bp}%
    \ifx\svgscale\undefined%
      \relax%
    \else%
      \setlength{\unitlength}{\unitlength * \real{\svgscale}}%
    \fi%
  \else%
    \setlength{\unitlength}{\svgwidth}%
  \fi%
  \global\let\svgwidth\undefined%
  \global\let\svgscale\undefined%
  \makeatother%
  \begin{picture}(1,1.00059637)%
    \put(0,0){\includegraphics[width=\unitlength,page=1]{PushOff.pdf}}%
  \end{picture}%
\endgroup%
\,\,} K(1/2)$$ under the Kirby moves in diagram $(6)$ from which we read-off that
\begin{align*}
\mu_i&\longmapsto\mu'_i\,\textrm{ for } i=1,\ldots m_2,\\
\mu_{m-1}&\longmapsto-t\mu'_{m-1},\\
 \mu_{m}&\longmapsto(1+t)\mu'_{m-1},
\end{align*}
where we write $\mu'_i$ for the $i$-th knot in the surgery diagram $(6)$. The statement then follows from the induction hypothesis.

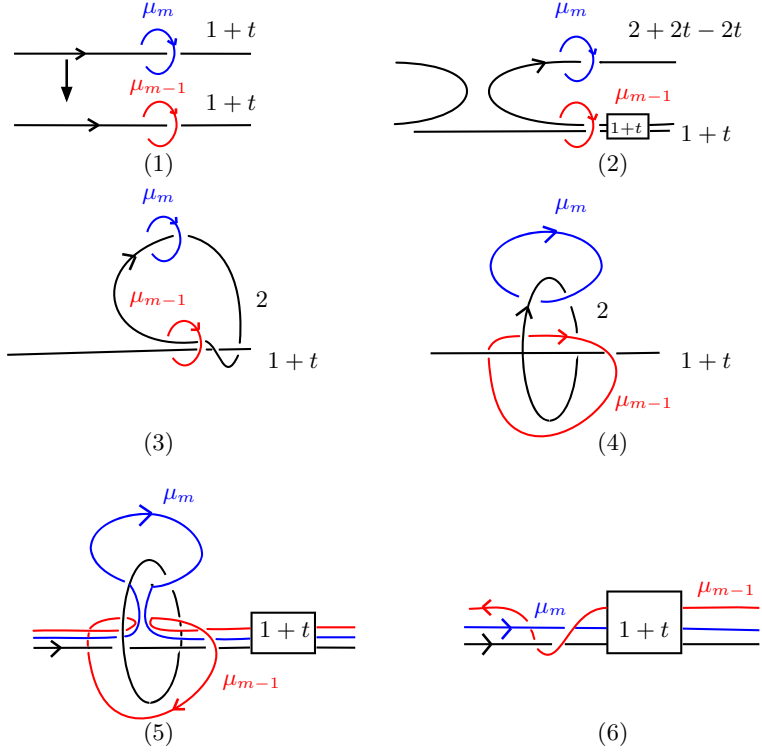
\begin{figure}[htbp] 
\centering
\def\svgwidth{0.99\columnwidth}
\labellist
\small\hair 2pt
  \pinlabel \textcolor{blue}{$\mu_m$} at 180 720
  \pinlabel \textcolor{red}{$\mu_{m-1}$} at 180 645
  \pinlabel $1+t$ at 250 700
  \pinlabel $1+t$ at 250 630
  \pinlabel $(1)$ at 180 570
  \pinlabel \textcolor{blue}{$\mu_m$} at 580 720
  \pinlabel \textcolor{red}{$\mu_{m-1}$} at 650 637
  \pinlabel $2+2t-2t$ at 690 700
  \pinlabel $1+t$ at 710 600
  \pinlabel $(2)$ at 620 570
  \pinlabel \textcolor{blue}{$\mu_m$} at 180 540
  \pinlabel \textcolor{red}{$\mu_{m-1}$} at 180 440
  \pinlabel $2$ at  280 440
  \pinlabel $1+t$ at 310 380
  \pinlabel $(3)$ at 180 300
  \pinlabel $(4)$ at 620 300
  \pinlabel \textcolor{blue}{$\mu_m$} at 580 530
  \pinlabel \textcolor{red}{$\mu_{m-1}$} at 650 340
  \pinlabel $1+t$ at 710 380
  \pinlabel $2$ at 610 430
  \pinlabel $(5)$ at 180 20
  \pinlabel $1+t$ at 302 120
  \pinlabel \textcolor{blue}{$\mu_m$} at 200 250
  \pinlabel \textcolor{red}{$\mu_{m-1}$} at 270 70
  \pinlabel $(6)$  at 620 20
  \pinlabel $1+t$ at 650 120
  \pinlabel \textcolor{red}{$\mu_{m-1}$} at 730 160
  \pinlabel \textcolor{blue}{$\mu_m$} at 560 140
  \tiny\pinlabel 1+$t$ at 633 605
  
  \endlabellist
\includegraphics[scale=0.4]{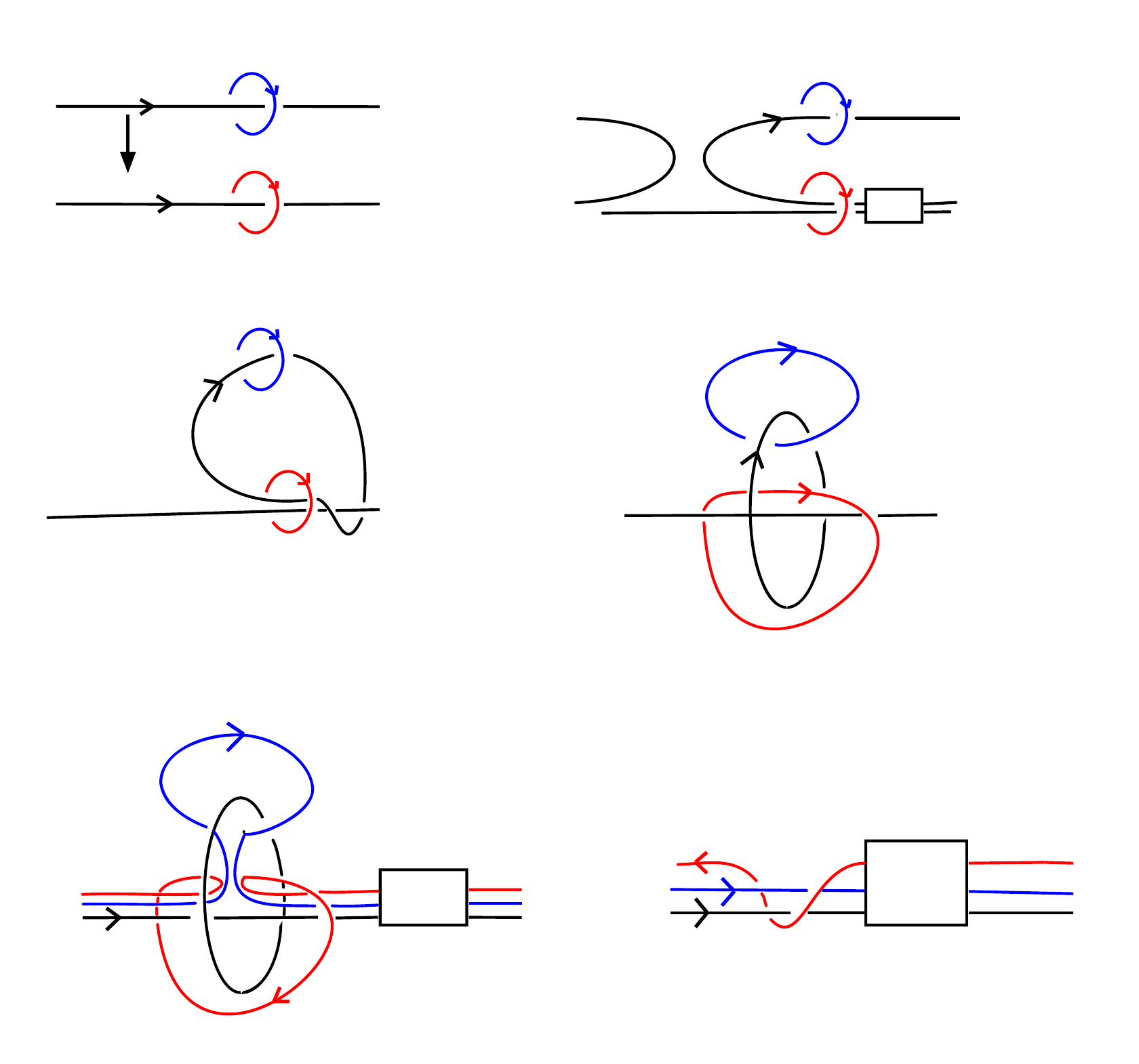}
\caption{Case 1: While performing the explicit Kirby moves from diagram $(1)$ to diagram $(6)$ we keep track of the meridians. Note that in this figure all surgery coefficients are measured with respect to the Seifert framing. }
\label{fig:kirbymove1}
\end{figure}

\noindent
\textbf{Case 2:} Next, we consider a surgery diagram where $K^1$ has contact surgery coefficient $ 1$ and all other $K^i$'s have contact surgery coefficient $-1$. 
We will show by induction that
\begin{align*}
\mu_1&\longmapsto- t\mu\\
\mu_2&\longmapsto(t-s_2)(1+t)\mu\\
\mu_3&\longmapsto(t-s_3)(1-t+s_2)(1+ t)\mu\\
&\vdots\\
\mu_{m-1}&\longmapsto(t-s_{m-1})(1-t+s_{m-2})(1-t+s_{m-3})\cdots(1+t)\mu\\
\mu_m&\longmapsto(1-t+s_{m-1})(1-t+s_{m-2})(1-t+s_{m-3})\cdots(1+t)\mu.
\end{align*}

The induction step (for $m\geq3$) follows also along the same lines as Case~$1$ from the Kirby moves depicted in Figure~\ref{fig:kirby2}. From Figure~\ref{fig:kirby2} we read-off that
\begin{align*}
\mu_i&\longmapsto\mu'_i\,\textrm{ for } i=1,\ldots m-2,\\
\mu_{m-1}&\longmapsto(t-s_{m-1})\mu'_{m-1},\\
 \mu_{m}&\longmapsto(1-t+s_{m-1})\mu'_{m-1}, 
\end{align*}
where we write $\mu'_i$ for the $i$-th knot in the surgery diagram $(6)$. The statement then follows from the induction hypothesis.

\begin{figure}[htbp] 
\centering
\def\svgwidth{0.99\columnwidth}
 \labellist
 	\small\hair 2pt
  \pinlabel \textcolor{blue}{$\mu_m$} at 150 710
  \pinlabel \textcolor{red}{$\mu_{m-1}$} at 150 640 
 \pinlabel $k-s_m$ at 275 700
 \pinlabel $k$ at 255 630
  \pinlabel $(1)$ at 150 560
  \pinlabel \textcolor{blue}{$\mu_m$} at 620 720
  \pinlabel \textcolor{red}{$\mu_{m-1}$} at 660 650
  \pinlabel $-2-s_m$ at 700 700
 \pinlabel $(2)$ at 600 560  
 \pinlabel \textcolor{blue}{$\mu_m$} at 135 510
 \pinlabel \textcolor{red}{$\mu_{m-1}$} at 170 440
 \pinlabel $(3)$ at 160 310
 \pinlabel $-2-s_m$ at 300 450
 \pinlabel $(4)$ at 600 300
 \pinlabel \textcolor{blue}{$\mu_m$} at 520 490
 \pinlabel\textcolor{red}{$\mu_{m-1}$} at 460 330
 \pinlabel $-2-s_m$ at 620 420
 \pinlabel $(5)$ at 160 10
 \pinlabel\textcolor{blue}{$\mu_m$} at 110 220
 \pinlabel \textcolor{red}{$\mu_{m-1}$} at 115 80
 \pinlabel $-2-s_m$ at 270 190
 \pinlabel $k$ at 325 150
 \pinlabel $(6)$ at 600 10
 \pinlabel \textcolor{blue}{$\mu_{m}$} at 530 160
  \pinlabel $k$ at 654 150
  \pinlabel \textcolor{red}{$\mu_{m-1}$} at  510 90
\tiny\pinlabel -$s_{m-1}$ at 49 647
  \pinlabel -$s_{m-1}$ at 410 647
 \pinlabel $-1+t$ at  657 610
  \endlabellist
\includegraphics[scale=0.4]{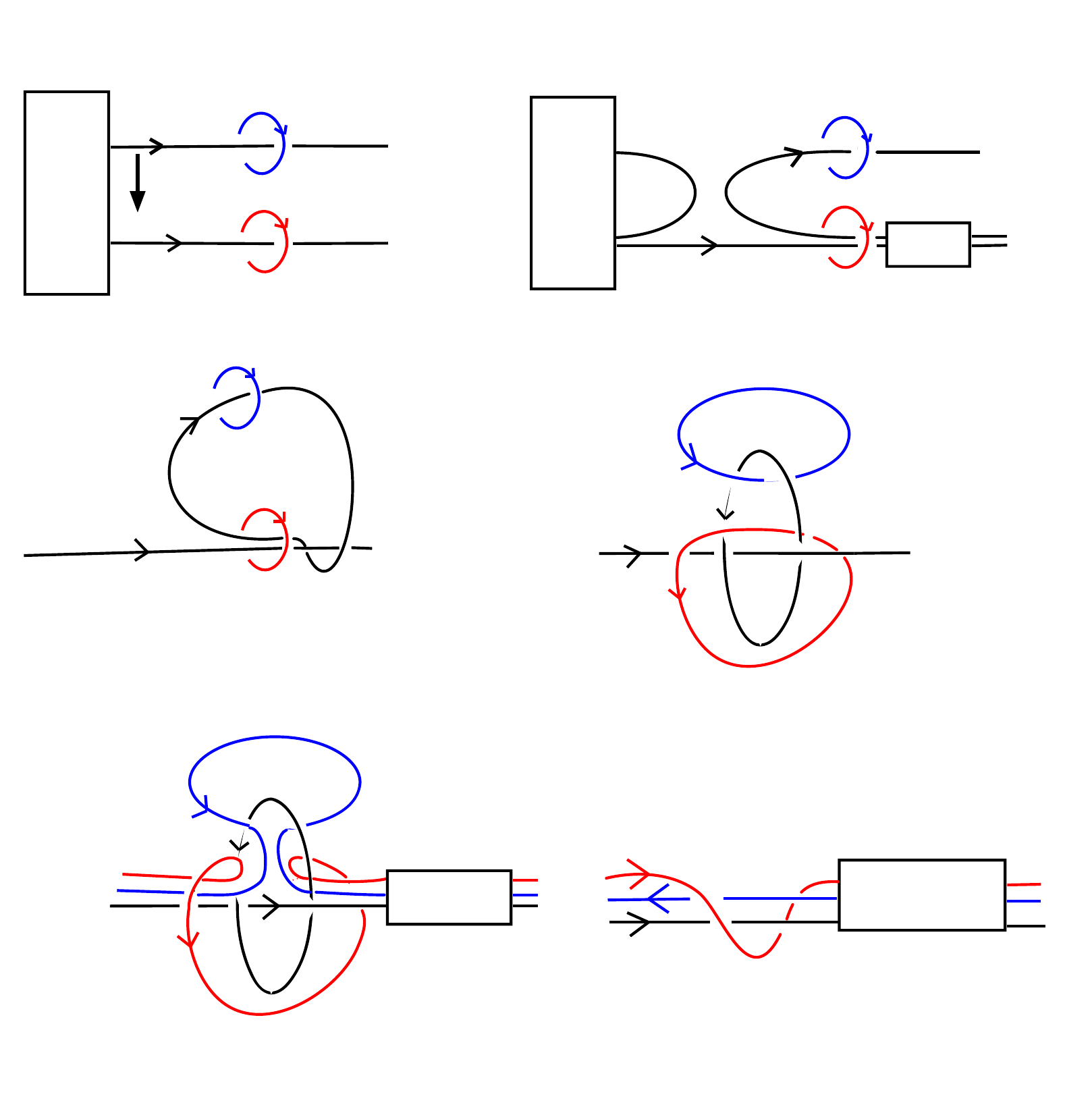}
\caption{In Case 2, performing the explicit Kirby moves from diagram $(1)$ to diagram $(6)$ we keep track of the meridians. All surgery coefficients are measured with respect to the Seifert framing. (1) A local picture with components $K_m$ and $K_{m-1}$. Here $k=-1+t-s_{m-1}$ (2) Slide $K_m$ over $K_{m-1}$. (3) and (4) Isotopy. (5) Slide $\mu_m$ and $\mu_{m-1}$ over $K_{m-1}$. (6) Slam-dunk $K_m$.}
\label{fig:kirby2}
\end{figure}
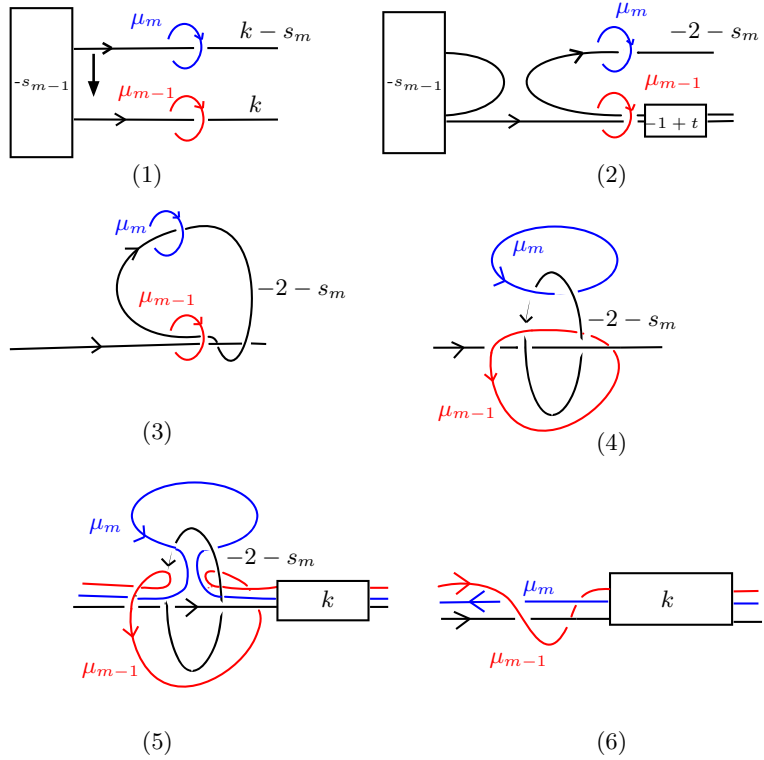

\noindent
\textbf{Case 3:} The general case is smoothly depicted in Figure~\ref{fig:general_surgery}. First, we just look at the last $(m-l+1)$ components starting from $K^l$. Notice that this coincides with Case~$2$. And thus we get
\begin{align*}
\mu_l&\longmapsto- t\mu_l',\\
\mu_{l+1}&\longmapsto(t-s_{l+1})(1+ t)\mu_l',\\
\mu_{l+2}&\longmapsto(t-s_{l+2})(1-t+s_{l+1})(1+ t)\mu_l',\\
&\vdots\\
\mu_{m-1}&\longmapsto(t-s_{m-1})(1-t+s_{m-2})(1-t+s_{m-3})\cdots(1-t+s_{l+1})(1+ t)\mu_l',\\
\mu_m&\longmapsto(1-t+s_{m-1})(1-t+s_{m-2})(1-t+s_{m-3})\cdots(1-t+s_{l+1})(1+ t)\mu_l'.
\end{align*}

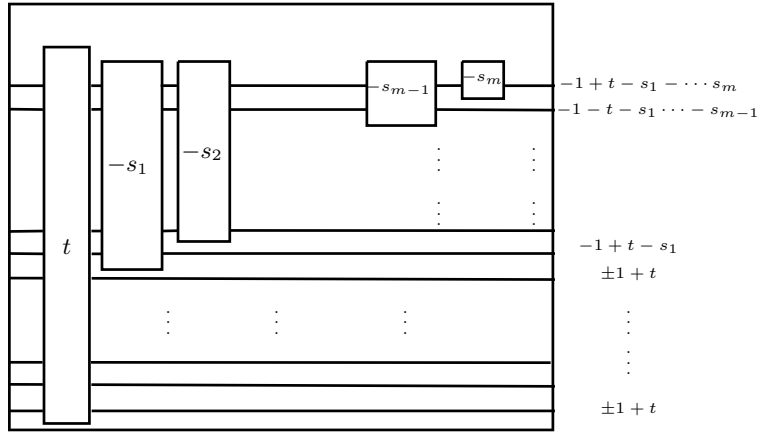
\begin{figure}
 \labellist
 	\small\hair 2pt
  \pinlabel $-s_2$ at 155 220
 	\pinlabel $-s_1$ at  100 210
  \pinlabel $t$ at 55 150
\tiny\pinlabel $-s_m$ at 362 275
 	 \pinlabel $-s_{m-1}$ at 300 267
 	\pinlabel $-1-t-s_1\cdots-s_{m-1}$ at 492 250
        \pinlabel $\vdots$ at 130 100
         \pinlabel $\vdots$ at 210 100
         \pinlabel $\pm 1+t$ at 470 30
           \pinlabel $\vdots$ at 470 70
         \pinlabel $\vdots$ at 305 100
          \pinlabel $\vdots$ at 470 100
             \pinlabel $\pm 1+t$  at 470 130
           \pinlabel $-1+t-s_1$ at 470 150
         \pinlabel $\vdots$ at 330 180
         \pinlabel $\vdots$ at 330 220
         \pinlabel $\vdots$ at 400 220
         \pinlabel $\vdots$ at 400 180
             \pinlabel $-1+t-s_1-\cdots s_m$ at 485 270
         \endlabellist
    \includegraphics[width=0.6\textwidth]{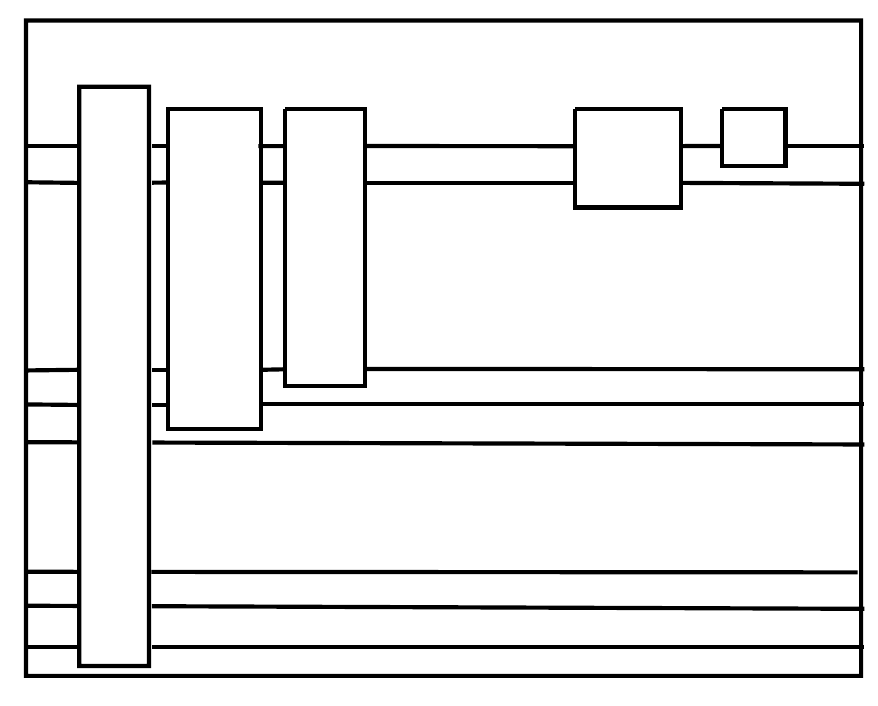}
    \caption{Case 3: A general surgery diagram again with surgery coefficients measured with respect to the Seifert longitude. The number inside the box represents the linking number between the components.}
     \label{fig:general_surgery}
\end{figure}

Here again $\mu_{l}'$ denotes the meridian of $K^l$ after the Kirby moves. 
Now we can forget about the components with index larger than $l$ and just look at the components $K^1,\ldots,K^l$. This part coincides with Case~$1$. Thus we have
\begin{align*}
   \mu_k&\longmapsto-t(1+t)^{k-1}\mu, \, \textrm{ for } k=1,2,\ldots, l-1,\\
    \mu_l'&\longmapsto(1+t)^{l-1}\mu.
\end{align*}
 Together this implies the general case of the lemma. 

 For $m=2$, the proof works the same. One just needs to keep track of the meridian as in Figure~\ref{fig:kirbymove1}. The only difference will be that $K_2$ will have topological surgery coefficient $-1+t-s_2$.  
\end{proof}

\begin{lemma}
    \label{lem:meridian_neg}
    For negative contact surgery, under the diffeomorphism from Equation~\ref{eq:pushoff} with $l=0$, the homology classes of the meridians $\mu_i$ for $m>2$ in $H_1(M)$ are mapped as follows:
    \begin{align*}
         \mu_1&\longmapsto t\mu\\
        \mu_{2}&\longmapsto (t-s_{2})(1-t)\mu\\
        \mu_{3}&\longmapsto (t-s_{2})(1-t+s_{2})(1-t)\mu\\
        &\vdots\\
        \mu_{m-1}&\longmapsto (t-s_{m-1})(1-t+s_{m-2})\cdots (1-t+s_{2})(1-t) \mu\\
        \mu_m&\longmapsto (1-t+s_{m-1})(1-t+s_{m-2})\cdots (1-t+s_{2})(1-t)\mu
        \end{align*}
        For $m=2$, the meridians are mapped as follows:
        \begin{align*}
           \mu_1&\longmapsto t\mu\\
           \mu_2&\longmapsto(1- t)\mu
       \end{align*}
        For contact $(-\frac{1}{m})$-surgery, we have the following:
       \begin{align*}
           \mu_k&\longmapsto t(1-t)^{k-1}\mu\,\textrm{ For }\,k=1,2,\cdots, m\\
           \mu_m&\longmapsto(1-t)^{m-1}\mu
       \end{align*}  
\end{lemma}

\begin{proof}
    The proof is similar to the proof of Lemma~\ref{lem:meridian}. 
\end{proof}

\begin{proof}[Proof of Theorem~\ref{thm:Euler_algorithm}]
 Lemma~\ref{lem:meridian} and Lemma~\ref{lem:meridian_neg} tells us how to express the meridians $\mu_i$ in the surgery description from Equation~\ref{eq:pushoff} in terms of $\mu$. By plugging this description into the formula from Lemma~\ref{lem:d3} and simplifying the expression we get the statement of the theorem.
\end{proof}

\section{Tight surgeries on Legendrian unknots}
In this section, we will classify which rational  contact surgeries on Legendrian unknots yield tight contact manifolds and which yield overtwisted contact manifolds. Our main result reads as follows.

\begin{theorem}\label{thm:unknot_surgery_classification}
    Let $U$ be a Legendrian unknot with $\tb(U)=t\leq-1$. Then $U(r)$ is tight if and only if
    \begin{enumerate}
        \item $r<0$, or
        \item $r\geq-t$ and $U$ is only stabilized with one sign (or not stabilized at all) and the first stabilization of $U(r)$ in the transformation lemma has the same sign as the stabilizations of $U$ (or has arbitrary sign in case that $U$ is not stabilized).
    \end{enumerate}
\end{theorem}

We first start with a simple lemma.

\begin{lemma}\label{lem:ot_easy}
    Let $U$ be a Legendrian unknot with $\tb(U)=t\leq-1$. Then every $U(r)$ is overtwisted if $0<r<-t$.
\end{lemma}

\begin{proof}
    Let $r=p/q$ for coprime integers $p>q>0$. We denote by $K$ a Legendrian meridian of $U$ with $\tb(K)=-1$ seen as a knot in $(S^3,\xist)$. Then $K$ can also be seen as a rational Legendrian unknot in $U(r)$, and thus any rational Seifert surface $F$ of $K$ in $U(r)$ has Euler characteristic $\chi(F)\leq1$. On the other hand, we compute, for example with~\cite[Section~4.5]{phdthesis}, the rational Thurston--Bennequin invariant of $K$ in $U(r)$ to be $\tb_\Q(K)=-1-\frac{q}{p+qt}>-1$. Thus $K$ violates the rational Bennequin inequality~\cite{Baker_Etnyre_ratTB} in $U(r)$ which implies that $U(r)$ is overtwisted.
\end{proof}

\begin{proof}[Proof of Theorem~\ref{thm:unknot_surgery_classification}]
First, we recall that contact surgery with a negative contact surgery coefficient preserves fillability. In particular, $K(r)$ is tight whenever $r<0$. Thus in the following, we will restrict to positive contact surgery coefficients.

We denote by $U=U_0$ the Legendrian unknot with Thurston--Bennequin invariant $\tb=-1$. Then it follows from Lemma~\ref{lem:ot_easy} that $U(r)$ is overtwisted for $r\in[0,1)$. If $r>1$, we can use the transformation lemma to write
\begin{equation*}
    U(r)=U(+1){\def\svgwidth{1,6ex}\,\,
\begingroup%
  \makeatletter%
  \providecommand\color[2][]{%
    \errmessage{(Inkscape) Color is used for the text in Inkscape, but the package 'color.sty' is not loaded}%
    \renewcommand\color[2][]{}%
  }%
  \providecommand\transparent[1]{%
    \errmessage{(Inkscape) Transparency is used (non-zero) for the text in Inkscape, but the package 'transparent.sty' is not loaded}%
    \renewcommand\transparent[1]{}%
  }%
  \providecommand\rotatebox[2]{#2}%
  \ifx\svgwidth\undefined%
    \setlength{\unitlength}{49.54732089bp}%
    \ifx\svgscale\undefined%
      \relax%
    \else%
      \setlength{\unitlength}{\unitlength * \real{\svgscale}}%
    \fi%
  \else%
    \setlength{\unitlength}{\svgwidth}%
  \fi%
  \global\let\svgwidth\undefined%
  \global\let\svgscale\undefined%
  \makeatother%
  \begin{picture}(1,1.00059637)%
    \put(0,0){\includegraphics[width=\unitlength,page=1]{PushOff.pdf}}%
  \end{picture}%
\endgroup%
\,\,} U\left(\frac{1}{\frac{1}{r}-1}\right).
\end{equation*}
Since $U(+1)$ yields the tight contact structure on $S^1\times S^2$ and $\frac{1}{\frac{1}{r}-1}$ is negative, it follows that every $U(r)$ is tight. Together with Lemma~\ref{lem:ot_easy} we have shown that $U(r)$ is overtwisted if and only if $r\in[0,1)$.

Now we denote by $U_{-t-1}$ a Legendrian unknot with Thurston--Bennequin invariant $\tb(U_{-t-1})=t\leq-2$. If $U_{-t-1}$ is stabilized with two different signs then every $U(r)$, $r\geq0$, is overtwisted~\cite{Ozbagci}, cf.~\cite[Theorem~3.3]{EKS_contact_surgery_numbers}. So we assume from now that $U_{-t-1}$ is stabilized with only one sign, i.e.\ $|\rot(U_{-t-1})|=t-1$. If $r\in[0,-t)$ then $U_{-t-1}(r)$ is always overtwisted by Lemma~\ref{lem:ot_easy}. Next, we consider the case $r\geq-t$. We choose coprime integers $p,q>0$, such that $r=p/q$. Then $p+tq\geq 0$ and 
\begin{equation*}
    2\geq\frac{1}{\frac{1}{r}-1}=\frac{p}{p-q}\geq1.
\end{equation*}
Thus the first terms in the negative continued fraction expansion read as
\begin{equation*}
    \frac{p}{p-q}=-3+1-\frac{1}{-\frac{p-q}{p-2q}}.
\end{equation*}
Then the transformation lemma yields
\begin{align*}
    U_{-t-1}(p/q)=U_{-t-1}(+1){\def\svgwidth{1,6ex}\,\,
\begingroup%
  \makeatletter%
  \providecommand\color[2][]{%
    \errmessage{(Inkscape) Color is used for the text in Inkscape, but the package 'color.sty' is not loaded}%
    \renewcommand\color[2][]{}%
  }%
  \providecommand\transparent[1]{%
    \errmessage{(Inkscape) Transparency is used (non-zero) for the text in Inkscape, but the package 'transparent.sty' is not loaded}%
    \renewcommand\transparent[1]{}%
  }%
  \providecommand\rotatebox[2]{#2}%
  \ifx\svgwidth\undefined%
    \setlength{\unitlength}{49.54732089bp}%
    \ifx\svgscale\undefined%
      \relax%
    \else%
      \setlength{\unitlength}{\unitlength * \real{\svgscale}}%
    \fi%
  \else%
    \setlength{\unitlength}{\svgwidth}%
  \fi%
  \global\let\svgwidth\undefined%
  \global\let\svgscale\undefined%
  \makeatother%
  \begin{picture}(1,1.00059637)%
    \put(0,0){\includegraphics[width=\unitlength,page=1]{PushOff.pdf}}%
  \end{picture}%
\endgroup%
\,\,} U_{-t-1,1}(-1){\def\svgwidth{1,6ex}\,\,
\begingroup%
  \makeatletter%
  \providecommand\color[2][]{%
    \errmessage{(Inkscape) Color is used for the text in Inkscape, but the package 'color.sty' is not loaded}%
    \renewcommand\color[2][]{}%
  }%
  \providecommand\transparent[1]{%
    \errmessage{(Inkscape) Transparency is used (non-zero) for the text in Inkscape, but the package 'transparent.sty' is not loaded}%
    \renewcommand\transparent[1]{}%
  }%
  \providecommand\rotatebox[2]{#2}%
  \ifx\svgwidth\undefined%
    \setlength{\unitlength}{49.54732089bp}%
    \ifx\svgscale\undefined%
      \relax%
    \else%
      \setlength{\unitlength}{\unitlength * \real{\svgscale}}%
    \fi%
  \else%
    \setlength{\unitlength}{\svgwidth}%
  \fi%
  \global\let\svgwidth\undefined%
  \global\let\svgscale\undefined%
  \makeatother%
  \begin{picture}(1,1.00059637)%
    \put(0,0){\includegraphics[width=\unitlength,page=1]{PushOff.pdf}}%
  \end{picture}%
\endgroup%
\,\,} U_{-t-1,1}\left(-\frac{p-q}{p-2q}\right).
\end{align*}
If the extra stabilization of $U_{-t-1,1}$ has a different sign than the stabilizations of $U_{-t-1}$ then the contact structure is overtwisted~\cite[Theorem~3.3]{EKS_contact_surgery_numbers} if the extra stabilization has the same sign we use the lantern destabilization~\cite{Lisca_Stipsicz_11_transverse,EKS_contact_surgery_numbers} and the transformation lemma again to write
\begin{align*}
    U_{-t-1}(p/q)=&U_{-t-1}(+1){\def\svgwidth{1,6ex}\,\,
\begingroup%
  \makeatletter%
  \providecommand\color[2][]{%
    \errmessage{(Inkscape) Color is used for the text in Inkscape, but the package 'color.sty' is not loaded}%
    \renewcommand\color[2][]{}%
  }%
  \providecommand\transparent[1]{%
    \errmessage{(Inkscape) Transparency is used (non-zero) for the text in Inkscape, but the package 'transparent.sty' is not loaded}%
    \renewcommand\transparent[1]{}%
  }%
  \providecommand\rotatebox[2]{#2}%
  \ifx\svgwidth\undefined%
    \setlength{\unitlength}{49.54732089bp}%
    \ifx\svgscale\undefined%
      \relax%
    \else%
      \setlength{\unitlength}{\unitlength * \real{\svgscale}}%
    \fi%
  \else%
    \setlength{\unitlength}{\svgwidth}%
  \fi%
  \global\let\svgwidth\undefined%
  \global\let\svgscale\undefined%
  \makeatother%
  \begin{picture}(1,1.00059637)%
    \put(0,0){\includegraphics[width=\unitlength,page=1]{PushOff.pdf}}%
  \end{picture}%
\endgroup%
\,\,} U_{-t-1,1}(-1){\def\svgwidth{1,6ex}\,\,
\begingroup%
  \makeatletter%
  \providecommand\color[2][]{%
    \errmessage{(Inkscape) Color is used for the text in Inkscape, but the package 'color.sty' is not loaded}%
    \renewcommand\color[2][]{}%
  }%
  \providecommand\transparent[1]{%
    \errmessage{(Inkscape) Transparency is used (non-zero) for the text in Inkscape, but the package 'transparent.sty' is not loaded}%
    \renewcommand\transparent[1]{}%
  }%
  \providecommand\rotatebox[2]{#2}%
  \ifx\svgwidth\undefined%
    \setlength{\unitlength}{49.54732089bp}%
    \ifx\svgscale\undefined%
      \relax%
    \else%
      \setlength{\unitlength}{\unitlength * \real{\svgscale}}%
    \fi%
  \else%
    \setlength{\unitlength}{\svgwidth}%
  \fi%
  \global\let\svgwidth\undefined%
  \global\let\svgscale\undefined%
  \makeatother%
  \begin{picture}(1,1.00059637)%
    \put(0,0){\includegraphics[width=\unitlength,page=1]{PushOff.pdf}}%
  \end{picture}%
\endgroup%
\,\,} U_{-t-1,1}\left(-\frac{p-q}{p-2q}\right)\\
    =&U_{-t-2}(+1){\def\svgwidth{1,6ex}\,\,
\begingroup%
  \makeatletter%
  \providecommand\color[2][]{%
    \errmessage{(Inkscape) Color is used for the text in Inkscape, but the package 'color.sty' is not loaded}%
    \renewcommand\color[2][]{}%
  }%
  \providecommand\transparent[1]{%
    \errmessage{(Inkscape) Transparency is used (non-zero) for the text in Inkscape, but the package 'transparent.sty' is not loaded}%
    \renewcommand\transparent[1]{}%
  }%
  \providecommand\rotatebox[2]{#2}%
  \ifx\svgwidth\undefined%
    \setlength{\unitlength}{49.54732089bp}%
    \ifx\svgscale\undefined%
      \relax%
    \else%
      \setlength{\unitlength}{\unitlength * \real{\svgscale}}%
    \fi%
  \else%
    \setlength{\unitlength}{\svgwidth}%
  \fi%
  \global\let\svgwidth\undefined%
  \global\let\svgscale\undefined%
  \makeatother%
  \begin{picture}(1,1.00059637)%
    \put(0,0){\includegraphics[width=\unitlength,page=1]{PushOff.pdf}}%
  \end{picture}%
\endgroup%
\,\,} U_{-t-2,1}\left(-\frac{p-q}{p-2q}\right)\\
    =&U_{-t-2}\left(-\frac{p-q}{q}\right)\\
    =&U_{-t-2}(p/q-1).
\end{align*}
Inductively, this implies
\begin{align*}
    U_{-t-1}(p/q)=U(p/q-t+1)
\end{align*}
and since $p/q-t+1>1$ it follows that this contact structure is tight as claimed.
\end{proof}

\section{Proofs of the main results}\label{sec:proof}
\subsection{The Brieskorn sphere \texorpdfstring{$\Sigma(2,3,11)$}{E(2,3,11)}}
\begin{proof}[Proof of Theorem~\ref{thm:Brieskorn}]
Let $(\Sigma(2,3,11),\xi)$ be contactomorphic to $K(r)$ for some Legendrian knot $K$ in $(S^3,\xist)$ and some $r\in\Q\setminus\{0\}$. Then a recent result of Baldwin--Sivek~\cite{baldwin2022characterizing} says that either $K$ is a Legendrian realization of the positive twist knot $K5a1$ with topological surgery coefficient $-1$ or $K$ is a Legendrian realization of the left-handed trefoil knot $-K3a1$ with topological surgery coefficient $-1/2$. These surgery diagrams are shown in Figure~\ref{fig:topological_surgery_diagrams}.

\begin{figure}[htbp] 
\centering
\def\svgwidth{0.99\columnwidth}
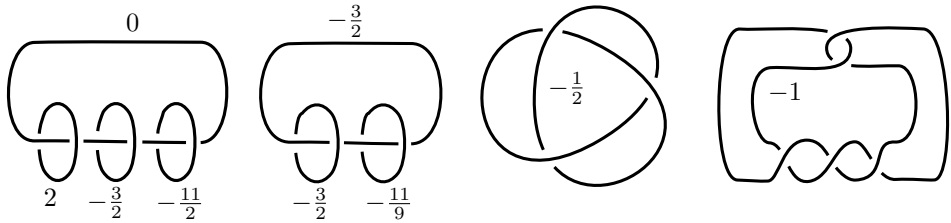
\caption{Four different surgery diagrams of the Brieskorn sphere $\Sigma(2,3,11)$. The left diagram shows the description of that manifold as Seifert fibered space.}
	\label{fig:topological_surgery_diagrams}
\end{figure}

To prove statement $(1)$, we only need to consider the first possibility (since the latter will always yield a rational contact surgery). Since $K5a1$ is Legendrian simple, every Legendrian realization $K$ of $K5a1$ is obtained by stabilizing the unique Legendrian representative of $K5a1$ with $\tb=-8$ and $\rot=1$~\cite{Etnyre_Ng_vertesi_twist}. Then we can readily apply Corollary~4.7 from~\cite{EKS_contact_surgery_numbers} to deduce statement $(1)$.

For $(2)$, we also need to consider Legendrian realizations of the left-handed trefoil. Let $K$ be a Legendrian left-handed trefoil. From the classification of Legendrian realizations of $-K3a1$, we deduce that $K$ has Thurston--Bennequin invariant $t\leq-6$~\cite{Etnyre_Honda_torus_knots}. Thus $K(\frac{-1-2t}{2})$ yields a contact structure on $\Sigma(2,3,11)$. (And together with the contact structures obtained in $(1)$ this is the complete list of contact structures on $\Sigma(2,3,11)$ with $\cs=1$.) We compute the possible $\de_3$-invariants of these contact structures. For that, we use Lemma~\ref{lem:Kirby} to change the contact surgery diagram into one with only reciprocal integer coefficients, so that Lemma~\ref{lem:d3} applies. We apply the transformation lemma to see that
\begin{align*}
    K\left(\frac{-1-2t}{2}\right)&=K(+1){\def\svgwidth{1,6ex}\,\,
\begingroup%
  \makeatletter%
  \providecommand\color[2][]{%
    \errmessage{(Inkscape) Color is used for the text in Inkscape, but the package 'color.sty' is not loaded}%
    \renewcommand\color[2][]{}%
  }%
  \providecommand\transparent[1]{%
    \errmessage{(Inkscape) Transparency is used (non-zero) for the text in Inkscape, but the package 'transparent.sty' is not loaded}%
    \renewcommand\transparent[1]{}%
  }%
  \providecommand\rotatebox[2]{#2}%
  \ifx\svgwidth\undefined%
    \setlength{\unitlength}{49.54732089bp}%
    \ifx\svgscale\undefined%
      \relax%
    \else%
      \setlength{\unitlength}{\unitlength * \real{\svgscale}}%
    \fi%
  \else%
    \setlength{\unitlength}{\svgwidth}%
  \fi%
  \global\let\svgwidth\undefined%
  \global\let\svgscale\undefined%
  \makeatother%
  \begin{picture}(1,1.00059637)%
    \put(0,0){\includegraphics[width=\unitlength,page=1]{PushOff.pdf}}%
  \end{picture}%
\endgroup%
\,\,} K\left(-\frac{1+2t}{3+2t}\right)=K(+1){\def\svgwidth{1,6ex}\,\,
\begingroup%
  \makeatletter%
  \providecommand\color[2][]{%
    \errmessage{(Inkscape) Color is used for the text in Inkscape, but the package 'color.sty' is not loaded}%
    \renewcommand\color[2][]{}%
  }%
  \providecommand\transparent[1]{%
    \errmessage{(Inkscape) Transparency is used (non-zero) for the text in Inkscape, but the package 'transparent.sty' is not loaded}%
    \renewcommand\transparent[1]{}%
  }%
  \providecommand\rotatebox[2]{#2}%
  \ifx\svgwidth\undefined%
    \setlength{\unitlength}{49.54732089bp}%
    \ifx\svgscale\undefined%
      \relax%
    \else%
      \setlength{\unitlength}{\unitlength * \real{\svgscale}}%
    \fi%
  \else%
    \setlength{\unitlength}{\svgwidth}%
  \fi%
  \global\let\svgwidth\undefined%
  \global\let\svgscale\undefined%
  \makeatother%
  \begin{picture}(1,1.00059637)%
    \put(0,0){\includegraphics[width=\unitlength,page=1]{PushOff.pdf}}%
  \end{picture}%
\endgroup%
\,\,} K_1\left(-\frac{1}{-t-2}\right){\def\svgwidth{1,6ex}\,\,
\begingroup%
  \makeatletter%
  \providecommand\color[2][]{%
    \errmessage{(Inkscape) Color is used for the text in Inkscape, but the package 'color.sty' is not loaded}%
    \renewcommand\color[2][]{}%
  }%
  \providecommand\transparent[1]{%
    \errmessage{(Inkscape) Transparency is used (non-zero) for the text in Inkscape, but the package 'transparent.sty' is not loaded}%
    \renewcommand\transparent[1]{}%
  }%
  \providecommand\rotatebox[2]{#2}%
  \ifx\svgwidth\undefined%
    \setlength{\unitlength}{49.54732089bp}%
    \ifx\svgscale\undefined%
      \relax%
    \else%
      \setlength{\unitlength}{\unitlength * \real{\svgscale}}%
    \fi%
  \else%
    \setlength{\unitlength}{\svgwidth}%
  \fi%
  \global\let\svgwidth\undefined%
  \global\let\svgscale\undefined%
  \makeatother%
  \begin{picture}(1,1.00059637)%
    \put(0,0){\includegraphics[width=\unitlength,page=1]{PushOff.pdf}}%
  \end{picture}%
\endgroup%
\,\,} K_{1,1}(-1),
\end{align*}
where the last equality comes from the negative continued fraction expansion 
\begin{equation}\label{eq:cfe}
    -\frac{1+2t}{3+2t}=[\underbrace{-2,\ldots,-2}_{-t-2},-3].
\end{equation}
From this surgery description, we deduce the generalized linking matrix to be
\begin{equation*}
    Q=\begin{pmatrix}
1+t & -t(t+2) & t\\
t & 1-t-t^2 & t-1\\
t& -(t+2)(t-1)&t-3
\end{pmatrix}.
\end{equation*}
Next, we compute the signature of $Q$ to be $1$ and solve $Q\mathbf{b}=\mathbf{r}$ from which we get the $\de_3$-invariants. We write $r$ for the rotation number of $L$. In the case of $\mathbf r=(r,r\pm1,r\pm2)$, we get
\begin{equation*}
    \de_3=-(t\pm r)-\frac{(t\pm r)^2+5}{2}
\end{equation*}
and for $\mathbf r=(r,r\pm1,r)$, we compute
\begin{equation*}
    \de_3=-2(t\pm r)-\frac{(t\pm r)^2+7}{2}.
\end{equation*}
By writing $t\pm r=2m+1$ for $m\leq-3$ the claimed formulas follow.
\end{proof}

\begin{proof}[Proof of Corollary~\ref{cor:Brieskorn1}]
In the proof of Theorem~\ref{thm:Brieskorn} we created all contact surgery diagrams along a single Legendrian knot yielding contact structures on $\Sigma(2,3,11)$. The lower bound of the corollary follows by observing that these were all with contact surgery coefficients not of the form of a reciprocal integer. The upper bound follows from Proposition~6.8 in~\cite{EKS_contact_surgery_numbers}.
\end{proof}

\begin{proof}[Proof of Corollary~\ref{cor:Brieskorn_tight}]
Figure~\ref{fig:contact_Brieskorn} shows how to obtain a contact surgery diagram of $(\Sigma(2,3,11),\xist)$ along a two-component Legendrian link. This yields the upper bound of $2$ for $\cs$ and by applying the transformation lemma it also yields the upper bound $3$ for $\cs_{1/\Z}$. For the lower bound, we compute $\de_3(\xist)=2$ and observe that $2$ is never attained for the possible $\de_3$-invariants in Theorem~\ref{thm:Brieskorn}. Finally, the upper bound of $4$ for $\cs_{\pm1}$ is given by Theorem~6.9 in~\cite{EKS_contact_surgery_numbers}.
\end{proof}

\begin{figure}[htbp] 
\centering
\def\svgwidth{0.99\columnwidth}
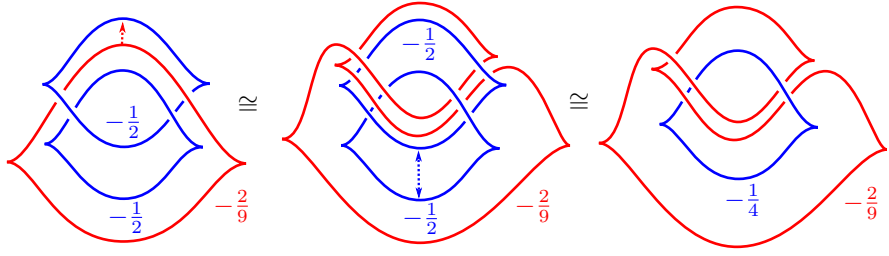
\caption{Three different contact surgery diagrams of $(\Sigma(2,3,11),\xist)$. The left diagram is obtained by Legendrian realizing the topological surgery diagram of $\Sigma(2,3,11)$ along the three component link in Figure~\ref{fig:topological_surgery_diagrams}. Since it only has negative coefficients it represents the tight contact structure~\cite{Wand}. The middle diagram is obtained by performing a handle slide~\cite{casals2021stein} of the red curve along the red dotted arc over the blue curve. The two curves in the middle diagram bound an obvious annulus. Thus we can apply the replacement lemma to get the surgery diagram on the right along a two-component link.}
	\label{fig:contact_Brieskorn}
\end{figure}

\subsection{The mirror of the Brieskorn sphere \texorpdfstring{$\Sigma(2,3,11)$}{E(2,3,11)}}
\begin{proof}[Proof of Theorem~\ref{thm:-Brieskorn}]
First, we observe that if topological $r$-surgery (i.e.\ with respect to the Seifert framing) along a smooth knot $K$ yields the smooth manifold $M$, then $(-r)$-surgery along the mirror $-K$ of $K$ yields the mirrored manifold $-M$. Thus we can deduce from~\cite{baldwin2022characterizing} that any Legendrian knot $K$ such that $K(r)$ yields a contact structure on $-\Sigma(2,3,11)$ is either isotopic to a Legendrian realization of $-K5a1$ with topological surgery coefficient $+1$ or isotopic to a Legendrian realization of the right-handed trefoil $K3a1$ with topological surgery coefficient $1/2$. Since the classification of all Legendrian realizations of $K3a1$~\cite{Etnyre_Honda_torus_knots} and $-K5a1$~\cite{Etnyre_Ng_vertesi_twist} are known, the same strategy as above works.

Statement $(1)$ follows directly from the classification of Legendrian realization of $-K5a1$ and Corollary~4.76 from~\cite{EKS_contact_surgery_numbers}.

For statement $(2)$ we consider a Legendrian realization $K$ of the right-handed trefoil $K3a1$ with Thurston--Bennequin invariant $t\leq1$. Then $K(\frac{1-2t}{2})$ yields a contact structure on $\Sigma(2,3,11)$. For $t=1$ this corresponds to a contact $(-1/2)$-surgery. Since negative surgery preserves tightness~\cite{Wand} the resulting contact structure is contactomorphic to the standard tight contact structure $\xist$, with $\de_3=-1$. If $t=0$, we perform a contact $(1/2)$-surgery and we compute directly that $\de_3=0$. For $t\leq-1$, we use Lemma~\ref{lem:Kirby} to change the contact surgery diagram to
\begin{align*}
    K\left(\frac{1-2t}{2}\right)&=K(+1){\def\svgwidth{1,6ex}\,\,
\begingroup%
  \makeatletter%
  \providecommand\color[2][]{%
    \errmessage{(Inkscape) Color is used for the text in Inkscape, but the package 'color.sty' is not loaded}%
    \renewcommand\color[2][]{}%
  }%
  \providecommand\transparent[1]{%
    \errmessage{(Inkscape) Transparency is used (non-zero) for the text in Inkscape, but the package 'transparent.sty' is not loaded}%
    \renewcommand\transparent[1]{}%
  }%
  \providecommand\rotatebox[2]{#2}%
  \ifx\svgwidth\undefined%
    \setlength{\unitlength}{49.54732089bp}%
    \ifx\svgscale\undefined%
      \relax%
    \else%
      \setlength{\unitlength}{\unitlength * \real{\svgscale}}%
    \fi%
  \else%
    \setlength{\unitlength}{\svgwidth}%
  \fi%
  \global\let\svgwidth\undefined%
  \global\let\svgscale\undefined%
  \makeatother%
  \begin{picture}(1,1.00059637)%
    \put(0,0){\includegraphics[width=\unitlength,page=1]{PushOff.pdf}}%
  \end{picture}%
\endgroup%
\,\,} K\left(\frac{1-2t}{1+2t}\right)=K(+1){\def\svgwidth{1,6ex}\,\,
\begingroup%
  \makeatletter%
  \providecommand\color[2][]{%
    \errmessage{(Inkscape) Color is used for the text in Inkscape, but the package 'color.sty' is not loaded}%
    \renewcommand\color[2][]{}%
  }%
  \providecommand\transparent[1]{%
    \errmessage{(Inkscape) Transparency is used (non-zero) for the text in Inkscape, but the package 'transparent.sty' is not loaded}%
    \renewcommand\transparent[1]{}%
  }%
  \providecommand\rotatebox[2]{#2}%
  \ifx\svgwidth\undefined%
    \setlength{\unitlength}{49.54732089bp}%
    \ifx\svgscale\undefined%
      \relax%
    \else%
      \setlength{\unitlength}{\unitlength * \real{\svgscale}}%
    \fi%
  \else%
    \setlength{\unitlength}{\svgwidth}%
  \fi%
  \global\let\svgwidth\undefined%
  \global\let\svgscale\undefined%
  \makeatother%
  \begin{picture}(1,1.00059637)%
    \put(0,0){\includegraphics[width=\unitlength,page=1]{PushOff.pdf}}%
  \end{picture}%
\endgroup%
\,\,} K_1\left(-\frac{1}{-t-1}\right){\def\svgwidth{1,6ex}\,\,
\begingroup%
  \makeatletter%
  \providecommand\color[2][]{%
    \errmessage{(Inkscape) Color is used for the text in Inkscape, but the package 'color.sty' is not loaded}%
    \renewcommand\color[2][]{}%
  }%
  \providecommand\transparent[1]{%
    \errmessage{(Inkscape) Transparency is used (non-zero) for the text in Inkscape, but the package 'transparent.sty' is not loaded}%
    \renewcommand\transparent[1]{}%
  }%
  \providecommand\rotatebox[2]{#2}%
  \ifx\svgwidth\undefined%
    \setlength{\unitlength}{49.54732089bp}%
    \ifx\svgscale\undefined%
      \relax%
    \else%
      \setlength{\unitlength}{\unitlength * \real{\svgscale}}%
    \fi%
  \else%
    \setlength{\unitlength}{\svgwidth}%
  \fi%
  \global\let\svgwidth\undefined%
  \global\let\svgscale\undefined%
  \makeatother%
  \begin{picture}(1,1.00059637)%
    \put(0,0){\includegraphics[width=\unitlength,page=1]{PushOff.pdf}}%
  \end{picture}%
\endgroup%
\,\,} K_{1,1}(-1),
\end{align*}
where the last equality comes from the negative continued fraction expansion from Equation~(\ref{eq:cfe}). From that, we deduce the generalized linking matrix to be
\begin{equation*}
    Q=\begin{pmatrix}
1+t & -t^2-t & t\\
t & -t^2 & t-1\\
t& -t^2+1&t-3
\end{pmatrix}.
\end{equation*}
Next, we compute the signature of $Q$ to be $-1$ and solve $Q\mathbf{b}=\mathbf{r}$ from which we get the $\de_3$-invariants. We write $r$ for the rotation number of $L$. In the case of $\mathbf r=(r,r\pm1,r\pm2)$, we get
\begin{equation*}
    \de_3=\frac{(t\pm r)^2+1}{2}-1
\end{equation*}
and for $\mathbf r=(r,r\pm 1,r)$, we compute
\begin{equation*}
    \de_3=\frac{(t\pm r)^2+1}{2}+(t\pm r).
\end{equation*}
By writing $t\pm r=2m+1$ for $m\leq-1$ the claimed formulas follow.
\end{proof}

\begin{proof}[Proof of Corollary~\ref{cor:-Brieskorn1}]
In the proof of Theorem~\ref{thm:Brieskorn} we created all contact surgery diagrams of contact structures on $\Sigma(2,3,11)$ along a single Legendrian knot. We observe that the only diagrams with reciprocal integer contact surgery coefficients are the ones shown in Figure~\ref{fig:contact_-Brieskorn}. This directly implies the statement.
\end{proof}

\begin{proof}[Proof of Corollary~\ref{cor:-Brieskorn_tight}]
In Figure~\ref{fig:contact_-Brieskorn} we can replace the contact $(-1/2)$-surgery by two contact $(-1)$-surgeries along two Legendrian push-offs and thus $\cs_{\pm1}(\xist)\leq2$. The other statements follow directly from the proof of Theorem~\ref{thm:-Brieskorn} and Corollary~\ref{cor:-Brieskorn1} by noticing that $\de_3(\xist)=-1$. 
\end{proof}

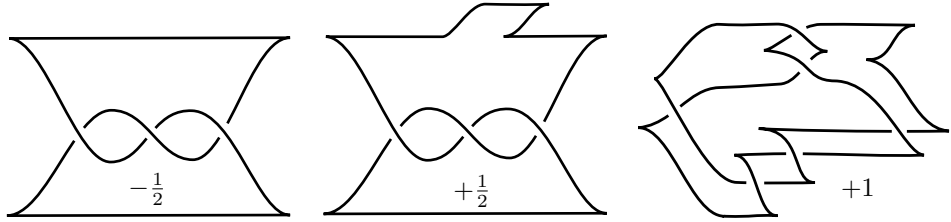
\begin{figure}[htbp] 
\centering
\def\svgwidth{0.99\columnwidth}
\begingroup%
  \makeatletter%
  \providecommand\color[2][]{%
    \errmessage{(Inkscape) Color is used for the text in Inkscape, but the package 'color.sty' is not loaded}%
    \renewcommand\color[2][]{}%
  }%
  \providecommand\transparent[1]{%
    \errmessage{(Inkscape) Transparency is used (non-zero) for the text in Inkscape, but the package 'transparent.sty' is not loaded}%
    \renewcommand\transparent[1]{}%
  }%
  \providecommand\rotatebox[2]{#2}%
  \newcommand*\fsize{\dimexpr\f@size pt\relax}%
  \newcommand*\lineheight[1]{\fontsize{\fsize}{#1\fsize}\selectfont}%
  \ifx\svgwidth\undefined%
    \setlength{\unitlength}{562.40989545bp}%
    \ifx\svgscale\undefined%
      \relax%
    \else%
      \setlength{\unitlength}{\unitlength * \real{\svgscale}}%
    \fi%
  \else%
    \setlength{\unitlength}{\svgwidth}%
  \fi%
  \global\let\svgwidth\undefined%
  \global\let\svgscale\undefined%
  \makeatother%
  \begin{picture}(1,0.22843293)%
    \lineheight{1}%
    \setlength\tabcolsep{0pt}%
    \put(0.12920376,0.02379814){\color[rgb]{0,0,0}\makebox(0,0)[lt]{\lineheight{0.1}\smash{\begin{tabular}[t]{l}$-\frac{1}{2}$\end{tabular}}}}%
    \put(0.47395089,0.02151439){\color[rgb]{0,0,0}\makebox(0,0)[lt]{\lineheight{0.1}\smash{\begin{tabular}[t]{l}$+\frac{1}{2}$\end{tabular}}}}%
    \put(0.88532933,0.02283473){\color[rgb]{0,0,0}\makebox(0,0)[lt]{\lineheight{0.1}\smash{\begin{tabular}[t]{l}$+1$\end{tabular}}}}%
    \put(0,0){\includegraphics[width=\unitlength,page=1]{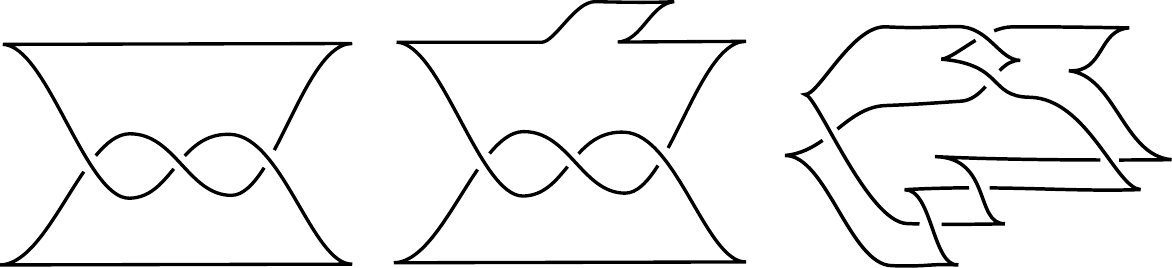}}%
  \end{picture}%
\endgroup%

\caption{The only three contact surgery diagrams along a single Legendrian knot with reciprocal integer contact surgery coefficients that yield contact structures on $-\Sigma(2,3,11)$. The left diagram yields $\xist$ while the other two diagrams represent the overtwisted contact structure with $\de_3=0$.}
	\label{fig:contact_-Brieskorn}
\end{figure}

\subsection{Surgery diagrams of the lens spaces \texorpdfstring{$L(4m+3,4)$}{L(4m+3,4)}}
Recall from the introduction that for an integer $m\geq1$ we can represent $L(4m+3,4)$ in its \textit{standard surgery diagram}: The $(-m-3/4)$-surgery along the unknot $U$. 

As preparation for the classification of the contact structures on $L(4m+3,4)$ that have contact surgery number $1$ we enumerate in this section all smooth surgery diagrams of $L(4m+3,4)$ along a single knot $K$. Moreover, we will present the meridians $\mu_K$ of these surgery knots in the standard basis of $H_1(L(4m+3,4))$ generated by the meridian $\mu$ of the unknot in the standard surgery description. The result is as follows.

\begin{theorem}\label{thm:surgery_diagrams_of_lens}
    Let $K$ be a knot in $S^3$ and $r\in\Q$ be a surgery coefficient measured with respect to the Seifert framing such that $K(r)$ is diffeomorphic to $L(4m+3,4)$. Then $(K,r)$ is either
    \begin{itemize}
        \item the torus knot $T_{(2,-(2m+1))}$ with surgery coefficient $r=-4m-3$, or
        \item an unknot $U_k$ with surgery coefficient $-\frac{4m+3}{4-4km-3k}$ for an integer $k\in\Z$, or
        \item an unknot $U^*_k$ with surgery coefficient $-\frac{4m+3}{m+1-4km-3k}$ for an integer $k\in\Z$. 
    \end{itemize}
    Moreover, the meridians $\mu_T$, $\mu_k$, and $\mu^*_k$ of the torus knot $T_{(2,-(2m+1))}$, $U_k$, and $U^*_k$ can be expressed in the standard basis as
    \begin{align*}
        \mu_T&\longmapsto 2(m+1)\mu,\\
        \mu_k&\longmapsto \mu,\\
        \mu^*_k&\longmapsto (m+1)\mu.\\
    \end{align*}
\end{theorem}

\begin{proof}
    First, the cyclic surgery theorem~\cite{cyclic_dehn_surgery} implies that $K$ is either an unknot or $r$ is an integer. If $r$ is an integer then a result of Rasmussen~\cite{rasmussen2007lens} implies that $K$ is isotopic to the torus knot $T_{(2,-(2m+1))}$ with surgery coefficient $r=-4m-3$.

    If $K$ is an unknot then the classification of lens spaces gives us all surgery diagrams along unknots yielding $L(4m+3,4)$. For that, we consider the surgery plumbing graph corresponding to the continued fraction expansion of this lens space. I.e.\ we consider the positive Hopf link with surgery coefficients $-4$ on the first component $K_1$ and surgery coefficient $-m-1$ on the second component $K_2$ shown in Figure~\ref{fig:Hopf}. 

\begin{figure}[htbp] 
\centering
\def\svgwidth{0.5\columnwidth}
 \labellist
 	\small\hair 2pt
\pinlabel $K_1$ at 370 250
 	\pinlabel $K_2$ at 160 250
 	\pinlabel $-m-1$ at  35 150
 	\pinlabel $-4$ at 370 150
 	\endlabellist
  \includegraphics[scale=0.3]{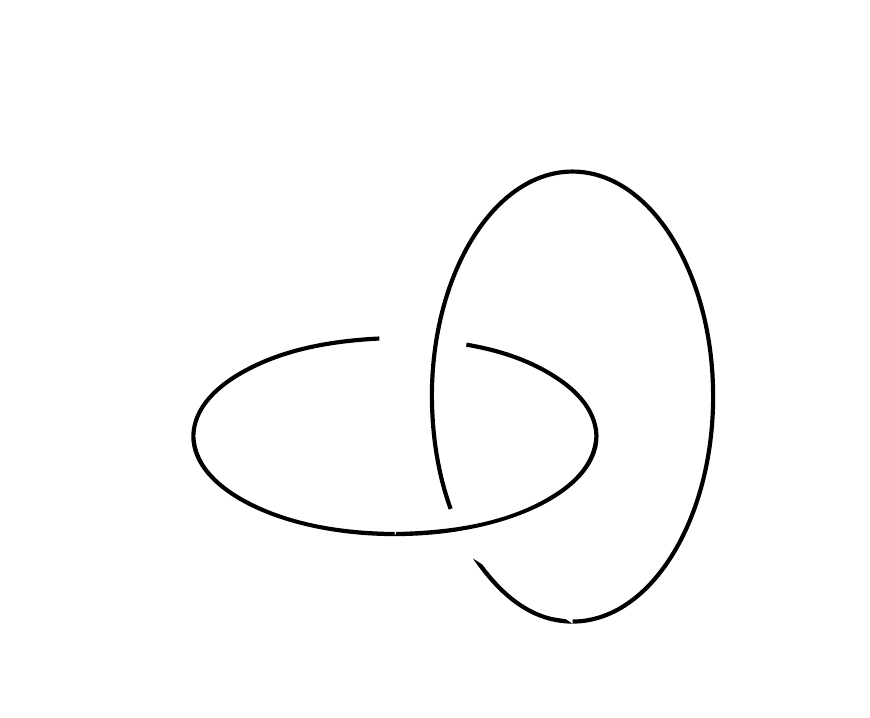}	
\caption{A surgery diagram on a positive Hopf link yielding $L(4m+3,4)$.}
	\label{fig:Hopf}    
\end{figure}
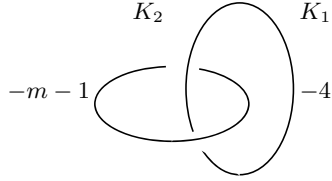

By slam dunking $K_1$ away, we get the standard surgery diagram of $L(4m+3,4)$. In Figure~\ref{fig:Hopf_slam_dunk} we keep track of the meridians $\mu_1$ and $\mu_2$ of $K_1$ and $K_2$ under this slam dunk, which shows that under the slam dunk diffeomorphism, the meridians get mapped as follows
\begin{align*}
    \mu_1&\longmapsto(m+1)\mu,\\
    \mu_2&\longmapsto\mu.
\end{align*}
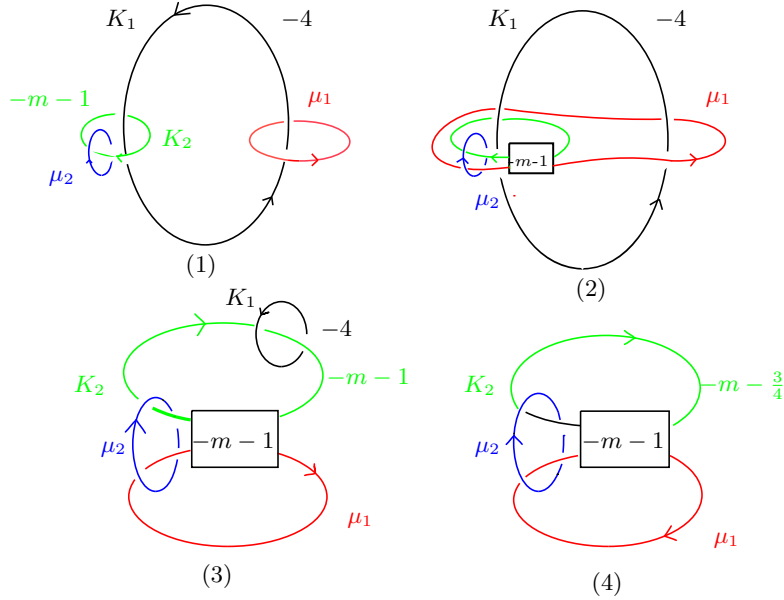
\begin{figure}[htbp] 
\centering
\def\svgwidth{0.5\columnwidth}
        \labellist
 	\small\hair 2pt
        \pinlabel $K_1$ at 150 750
 	\pinlabel \textcolor{green}{${-m-1}$} at 60 650
 	\pinlabel \textcolor{blue}{$\mu_2$} at  75 550
 	\pinlabel \textcolor{red}{$\mu_1$} at 400 650
        \pinlabel \textcolor{green}{$K_2$} at 220 600
        \pinlabel $-4$ at 370 750
        \pinlabel $(1)$ at 250 440
        \pinlabel $K_1$ at 630 750
        \pinlabel $-4$ at 840 750
        \pinlabel \textcolor{red}{$\mu_1$} at 900 650
        \pinlabel \textcolor{blue}{$\mu_2$} at 610 520
        \pinlabel $(2)$ at 740 410
        \pinlabel $K_1$ at 300 400
        \pinlabel $-4$ at 420 360
        \pinlabel \textcolor{green}{$-m-1$} at 460 300
        \pinlabel \textcolor{blue}{$\mu_2$} at 140 210
        \pinlabel \textcolor{red}{$\mu_1$} at 450 120
        \pinlabel $-m-1$ at 290 220
        \pinlabel $(3)$ at 270 50
        \pinlabel \textcolor{green}{$K_2$} at 110 290
        \pinlabel $(4)$ at 760 40
        \pinlabel \textcolor{blue}{$\mu_2$} at 610 210
        \pinlabel \textcolor{green}{$-m-\frac{3}{4}$} at 930 290
        \pinlabel \textcolor{green}{$K_2$} at 600 290
        \pinlabel $-m-1$ at 780 220
        \pinlabel \textcolor{red}{$\mu_1$} at 910 95
        \tiny\pinlabel -$m$-1 at 661 572
 	\endlabellist
 	\includegraphics[scale=0.31]{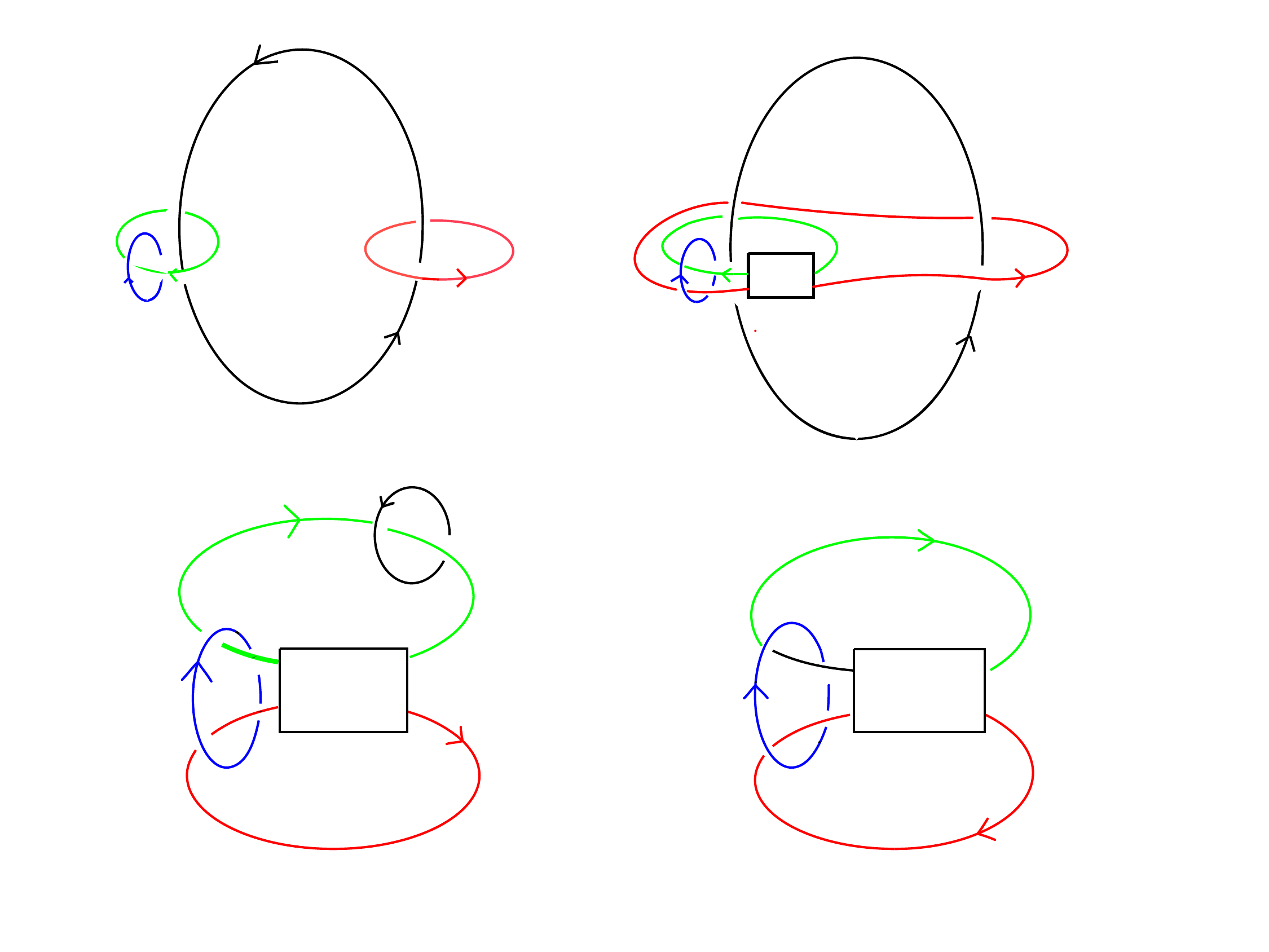}
\caption{A sequence of Kirby moves from the surgery diagram of $L(4m+3,4)$ along the positive Hopf link to its standard surgery diagram. (1) The surgery diagram of $L(4m+3,4)$ along positive Hopf link. (2) Sliding $\mu_1$ along $K_2$. (3) Isotopy. (4) Slam-dunk $K_1$ to get the final diagram.}
	\label{fig:Hopf_slam_dunk}
\end{figure}

For any integer $k\in\Z$ we can perform a Rolfsen twist on the standard surgery description to get again an unknot but with the surgery coefficient changed to $-\frac{4m+3}{4-4km-3k}$. The Rolfsen twist does not change the homology class of the meridian and thus
\begin{align*}
    \mu_k\longmapsto \mu.
\end{align*}
Finally, there are the \textit{dual} surgery descriptions (which we get by interchanging the roles of the two Heegaard tori in the above surgery descriptions). In the surgery picture we obtain these by slam dunking $K_2$ away in Figure~\ref{fig:Hopf}. This yields an unknot $U^*$ with surgery coefficient $-\frac{4m+3}{m+1}$. This dual slam dunk diffeomorphism sends $\mu_1$ to $\mu^*$ and thus
\begin{align*}
    \mu^*\longmapsto (m+1)\mu.
\end{align*}
By performing a $k$-fold Rolfsen twist on $U^*$ we get the surgery descriptions on $U^*$ with surgery coefficient $-\frac{4m+3}{m+1-4km-3k}$. The Rolfsen twist does not change the homology class of the meridian of $U^*$ and thus
\begin{align*}
    \mu^*_k\longmapsto (m+1)\mu.
\end{align*}
The classification of lens spaces implies that this is the complete list of $1$-component surgery descriptions of $L(4m+3,4)$.

It remains to express $\mu_T$ as a multiple of $\mu$. For that, we refer to Figure~\ref{fig:torus_Kirby_moves} in which a sequence of Kirby moves from the surgery diagram along the torus knot $T_{(2,-(2m+1)}$ to the dual surgery description is presented. In that figure, we observe that $\mu_T$ gets mapped to $2\mu^*$ and thus we obtain
\begin{align*}
    \mu_T\longmapsto2(m+1)\mu
\end{align*}
    as claimed.
\end{proof}
\begin{figure}[htbp] 
\centering
\def\svgwidth{0.99\columnwidth}
\labellist
 	\small\hair 2pt
 	\pinlabel \textcolor{red}{$-4m-3$} at 60 650
        \pinlabel \textcolor{red}{$\vdots$} at 130 570
        \pinlabel \textcolor{red}{$\vdots$} at 130 530
 	\pinlabel $-2m-1$ at  5 530
 	\pinlabel \textcolor{blue}{$\mu_T$} at 337 520
        \pinlabel \textcolor{red}{$T$} at 300 650
        \pinlabel $(1)$ at 170 360
        \pinlabel \textcolor{blue}{$\mu_T$} at 670 520
        \pinlabel $\frac{1}{m+1}$ at 530 470
        \pinlabel \textcolor{red}{$T$} at 640 650
        \pinlabel \textcolor{red}{$1$} at 430 650
        \pinlabel $(2)$ at 530 360
        \pinlabel $(3)$ at 900 360
        \pinlabel \textcolor{blue}{$\mu_T$} at 900 560
        \pinlabel $\frac{1}{m+1}$ at 1000 650
        \pinlabel \textcolor{red}{$T$} at 750 470
        \pinlabel \textcolor{red}{$1$} at 759 570
        \pinlabel $(4)$ at 400 10
        \pinlabel $\frac{1}{m+1}-4$ at 220 300
        \pinlabel \textcolor{blue}{$\mu_T$} at 370 200
        \pinlabel $(5)$ at 800 10
        \pinlabel $-\frac{4m+3}{m+1}$ at 800 300
        \pinlabel \textcolor{blue}{$\mu_T$} at 920 120
 	\endlabellist
\includegraphics[scale=0.3]{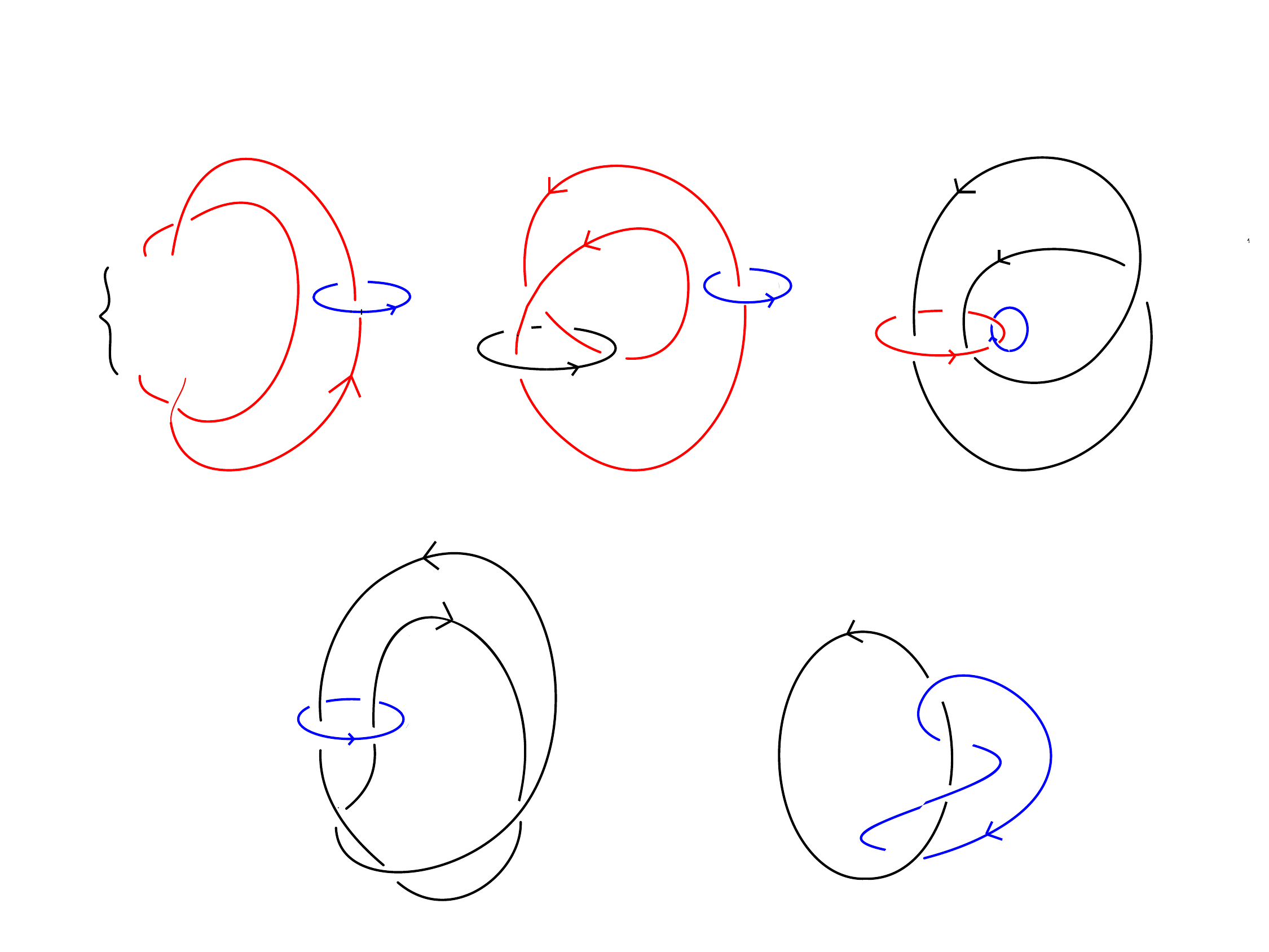}
\caption{A sequence of Kirby moves from from the surgery diagram of $L(4m+3,4)$ along the torus knot $T_{(2,-(2m+1))}$ to its dual surgery diagram along the unknot.(1) Surgery representation of $L(4m+3,4)$ along $T_{(2, -(2m+1))}$. (2) Result of $(m+1)$ Rolfsen twist. (3) Isotopy. (4) Blow down $T$. (5) Isotopy.}
	\label{fig:torus_Kirby_moves}
\end{figure}
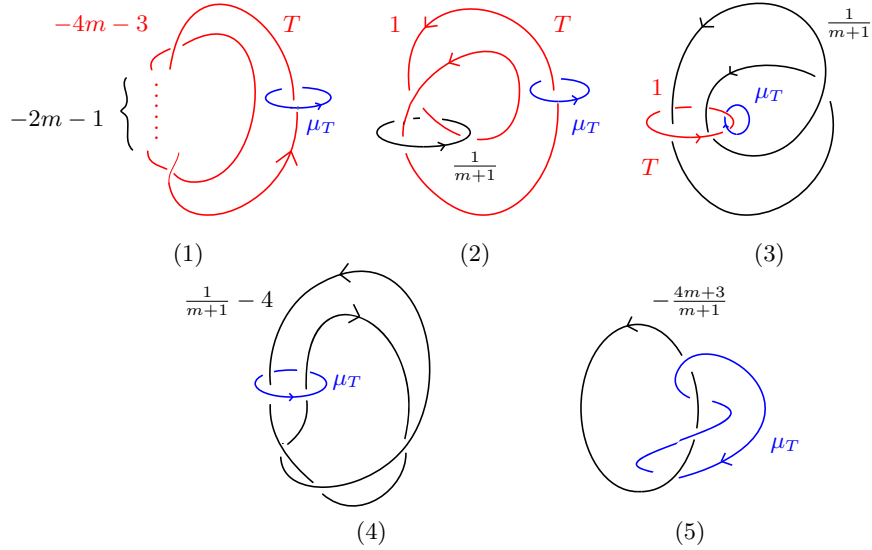

\subsection{Integral contact surgery numbers of the lens spaces \texorpdfstring{$L(4m+3,4)$}{L(4m+3,4)}}

\begin{proof}[Proof of Theorem~\ref{thm:lens_tight}]
By Honda's classification of tight contact structures on lens spaces~\cite{Honda_lens} any tight contact structure can be obtained by Legendrian surgery on a chain of Legendrian unknots. For a lens space of the form $L(4m+3,4)$ the complete list of tight contact structures is given by the surgery descriptions shown in Figure~\ref{fig:tight_lens}. In particular, it follows that any tight contact structure on such a lens space has contact surgery number $\cs_{\pm1}\leq2$. 

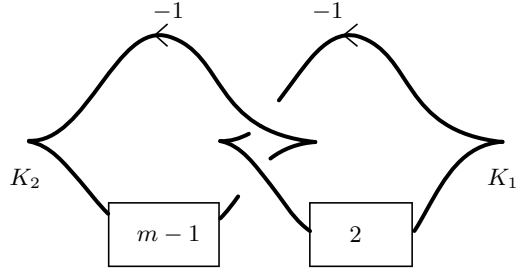
\begin{figure}[htbp] 
\centering
\def\svgwidth{0.4\columnwidth}
\labellist\small\hair 2pt
\pinlabel $m-1$ at 230 80
\pinlabel $2$  at 465 80
\pinlabel $-1$ at 230 360
\pinlabel $-1$ at 430 360
\pinlabel $K_1$ at 650 150
\pinlabel $K_2$ at 50 150
\endlabellist
\includegraphics[scale=0.3]{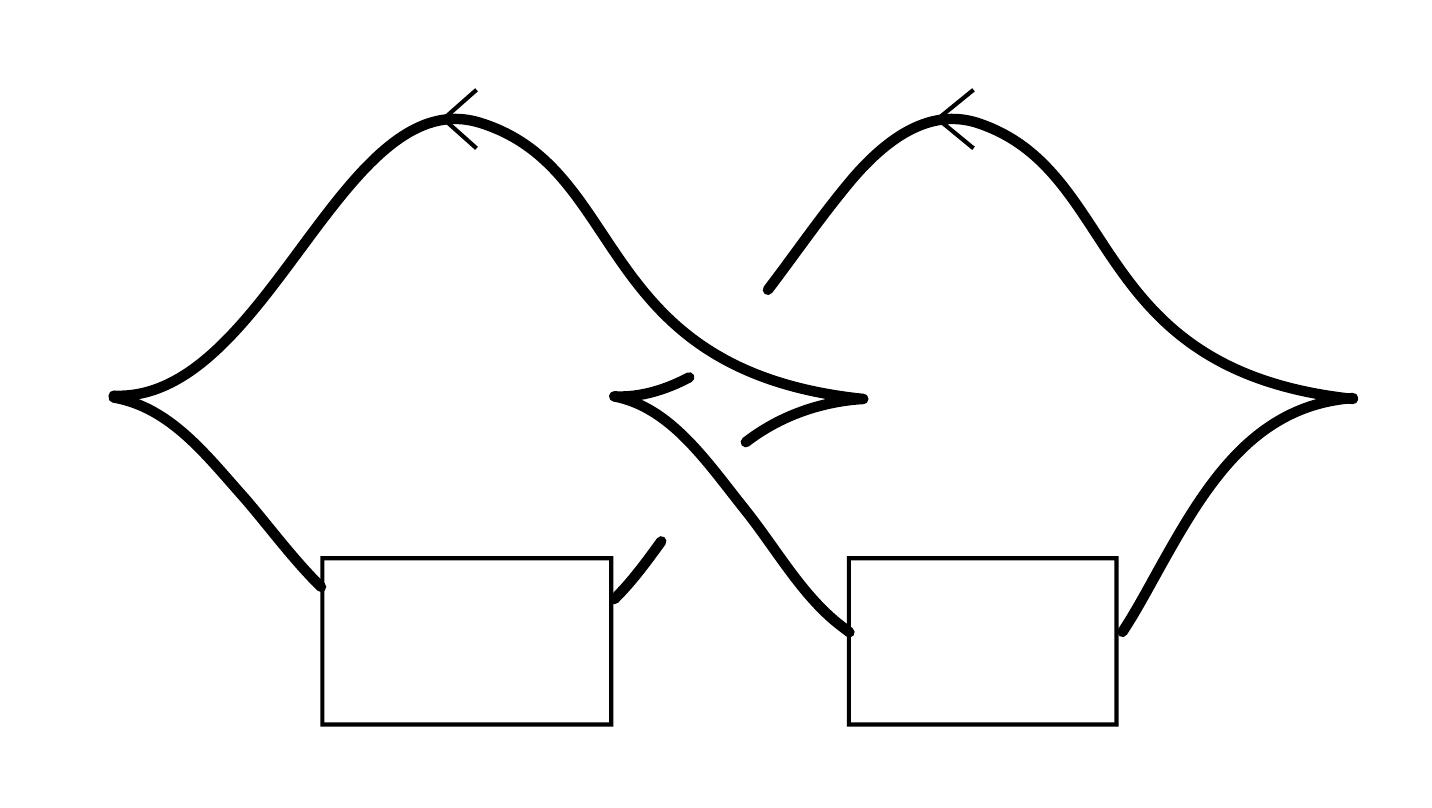}
\caption{Surgery diagrams of all tight contact structures on $L(4m+3,4)$. The integer-labeled boxes denote that many stabilizations.}
	\label{fig:tight_lens}
\end{figure}

Now let $\xi$ be a tight contact structure on $L(4m+3,4)$ that is contactomorphic to $K(n)$ for some Legendrian knot $K$ in $(S^3,\xist)$ and some $n\in\Z\setminus\{0\}$. Then Theorem~\ref{thm:surgery_diagrams_of_lens} implies that $K$ is a Legendrian realization of the torus knot $T_{(2,-(2m+1))}$ with topological surgery coefficient $-4m-3$. By~\cite{Etnyre_Honda_torus_knots}, torus knots are Legendrian simple and the torus knots at hand have $\tb\leq -4m-2$. There are exactly $m$ different Legendrian realizations of $T_{(2,-(2m+1))}$ with Thurston--Bennequin invariant $t=-4m-2$. These have rotation numbers $r\in\{1,3,5,\ldots,2m-1\}$ and are shown in Figure~\ref{fig:tight_lens_one_knot}. 
Legendrian surgery on any of these Legendrian realizations yields a tight contact structure on $L(4m+3,4)$. From Lemma~\ref{lem:d3} we see that the Poincar\'e dual of the Euler class is given by $\rot(K)\mu_T$. Using Theorem~\ref{thm:surgery_diagrams_of_lens} we can express the Euler class in the standard basis to see that the tight contact structures claimed in Theorem~\ref{thm:lens_tight} all have contact surgery numbers $\cs_{\pm1}=1$. 

To see that the other tight contact structures have $\cs_{\pm1}\geq2$, we observe that any other integer surgery on a Legendrian realization of $T_{(2,-(2m+1))}$ that yields a contact structure on $L(4m+3,4)$ is along a stabilized knot with a positive surgery coefficient and thus is overtwisted by~\cite{Lisca_Stipsicz_11_transverse}.

The same argument yields the result for the integer contact surgery numbers. 
\end{proof}

\begin{proof}[Proof of Corollary~\ref{thm:cor_lens}]
In the proof of Theorem~\ref{thm:lens_tight} we have constructed the complete list of integer contact surgery diagrams along a single Legendrian knot of tight contact structures on $L(4m+3,4)$. These are shown in Figure~\ref{fig:tight_lens_one_knot}. This immediately implies the corollary.
\end{proof}

\begin{figure}[htbp] 
\centering
\def\svgwidth{0.99\columnwidth}
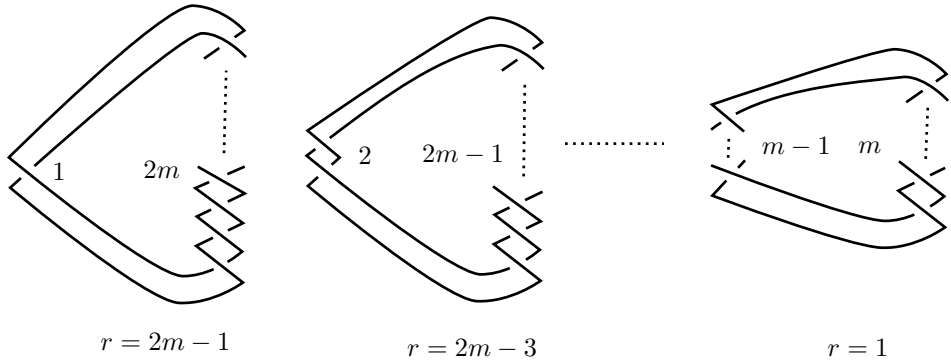
\caption{All Legendrian realizations of $T_{2,-(2m+1)}$ with maximal Thurston--Bennequin invariant (equal to $t=-4m-2$) and rotation numbers $r=1,3,\ldots,2m-1$. (The numbers indicate the number of crossings.) Performing contact $(-1)$-surgery along these knots gives the complete list of all integer contact surgery diagrams along a single Legendrian knot yielding a tight contact structure on the lens space $L(4m+3,4)$.}
	\label{fig:tight_lens_one_knot}
\end{figure}

Next, we study the overtwisted contact structures on $L(4m+3,4)$. 

\begin{proof}[Proof of Theorem~\ref{thm:lens_overtwisted}]
Let $\xi$ be a contact structure on $L(4m+3,4)$ that is contactomorphic to $K(n)$ for some Legendrian knot $K$ in $(S^3,\xist)$ and some $n\in\Z\setminus\{0\}$. The same argument as in the proof of Theorem~\ref{thm:lens_tight} shows that $K$ is a Legendrian realization of $T_{(2,-(2m+1))}$ with Thurston--Bennequin invariant $t\leq-4m-2$ and the topological surgery coefficient is $-4m-3$.

If $t=-4m-2$ the contact structure $\xi$ is tight and was already handled in the proof of Theorem~\ref{thm:lens_tight}. We do not need to consider the case when $t=-4m-3$ since then the contact surgery coefficient would vanish.

We recall that any other integer surgery on a Legendrian realization of $T_{(2,-(2m+1))}$ that yields a contact structure on $L(4m+3,4)$ is along a stabilized knot with a positive surgery coefficient and thus is overtwisted by~\cite{Lisca_Stipsicz_11_transverse}. 

If $t=-4m-4$, the contact surgery coefficient is $+1$. Thus the generalized linking matrix is $Q=(-4m-3)$ and we can readily apply Lemma~\ref{lem:d3} to compute that $\e(\xi)=r\mu_T=2r(m+1)\mu$ and
\begin{equation*}
    \de_3(\xi)=\frac{1}{4}\left( 5-\frac{r^2}{4m+3}\right),
\end{equation*}
where $r$ denotes the rotation number of $K$. From~\cite{Etnyre_Honda_torus_knots} we see that the possible range of rotation numbers of Legendrian realizations of $K$ with $t=-4m-4$ is given by $r=2k+1$, for $k\in\{-1,0,1,2,\ldots,m\}$. Plugging this in yields the formulas claimed in $(1)$.

Next, we assume that $t\leq-4m-5$. Then the contact surgery coefficient is $n=-4m-3-t\geq 2$. Thus we deduce from the transformation lemma~\ref{lem:Kirby} that
\begin{equation}\label{surgery_desc}
    K(n)\cong K(+1){\def\svgwidth{1,6ex}\,\,
\begingroup%
  \makeatletter%
  \providecommand\color[2][]{%
    \errmessage{(Inkscape) Color is used for the text in Inkscape, but the package 'color.sty' is not loaded}%
    \renewcommand\color[2][]{}%
  }%
  \providecommand\transparent[1]{%
    \errmessage{(Inkscape) Transparency is used (non-zero) for the text in Inkscape, but the package 'transparent.sty' is not loaded}%
    \renewcommand\transparent[1]{}%
  }%
  \providecommand\rotatebox[2]{#2}%
  \ifx\svgwidth\undefined%
    \setlength{\unitlength}{49.54732089bp}%
    \ifx\svgscale\undefined%
      \relax%
    \else%
      \setlength{\unitlength}{\unitlength * \real{\svgscale}}%
    \fi%
  \else%
    \setlength{\unitlength}{\svgwidth}%
  \fi%
  \global\let\svgwidth\undefined%
  \global\let\svgscale\undefined%
  \makeatother%
  \begin{picture}(1,1.00059637)%
    \put(0,0){\includegraphics[width=\unitlength,page=1]{PushOff.pdf}}%
  \end{picture}%
\endgroup%
\,\,} K_1\left(-\frac{1}{-4m-4-t}\right).
\end{equation}

From this surgery description, we compute the generalized linking matrix to be
\begin{equation*}
    Q=\begin{pmatrix}
1+t & -t(4m+4+t) \\
t & 4m+3-t^2-t(4m+3) \\
\end{pmatrix}.
\end{equation*}
We first compute the $\de_3$-invariants. For that we observe the signature of $Q$ to be $-2$ and solve $Q\mathbf{b}=(r,r\pm1)$ (where $r$ denotes again the rotation number of $K$) from which we get 
\begin{equation*}
    \de_3=-m-\frac{t\pm r}{2}-\left(\frac{(t\pm r)^2+2(t\pm r)+1}{4(4m+3)}\right).
\end{equation*}
From~\cite{Etnyre_Honda_torus_knots} we conclude that $t\pm r=2k+1$, for $k\leq-m-2$, which directly yields the claimed formula for the $\de_3$-invariants in $(2)$. 

It remains to compute the Euler classes. For that, we write $\mu_1$ for the meridian of $K$ and $\mu_2$ for the meridian of $K_1$ in the surgery description~\eqref{surgery_desc}. Then Lemma~\ref{lem:d3} implies that the Euler class is given by 
\begin{equation*}
    \operatorname{PD}(\e)=r\mu_1 +(r\pm1)\mu_2
\end{equation*}
Now from Lemma~\ref{lem:meridian} it follows
$\mu_1=-t\mu_T$ and $\mu_2=(1+t)\mu_T$. Putting these together we have the claimed result.
\end{proof}

\begin{figure}[htbp] 
\centering
\def\svgwidth{0.99\columnwidth}
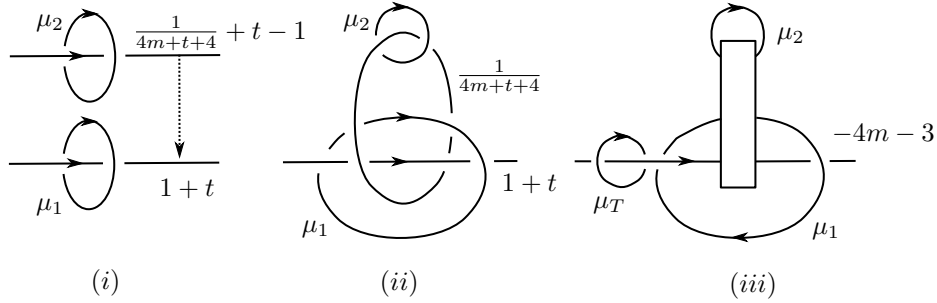
\caption{The diagram $(i)$ shows a topological realization of the surgery description~\eqref{surgery_desc} (where all surgery coefficients are measured with respect to the Seifert framing) together with the meridians $\mu_1$ and $\mu_2$. We obtain $(ii)$ from $(i)$ by the handle slide indicated with the dotted line. An $-(4m+t-4)$-fold Rolfsen twist yields the standard surgery diagram along $T_{2,-(2m+1)}$ shown in $(iii)$ where we can read-off the homology classes of $\mu_1$ and $\mu_2$ in terms of $\mu_T$. The box denotes $4m+t+4$ left-handed full-twists.}
	\label{fig:Kirby}
\end{figure}

\subsection{Rational contact surgery numbers of the lens spaces \texorpdfstring{$L(4m+3,4$)}{L(4m+3,4}}
For the rational contact surgery numbers the same strategy works to classify all contact structures on $L(4m+3,4)$ that have $\cs=1$. However, here we have to consider several infinite families of contact surgery diagrams given by Legendrian realizations of the smooth surgery diagrams given in Theorem~\ref{thm:surgery_diagrams_of_lens}. The main result is as follows.

\begin{theorem}
\label{thm:lens_rational}
    An overtwisted contact structure on $L(4m+3,4)$ has contact surgery number $\cs_{\Q}=1$ if and only if its tuple $(e,\de_3)$ of Euler class and $\de_3$-invariant appears in the lists of Theorem~\ref{thm:lens_overtwisted}~$(1)$ or~$(2)$ or in Table~\ref{tab:horrible}.
\end{theorem}

{ \tiny
    \begin{table}[htbp] 
    \caption{Overtwisted contact structures on $L(4m+3,4)$. }
	\label{tab:horrible}
	\begin{tabular}{r|l}
		\toprule
	  & $e=\pm(2n+1\pm 1)+(2-t)(1+t)(1-t)^{-t-m-3}$ \\
      (3) &  $d_3=\frac{1}{2}-\frac{m(7+x^2)+(19+x^2)}{4(4m+3)}-\frac{m^2+1+n(n+8+4m)\mp 4(1+n)}{4m+3}  $\\
     &  for $x\in\{1,3\},\ n\leq -1,\ t< -2-m$\\
        \midrule
     &  $e=(-1-m)y+(-m-1+2l)$ \\
    (4)  & $d_3=\frac{1}{2}-\frac{(3+m)(m+1-2l)^2+(m+1-2l)(2+m)y}{4m+3}+\frac{2(1+m-2l)y-(1+m)y^2}{4(4m+3)}   $\\
      &  for $y\in\{0,\pm 2,\pm 4\},\ l\in\{0,1,\ldots, m+1\}$\\
        \midrule
      &  $e=(-m+2l)$\\
   (5)   &  $d_3=\frac{1}{2}-\frac{(m-2l)^2}{4m+3}$\\
     &   for $l\in\{0,1,\dots m\}$\\
          \midrule
       &  $e=\pm(1+2n)\pm 1+(2-t)(1+t)(1-t)^{-t-3}(\pm 1 +(2-t)(1-t)^{-k-2}((-m+2l)+r_4(1-t+m)))$\\ 
   (6)   &   $d_3=\frac{1}{2}+kn(n+1)-(1+2n)-\frac{(1+2n)^2+(-m+2l)^2}{4m+3}+\frac{2(-m+2l)r_4+r_4^2(m+1)\pm2(1+2n)(4(-m+2l)+r_4)}{4(4m+3)}$\\
     &  for $k\leq-2,\ r_4\in \{0,\pm 2\},\ l\in\{0,1,2,\dots m\}, \ t\leq -2, \ n\leq -1$\\
        \midrule
      &  $e=\pm(1+2n)\pm 1+(2-t)(1+t)(1-t)^{-t-3}(\mp 1 +(2-t)(1-t)^{-k-2}((-m+2l)+r_4(1-t+m)))$\\ 
   (7)   &  $d_3=\frac{1}{2}+k(n+1)(n+2)-\frac{5-4m-6(-m+2l)+n(8-4(-m+2l)-r_4)+(1+2n)^2+(-m+2l)^2}{4m+3} -\frac{r_4^2(m+1)+2(-m+2l)r_4-2r_4}{4(4m+3)} $\\
      &  for  $k\leq -2,\ r_4\in\{0,\pm 2\},\ l\in\{0,1,\dots m\},\ t\leq -2, \ n\leq -1$\\
        \midrule
      &  $e=\pm 1-r_1-3^{-k-2}5((m-2l)+(3+m)r_3)$\\
   (8)    & $d_3=\frac{1}{2}+\frac{kr_1(r_1\mp 2)}{4}+\frac{2(m-2l)(\pm 1-r_1)-(m-2l)^2\mp 2mr_1+4m+2}{4m+3}+\frac{r_1^2(4m-1)-2r_3(r_1+m-2l)-r_3^2(1+m)\pm 2(r_1+r_3)}{4(4m+3)}$\\
       & for $k\leq -2,\ r_1\in\{0,\pm 2\},\ r_3\in\{0,\pm 2\},\ l\in\{0,1,\ldots, m\}$\\   
        \midrule
       &  $e=\pm(2+2n)+(1-t)^{-t-3}(2-t)(1+t)((-m-1+2l)+(2-t+m)r_3)$\\
    (9)   &  $d_3=\frac{1}{2}-\frac{n(n+1)(4m+7)+2n(5+4m)\mp (n+1)(4(-m-1+2l)+r_3)+(-m-1+2l)^2}{4m+3} +\frac{2(-m-1+2l)r_3+r_3^2(m+1)}{4(4m+3)}$\\
       & for $r_3\in\{0,\pm 2\},\ l\in\{0,1,\dots m+1\},\ t<-2,\ n\leq -1$\\
        \midrule
       & $e=\pm(5+2m)+(3+2m)(m+2-2l)+(m-1)r_2$\\ 
    (10)   & $d_3=\frac{1}{2}+\frac{2r_2((m+2-2l)\mp 1)+r_2^2(1+m)}{4(4m+3)}+\frac{2+4m\pm 2r_1-(m+2-2l)^2}{4m+3}$\\
       & for $r_2\in\{0,\pm 2\},\ l\in\{0,1,\ldots, m+2\}$\\
        \midrule
      &  $e=0$\\
    (11)  &  $d_3=\frac{1}{2}-\frac{-2(2m+1)+(-m+2l)^2}{4m+3}-\frac{r_3^2(m+1)+2(-m+2l)r_3\pm2(r_3+4(-m+2l))}{4(4m+3)}$\\
       &  for $r_3\in\{0,\pm 2\},\,\ l\in\{0,1,\dots m\}$\\
          \midrule
       &  $e=0$\\
     (12)  &  $d_3=\frac{1}{2}+1-\frac{(m+1-2l)^2}{4m+3}-\frac{r_2^2(1+m)+2r_2(m+1-2l)}{4(4m+3)}$\\
       & for $r_2\in\{0,\pm 2\},\ \ l\in\{0,1,\ldots, m+1\}$\\
        \midrule
      &  $e=0$\\
    (13)  &  $d_3=\frac{1}{2}-\frac{m^2}{4m+3}+\frac{(1+m)(9-r^2)} {4(4m+3)}$\\
       & for $r \in\{1,3\}$\\
        \midrule
      &  $e=\pm(2n+1)\pm 1+(1-t)^{-t-2}(2-t)(1+t)(-k-1+2l_1-(t-k)(-m+1+2l_2-(t-m)r_4)))$\\
   (14)   &  $d_3=\frac{1}{2}-(1+n)+k(1+n)^2-\frac{4(n+1)^2\pm(n+1)(-k+1+2l_1+4(-m+1+2l_2)+r_4)}{4m+3} -\frac{(1+m){r_4}^2+2(-m+1+2l_2)r_4+4(-m+1+2l_2)^2}{4(4m+3)}   $\\
      &  for $r_4 \in\{0,\pm 2\},\ l_1\in\{0,1,\dots, k-1\},\ l_2\in\{0,1,\dots, m-1\}, t\leq -2, \ n< -1,k>0.$\\
        \midrule
      &  $e=0$\\
    (15)  &  $d_3=\frac{1}{2}+\frac{3(m+1)+l-(-m+1+2l)^2}{4m+3}+\frac{(1+m)(1-{r_2}^2)}{4(4m+3)}$\\ 
       & for $r_2\in\{\pm 1, \pm 3\}\ , \ l\in\{0,1,\dots m-1\}$\\
        \midrule
       & $e=0$\\
     (16)  & $d_3=\frac{1}{2}+\frac{-5-8m+r_2(l+1)+l+(-m+1+2l)^2}{4m+3}+\frac{(1+m)({r_2}^2-1)}{4(4m+3)}$\\
       & for $r_2\in\{\pm 1, \pm 3\}, \,l\in\{0,1,\dots m-1\}$\\
        \midrule
      &  $e=0$\\
    (17)  &  $d_3=\frac{1}{2}+1-\frac{(m-1-2l)^2}{4m+3}-\frac{r_1^2(1+m)+2r_1(m-1-2l)}{4(4m+3)}$\\
      &  for $r_1\in\{0,\pm 2,\pm 4\},\ \ l\in\{0,1,\ldots, m-1\}$\\
        \midrule
      &  $e=(m+1)((2n+2)+(1-t)^{-t-3}(1+t)(2-t)+(1-t)^{-t-k-5}(2-t)^2(1+t)(r_3+(4-t)(-m+1+2l)))$\\
    (18)  &  $d_3=\frac{1}{2}+kn(n+1)- \frac{n(m+1)(n+1)+n(6-r_3+m-1-2l)}{4m+3}  +\frac{(-m+1+2l)^2+4m(1+2n)+(m-l)+3-r_3(l-m(n+1))}{4m+3} -\frac{(m+1)({r_3}^2-1)}{4(4m+3)} $\\ 
      &  for $r_3\in\{\pm 1, \pm 3 \},\ n \leq -1,\  k\leq -2, \ t\leq -3, \ l\in\{0,1,\dots, m-1\}$\\
        \midrule
       &  $e=(m+1)(-(2n+2)-(1-t)^{-t-3}(1+t)(2-t)+(1-t)^{-t-k-5}(2-t)^2(1+t)(r_3+(4-t)(-m+1+2l)))$\\
    (19)   &  $d_3=\frac{1}{2}+kn(n+1)- \frac{n(m+1)(n+1)+n(6+r_3-m+1+2l)}{4m+3}+ \frac{(-m+1+2l)^2+4m(1+2n)+l+4+r_3(l+mn)}{4m+3}  -\frac{(m+1)({r_3}^2-1)}{4(4m+3)}$\\ 
       &  for $r_3\in\{\pm 1, \pm 3 \},\ n \leq -1,\ t\leq -3,\  k\leq -2, l\in\{0,1,\dots, m-1\}$\\
         \midrule 
       &  $e=(m+1)(\pm(2n+2)\mp(1-t)^{-t-3}(1+t)(2-t)+(1-t)^{-t-k-5}(2-t)^2(1+t)(r_3+(4-t)(-m+1+2l)))$\\ 
     (20)  &  $d_3= \frac{1}{2}+k(n+1)(n+2)- \frac{(m+1)(n^2+3n-r_3(n+1))-2(m-1-2l)+(m-1-2l)^2}{4m+3}-\frac{(-2m-r_3(l+1)+l)}{4m+3} -\frac{(m+1)({r_3}^2-1)}{4(4m+3)}$\\ 
      &   for $r_3\in\{\pm 1, \pm 3 \},\ n \leq -1,\ t\leq -3,\  k\leq -2,\  l\in\{0,1,\dots, m-1\}$\\
	\bottomrule
\end{tabular}
\end{table}
}

{ \tiny
    \begin{table}[htbp] 
    \caption{Table~\ref{tab:horrible} continued.}
	\label{tab:horrible2}
	\begin{tabular}{r|l}
		\toprule
        &  $e=(m+1)(1 -r_1-5(3)^{-k-2}r_2-3^{-k-2}30(m-1-2l))$ \\
     (21)   &  $d_3=\frac{1}{2}+\frac{kr_1(r_1-2)}{4}-\frac{2(1+m)r_1r_2+2r_1(m-1-2l)+(r_2^2-1)(1+m)+(3m+2)(r_1^2-2r_1)}{4(4m+3)} +\frac{(4m+2)-(r_2+1)l+(m-1-2l)^2+r_1}{4m+3}$\\
        &  for $r_1\in\{0,\pm 2\}, \ r_2\in\{\pm 1, \pm 3\},\ \ l\in\{0,1,\ldots,  m-1\}, \ k\leq -2$\\
         \midrule
         &  $e=(m+1)(-1 -r_1-5(3)^{-k-2}r_2-3^{-k-2}30(m-1-2l))$\\ 
       (22)  &  $d_3=\frac{1}{2}+\frac{kr_1(r_1+2)}{4}-\frac{2(1+m)r_1r_2+2r_1(m-1-2l)+(r_2^2-1)(1+m)+(3m+2)(r_1^2+2r_1)}{4(4m+3)} +\frac{3(m+1)-l-r_2(m-l)-(m-1-2l)^2+r_1}{4m+3}$ \\
        &   for $r_1\in\{0,\pm 2\}, \ r_2\in\{\pm 1, \pm 3\},\ \ l\in\{0,1,\ldots,  m-1\}, \ k\leq -2$\\
         \midrule
       &  $e=(m+1)(\pm(2n+2)+(2-t)(1+t)(1-t)^{-t-3}((5-t)(-m+1+2l)+r_2))$\\
     (23)  &  $d_3=\frac{1}{2}-(1+n+n^2)-\frac{(m+1)(1+n+n^2)}{4m+3}+\frac{7n\mp(1+n)(r_2(m+1)+(-m+1+2l))+{r_3}^2+9mn}{4m+3}-\frac{2r_2(-m+1+2l)+{r_2}^2(1+m)}{4(4m+3)} $\\ 
       &  for $t\leq -3,\  n\leq -1, \ r_2\in\{0,\pm 2,\pm 4\},\ l \in\{0,1,\dots, m-1\}$\\
        \midrule
      &   $e=(m+1)(1-r-27(-m+1+2l))$\\
    (24)  &   $d_3=\frac{1}{2}+\frac{r(m-l)+((m-1-2l)^2-l-r)}{4m+3}+\frac{(1+m)(r^2-1)}{4(4m+3)})$  \\
      &   for $r\in\{\pm 1,\pm 3,\pm 5\},\ l \in\{0,1,\dots, m-1\}$\\
        \midrule
      &   $e=(m+1)(-1-r-27(-m+1+2l))$\\
     (25) &   $d_3=\frac{1}{2}-\frac{r(l+1)+((m-1-2l)^2-m+r)}{4m+3}+\frac{(1+m)(r^2-1)}{4(4m+3)})$  \\
      &   for $r\in\{\pm 1,\pm 3,\pm 5\},\ l \in\{0,1,\dots, m-1\}$\\
        \midrule
       &  $e=0$\\
   (26)   &   $d_3=\frac{1}{2}-\frac{(-m+2l)^2}{3+4m} $\\
      &   for $l\in\{0,1,\dots ,m\}$\\  
        \midrule
       &  $e=(m+1)(\pm(2n+2)+(1+t)(2-t)(1-t)^{-t-2}((-k+1+2l_1)-(t+k)((3+t)(-m+1+2l_2)+r_2))$\\
    (27)   &  $d_3=\frac{1}{2}+k(1+n)^2\mp(1+n)(-k+1+2l_1)\mp\frac{(n+1)(r_2(m+1)+(-m+1+2l_2))}{4m+3}-\frac{(-m+1+2l_2)^2}{4m+3}$\\ 
       &  $\,\,\,\,\,\,\,\,\,\,\,\,\,\,-\frac{2r_2(-m+1+2l_2)\pm(1+m){r_2}^2}{4(4m+3)} -\frac{(n+1)(5m+mn+n+4)}{4m+3}$\\
       &  for $t< -1,\ \  r_2\in\{0,\pm 2\},\ l_1\in\{0,1,\dots, k-1\},\ l_2\in\{0,1,\dots, m-1\},\, k\geq 1, n<-1$\\
        \midrule
       &  $e=(m+1)(\pm(2n+2)-(t-2)(1+t)(1-t)^{-t-6}(-m+2l))$\\
     (28)  &  $d_3=\frac{1}{2}\pm\frac{(-m+2l)(1+n)}{4m+3} -\frac{(-m+2l)^2-(m+1)(n^2+5n+2)+mn+3m+2}{4m+3}$\\
       &  for $t\leq -6,\  n\leq -1,\,\ l\in\{0,1,\dots, m\}$\\
        \midrule
      &   $e=(m+1)(r-4(m+1-2l))$\\
    (29)  &   $d_3=\frac{1}{2}+1-\frac{r^2(1+m)-2r(m+1-2l)}{4(4m+3)}-\frac{(m+1-2l)^2}{4m+3}$\\
       &  for $r\in\{\pm 4, \pm 2, 0\},\,l\in\{0,1,\dots, m+1\}$\\
        \midrule
      &   $e=(m+1)r$\\
    (30)  &   $d_3=\frac{1}{2}+\frac{(m+1)(9-r^2)}{4(4m+3)} -\frac{m^2}{4m+3}$\\
       &  for $r\in\{1,3\}$\\
		\bottomrule
\end{tabular}
\end{table}
}

\begin{proof}
Let $K$ be a Legendrian knot such that contact $r_c$-surgery, $r_c\in\Q$, on $K$ yields an overtwisted contact structure on $L(4m+3,4)$. The case that $r_c$ is an integer was already discussed in Theorem~\ref{thm:lens_overtwisted}. Thus here we concentrate on the case that $r_c\in\Q\setminus\Z$.
From Theorem~\ref{thm:surgery_diagrams_of_lens}, we know that $K$ is a Legendrian realization of the unknot and that the topological surgery coefficient is $-\frac{4m+3}{4-4km-3k}$ or $-\frac{4m+3}{m+1-4km-3k}$ for $k\in\Z$.

Now the strategy is the same as in the proof of Theorem~\ref{thm:lens_overtwisted}. We will enumerate all possible contact surgery diagrams and compute their homotopical invariants. For that we convert the topological surgery coefficient into a contact surgery coefficient. We denote by $t$ the Thurston--Bennequin invariant of $K$. Then it follows that
\begin{align*}
    r_c=\frac{-(4m+3)(1-kt)-4t}{4-k(4m+3)} \text{ or }
    r_c=\frac{-(4m+3)(1-kt)-t(m+1)}{m+1-k(4m+3)}.
\end{align*}

Next, we consider the different possibilities for $k\in\Z$, $t\leq-1$, possible rotation numbers of $K$, and the possible stabilizations in the transformation lemma. In all of these cases we first use Theorem~\ref{thm:unknot_surgery_classification} to sort out the tight contact structures (these all have $\cs=1$ by Honda's classification~\cite{Honda_lens}) and for the remaining overtwisted contact structures we compute $\e$ and $\de_3$ with the methods from Section~\ref{sec:prelim}.

We start with the simplest case, where we choose $k=0$ and thus the contact surgery coefficient simplifies to $r_c=-m-t-\frac{3}{4}$. We need to consider three subcases here. In Table~\ref{tab:horrible}, these correspond to Families (3), (4) and (5).
    
     \noindent\textbf{Family (5):}   For $t+m=-1$ this corresponds to a contact $(1/4)$-surgery, which is overtwisted by Theorem~\ref{thm:unknot_surgery_classification}. To compute $\de_3$ we see that the generalized linking matrix is $Q=(-3-4m)$ with signature $-1$. This yields
    $$\de_3=\frac{1}{2}-\frac{r^2}{3+4m},$$ with the rotation number $r=\rot(K)=-m+2l$ for $l\in\{0,1,\cdots, m\}$. The Euler class is simply given by $r$. 

    \noindent\textbf{Family (3):} For $t+m\leq -3$ we see that $2<r_c\leq-t$ and thus it yields overtwisted manifolds by Theorem~\ref{thm:unknot_surgery_classification}. Applying the transformation lemma, we get
\begin{align*}
K\left(-\frac{4(t+m)+3}{4}\right)=& K(+1){\def\svgwidth{1,6ex}\,\,
\begingroup%
  \makeatletter%
  \providecommand\color[2][]{%
    \errmessage{(Inkscape) Color is used for the text in Inkscape, but the package 'color.sty' is not loaded}%
    \renewcommand\color[2][]{}%
  }%
  \providecommand\transparent[1]{%
    \errmessage{(Inkscape) Transparency is used (non-zero) for the text in Inkscape, but the package 'transparent.sty' is not loaded}%
    \renewcommand\transparent[1]{}%
  }%
  \providecommand\rotatebox[2]{#2}%
  \ifx\svgwidth\undefined%
    \setlength{\unitlength}{49.54732089bp}%
    \ifx\svgscale\undefined%
      \relax%
    \else%
      \setlength{\unitlength}{\unitlength * \real{\svgscale}}%
    \fi%
  \else%
    \setlength{\unitlength}{\svgwidth}%
  \fi%
  \global\let\svgwidth\undefined%
  \global\let\svgscale\undefined%
  \makeatother%
  \begin{picture}(1,1.00059637)%
    \put(0,0){\includegraphics[width=\unitlength,page=1]{PushOff.pdf}}%
  \end{picture}%
\endgroup%
\,\,} K\left(\frac{-4(t+m)-3}{4+4(t+m)+3}\right)\\
=&K(+1){\def\svgwidth{1,6ex}\,\,
\begingroup%
  \makeatletter%
  \providecommand\color[2][]{%
    \errmessage{(Inkscape) Color is used for the text in Inkscape, but the package 'color.sty' is not loaded}%
    \renewcommand\color[2][]{}%
  }%
  \providecommand\transparent[1]{%
    \errmessage{(Inkscape) Transparency is used (non-zero) for the text in Inkscape, but the package 'transparent.sty' is not loaded}%
    \renewcommand\transparent[1]{}%
  }%
  \providecommand\rotatebox[2]{#2}%
  \ifx\svgwidth\undefined%
    \setlength{\unitlength}{49.54732089bp}%
    \ifx\svgscale\undefined%
      \relax%
    \else%
      \setlength{\unitlength}{\unitlength * \real{\svgscale}}%
    \fi%
  \else%
    \setlength{\unitlength}{\svgwidth}%
  \fi%
  \global\let\svgwidth\undefined%
  \global\let\svgscale\undefined%
  \makeatother%
  \begin{picture}(1,1.00059637)%
    \put(0,0){\includegraphics[width=\unitlength,page=1]{PushOff.pdf}}%
  \end{picture}%
\endgroup%
\,\,} K_1\left(-\frac{1}{-t-m-2}\right){\def\svgwidth{1,6ex}\,\,
\begingroup%
  \makeatletter%
  \providecommand\color[2][]{%
    \errmessage{(Inkscape) Color is used for the text in Inkscape, but the package 'color.sty' is not loaded}%
    \renewcommand\color[2][]{}%
  }%
  \providecommand\transparent[1]{%
    \errmessage{(Inkscape) Transparency is used (non-zero) for the text in Inkscape, but the package 'transparent.sty' is not loaded}%
    \renewcommand\transparent[1]{}%
  }%
  \providecommand\rotatebox[2]{#2}%
  \ifx\svgwidth\undefined%
    \setlength{\unitlength}{49.54732089bp}%
    \ifx\svgscale\undefined%
      \relax%
    \else%
      \setlength{\unitlength}{\unitlength * \real{\svgscale}}%
    \fi%
  \else%
    \setlength{\unitlength}{\svgwidth}%
  \fi%
  \global\let\svgwidth\undefined%
  \global\let\svgscale\undefined%
  \makeatother%
  \begin{picture}(1,1.00059637)%
    \put(0,0){\includegraphics[width=\unitlength,page=1]{PushOff.pdf}}%
  \end{picture}%
\endgroup%
\,\,} K_{1,3}(-1).
\end{align*}
 The generalized linking matrix for this case is given by
    \begin{equation*}
    Q=\begin{pmatrix}
1+t & -t^2-tm-2t & t\\
t & -t^2-mt-t+m+1 & t-1\\
t& -t^2-t-tm+m+2+1&t-5
\end{pmatrix}.
\end{equation*}
To get the $\de_3$ invariants, we compute the signature of $Q$ to be $-3$ and solve $Q\mathbf{b}=(r,r\pm 1, r\pm 1+x)$ where $x\in\{\pm 1,\pm 3\}$. 
Note that, we are only considering the classification up to contactomorphism it suffices for us to consider $x\in\{1,3\}$ only. Now a straightforward calculation yields
 $$\de_3= \frac{1}{2}-\frac{m(7+x^2)+(19+x^2)}{4(4m+3)}-\frac{m^2+1+n(n+8+4m)\pm 4(1+n)}{4m+3} $$

where $n=\frac{t\pm r-1}{2}.$ 
For the Euler class, we use Theorem~\ref{thm:Euler_algorithm} to compute
    $$\e(\xi)=(r\pm(t+1))\mu+x(2-t)(1+t)(1-t)^{-t-m-3}\mu.$$
    Writing $t\pm r=2n+1$, we get the desired number.

\noindent\textbf{Family (4):} Note that when $t+m=-2$, we have a slightly different surgery diagram as in Family (3) and thus we have to compute the $\de_3$-invariant separately. Here the resulting contact structures are again overtwisted by Theorem~\ref{thm:unknot_surgery_classification}. In this case, the generalized linking matrix is  
 \begin{equation*}
    Q=\begin{pmatrix}
1+t &  -t\\
-t & t-5
\end{pmatrix}.
\end{equation*} with signature $-2.$ We solve $Q\mathbf{b}=({r},r+y)$ where $r$ is the rotation number of the first component and  $y\in\{4,2,0\}$. Solving this gives us 
$$\de_3=\frac{1}{2}-\frac{(3+m)r^2+(2+m)yr}{4m+3}+\frac{2ry-(1+m)y^2}{4(4m+3)}.$$
As $r\in\{\pm(-1-m),\pm (-3-m),\cdots, 0\}$, we can rewrite $r=m+1-2l$ where $l\in\{0,1,\cdots, (m+1)\}.$
To calculate the Euler class, we simply plug in the corresponding values in the formula from Theorem~\ref{thm:Euler_algorithm}. 

Finally, for $t+m\geq 0$, the contact surgery coefficient $r_c$ is negative and thus yield tight contact structures, which we do not consider here.

For $k\neq 0$, the calculation is much more involved. So, we just present a single example case here.

\noindent\textbf{Family (14):} We consider the first contact surgery coefficient with $k\geq 1$. Then $r_c$ is positive and from Theorem~\ref{thm:unknot_surgery_classification} it follows that these contact structures are overtwisted if and only if $t\neq-1$. Thus we assume $t<-1$. Applying the transformation lemma~\ref{lem:Kirby} we get
\begin{align*}
    &K\left(\frac{-(4m+3)(1-kt)-4t}{4-k(4m+3)}\right)=K(+1){\def\svgwidth{1,6ex}\,\,
\begingroup%
  \makeatletter%
  \providecommand\color[2][]{%
    \errmessage{(Inkscape) Color is used for the text in Inkscape, but the package 'color.sty' is not loaded}%
    \renewcommand\color[2][]{}%
  }%
  \providecommand\transparent[1]{%
    \errmessage{(Inkscape) Transparency is used (non-zero) for the text in Inkscape, but the package 'transparent.sty' is not loaded}%
    \renewcommand\transparent[1]{}%
  }%
  \providecommand\rotatebox[2]{#2}%
  \ifx\svgwidth\undefined%
    \setlength{\unitlength}{49.54732089bp}%
    \ifx\svgscale\undefined%
      \relax%
    \else%
      \setlength{\unitlength}{\unitlength * \real{\svgscale}}%
    \fi%
  \else%
    \setlength{\unitlength}{\svgwidth}%
  \fi%
  \global\let\svgwidth\undefined%
  \global\let\svgscale\undefined%
  \makeatother%
  \begin{picture}(1,1.00059637)%
    \put(0,0){\includegraphics[width=\unitlength,page=1]{PushOff.pdf}}%
  \end{picture}%
\endgroup%
\,\,} K\left(\frac{-((4m+3)(1-kt)+4t)}{(1-k-kt)(4m+3)+4(t+1)}\right)\\
    =&K(+1){\def\svgwidth{1,6ex}\,\,
\begingroup%
  \makeatletter%
  \providecommand\color[2][]{%
    \errmessage{(Inkscape) Color is used for the text in Inkscape, but the package 'color.sty' is not loaded}%
    \renewcommand\color[2][]{}%
  }%
  \providecommand\transparent[1]{%
    \errmessage{(Inkscape) Transparency is used (non-zero) for the text in Inkscape, but the package 'transparent.sty' is not loaded}%
    \renewcommand\transparent[1]{}%
  }%
  \providecommand\rotatebox[2]{#2}%
  \ifx\svgwidth\undefined%
    \setlength{\unitlength}{49.54732089bp}%
    \ifx\svgscale\undefined%
      \relax%
    \else%
      \setlength{\unitlength}{\unitlength * \real{\svgscale}}%
    \fi%
  \else%
    \setlength{\unitlength}{\svgwidth}%
  \fi%
  \global\let\svgwidth\undefined%
  \global\let\svgscale\undefined%
  \makeatother%
  \begin{picture}(1,1.00059637)%
    \put(0,0){\includegraphics[width=\unitlength,page=1]{PushOff.pdf}}%
  \end{picture}%
\endgroup%
\,\,} K_1\left(\frac{-1}{-t-1}\right){\def\svgwidth{1,6ex}\,\,
\begingroup%
  \makeatletter%
  \providecommand\color[2][]{%
    \errmessage{(Inkscape) Color is used for the text in Inkscape, but the package 'color.sty' is not loaded}%
    \renewcommand\color[2][]{}%
  }%
  \providecommand\transparent[1]{%
    \errmessage{(Inkscape) Transparency is used (non-zero) for the text in Inkscape, but the package 'transparent.sty' is not loaded}%
    \renewcommand\transparent[1]{}%
  }%
  \providecommand\rotatebox[2]{#2}%
  \ifx\svgwidth\undefined%
    \setlength{\unitlength}{49.54732089bp}%
    \ifx\svgscale\undefined%
      \relax%
    \else%
      \setlength{\unitlength}{\unitlength * \real{\svgscale}}%
    \fi%
  \else%
    \setlength{\unitlength}{\svgwidth}%
  \fi%
  \global\let\svgwidth\undefined%
  \global\let\svgscale\undefined%
  \makeatother%
  \begin{picture}(1,1.00059637)%
    \put(0,0){\includegraphics[width=\unitlength,page=1]{PushOff.pdf}}%
  \end{picture}%
\endgroup%
\,\,} K_{1,k-1}(-1) {\def\svgwidth{1,6ex}\,\,
\begingroup%
  \makeatletter%
  \providecommand\color[2][]{%
    \errmessage{(Inkscape) Color is used for the text in Inkscape, but the package 'color.sty' is not loaded}%
    \renewcommand\color[2][]{}%
  }%
  \providecommand\transparent[1]{%
    \errmessage{(Inkscape) Transparency is used (non-zero) for the text in Inkscape, but the package 'transparent.sty' is not loaded}%
    \renewcommand\transparent[1]{}%
  }%
  \providecommand\rotatebox[2]{#2}%
  \ifx\svgwidth\undefined%
    \setlength{\unitlength}{49.54732089bp}%
    \ifx\svgscale\undefined%
      \relax%
    \else%
      \setlength{\unitlength}{\unitlength * \real{\svgscale}}%
    \fi%
  \else%
    \setlength{\unitlength}{\svgwidth}%
  \fi%
  \global\let\svgwidth\undefined%
  \global\let\svgscale\undefined%
  \makeatother%
  \begin{picture}(1,1.00059637)%
    \put(0,0){\includegraphics[width=\unitlength,page=1]{PushOff.pdf}}%
  \end{picture}%
\endgroup%
\,\,} K_{1,k-1,m-1}(-1){\def\svgwidth{1,6ex}\,\,
\begingroup%
  \makeatletter%
  \providecommand\color[2][]{%
    \errmessage{(Inkscape) Color is used for the text in Inkscape, but the package 'color.sty' is not loaded}%
    \renewcommand\color[2][]{}%
  }%
  \providecommand\transparent[1]{%
    \errmessage{(Inkscape) Transparency is used (non-zero) for the text in Inkscape, but the package 'transparent.sty' is not loaded}%
    \renewcommand\transparent[1]{}%
  }%
  \providecommand\rotatebox[2]{#2}%
  \ifx\svgwidth\undefined%
    \setlength{\unitlength}{49.54732089bp}%
    \ifx\svgscale\undefined%
      \relax%
    \else%
      \setlength{\unitlength}{\unitlength * \real{\svgscale}}%
    \fi%
  \else%
    \setlength{\unitlength}{\svgwidth}%
  \fi%
  \global\let\svgwidth\undefined%
  \global\let\svgscale\undefined%
  \makeatother%
  \begin{picture}(1,1.00059637)%
    \put(0,0){\includegraphics[width=\unitlength,page=1]{PushOff.pdf}}%
  \end{picture}%
\endgroup%
\,\,} K_{1,k-1,m-1,2}(-1)
\end{align*}
    The generalized linking matrix is
    \begin{equation*}
       Q= \begin{pmatrix}
        
            t+1 &(-t-1)t &t&t&t\\
            t&-t^2&(t-1)&(t-1)&(t-1)\\
            t &-t^2+1&t-k-1&t-k&t-k\\
            t &-t^2+1 &t-k&t-k-m &t-k-m+1\\
            t &-t^2+1 &t-k &t-k-m+1 &t-k-m-2
        \end{pmatrix}
    \end{equation*}
     with signature $-3$. We solve 
     $$Q\mathbf{b}=(r, r\pm 1, r\pm 1+r_2, r\pm 1+r_2+r_3, r\pm 1+r_2+r_3+r_4) $$ where $r_i$ denote the rotation number of the corresponding components. Rewriting $r_2=-(k-1)+2l_1$, $r_3=-(m-1)+2l_2$, and $r_4\in\{0,2\}$ with $l_1\in\{0,1,\cdots, k-1\}$ and $l_2\in\{0,1,\cdots, m-1\}$. By writing $n=\frac{t\pm r-1}{2}< -1$ and plugging everything into the formulas from Lemma~\ref{lem:d3} and Theorem~\ref{thm:Euler_algorithm} we get the claimed values for $\de_3$ and $\e$.

     {\tiny{
\begin{table}[htbp]
\centering
 \caption{Surgery descriptions for all families from Theorem~\ref{thm:lens_rational}. Here $t$ represents the Thurston--Bennequin invariant of $K$ and $k$ denotes the number of Rolfsen twist. Note that, to get the overtwisted structures of (14) and (27) either one needs a mixed stabilization on the first component or if all the stabilizations are of the same sign the first pushoff must have a opposite single stabilization. Families (15)--(30) will give the dual representation of $L(4m+3,4).$ Note that, sometimes we have multiple families corresponding to some surgery picture. The different families correspond to different single stabilization of the first or the second component in the surgery diagram.} 
 \label{tab:surgery_representation}
	\begin{tabular}{c|c|c|c}
		\toprule
 (3)& $K(r)=K(+1){\def\svgwidth{1,6ex}\,\,
\begingroup%
  \makeatletter%
  \providecommand\color[2][]{%
    \errmessage{(Inkscape) Color is used for the text in Inkscape, but the package 'color.sty' is not loaded}%
    \renewcommand\color[2][]{}%
  }%
  \providecommand\transparent[1]{%
    \errmessage{(Inkscape) Transparency is used (non-zero) for the text in Inkscape, but the package 'transparent.sty' is not loaded}%
    \renewcommand\transparent[1]{}%
  }%
  \providecommand\rotatebox[2]{#2}%
  \ifx\svgwidth\undefined%
    \setlength{\unitlength}{49.54732089bp}%
    \ifx\svgscale\undefined%
      \relax%
    \else%
      \setlength{\unitlength}{\unitlength * \real{\svgscale}}%
    \fi%
  \else%
    \setlength{\unitlength}{\svgwidth}%
  \fi%
  \global\let\svgwidth\undefined%
  \global\let\svgscale\undefined%
  \makeatother%
  \begin{picture}(1,1.00059637)%
    \put(0,0){\includegraphics[width=\unitlength,page=1]{PushOff.pdf}}%
  \end{picture}%
\endgroup%
\,\,} K_1(\frac{-1}{-t-m-2}){\def\svgwidth{1,6ex}\,\,
\begingroup%
  \makeatletter%
  \providecommand\color[2][]{%
    \errmessage{(Inkscape) Color is used for the text in Inkscape, but the package 'color.sty' is not loaded}%
    \renewcommand\color[2][]{}%
  }%
  \providecommand\transparent[1]{%
    \errmessage{(Inkscape) Transparency is used (non-zero) for the text in Inkscape, but the package 'transparent.sty' is not loaded}%
    \renewcommand\transparent[1]{}%
  }%
  \providecommand\rotatebox[2]{#2}%
  \ifx\svgwidth\undefined%
    \setlength{\unitlength}{49.54732089bp}%
    \ifx\svgscale\undefined%
      \relax%
    \else%
      \setlength{\unitlength}{\unitlength * \real{\svgscale}}%
    \fi%
  \else%
    \setlength{\unitlength}{\svgwidth}%
  \fi%
  \global\let\svgwidth\undefined%
  \global\let\svgscale\undefined%
  \makeatother%
  \begin{picture}(1,1.00059637)%
    \put(0,0){\includegraphics[width=\unitlength,page=1]{PushOff.pdf}}%
  \end{picture}%
\endgroup%
\,\,} K_{1,3}(-1)$ & $t<-2-m$ &$k=0$ \\
  \midrule
  (4) & $K(r)=K(+1){\def\svgwidth{1,6ex}\,\,
\begingroup%
  \makeatletter%
  \providecommand\color[2][]{%
    \errmessage{(Inkscape) Color is used for the text in Inkscape, but the package 'color.sty' is not loaded}%
    \renewcommand\color[2][]{}%
  }%
  \providecommand\transparent[1]{%
    \errmessage{(Inkscape) Transparency is used (non-zero) for the text in Inkscape, but the package 'transparent.sty' is not loaded}%
    \renewcommand\transparent[1]{}%
  }%
  \providecommand\rotatebox[2]{#2}%
  \ifx\svgwidth\undefined%
    \setlength{\unitlength}{49.54732089bp}%
    \ifx\svgscale\undefined%
      \relax%
    \else%
      \setlength{\unitlength}{\unitlength * \real{\svgscale}}%
    \fi%
  \else%
    \setlength{\unitlength}{\svgwidth}%
  \fi%
  \global\let\svgwidth\undefined%
  \global\let\svgscale\undefined%
  \makeatother%
  \begin{picture}(1,1.00059637)%
    \put(0,0){\includegraphics[width=\unitlength,page=1]{PushOff.pdf}}%
  \end{picture}%
\endgroup%
\,\,} K_4(-1)$ & $t=-2-m$ &$k=0$\\
  \midrule
  (5) &$K(r)=K(\frac{1}{4})$ &$t=-2-m$ &$k=0$\\
  \midrule
  (6)\\ (7) &$K(r)=K(+1){\def\svgwidth{1,6ex}\,\,
\begingroup%
  \makeatletter%
  \providecommand\color[2][]{%
    \errmessage{(Inkscape) Color is used for the text in Inkscape, but the package 'color.sty' is not loaded}%
    \renewcommand\color[2][]{}%
  }%
  \providecommand\transparent[1]{%
    \errmessage{(Inkscape) Transparency is used (non-zero) for the text in Inkscape, but the package 'transparent.sty' is not loaded}%
    \renewcommand\transparent[1]{}%
  }%
  \providecommand\rotatebox[2]{#2}%
  \ifx\svgwidth\undefined%
    \setlength{\unitlength}{49.54732089bp}%
    \ifx\svgscale\undefined%
      \relax%
    \else%
      \setlength{\unitlength}{\unitlength * \real{\svgscale}}%
    \fi%
  \else%
    \setlength{\unitlength}{\svgwidth}%
  \fi%
  \global\let\svgwidth\undefined%
  \global\let\svgscale\undefined%
  \makeatother%
  \begin{picture}(1,1.00059637)%
    \put(0,0){\includegraphics[width=\unitlength,page=1]{PushOff.pdf}}%
  \end{picture}%
\endgroup%
\,\,} K_1(\frac{-1}{-t-2}){\def\svgwidth{1,6ex}\,\,
\begingroup%
  \makeatletter%
  \providecommand\color[2][]{%
    \errmessage{(Inkscape) Color is used for the text in Inkscape, but the package 'color.sty' is not loaded}%
    \renewcommand\color[2][]{}%
  }%
  \providecommand\transparent[1]{%
    \errmessage{(Inkscape) Transparency is used (non-zero) for the text in Inkscape, but the package 'transparent.sty' is not loaded}%
    \renewcommand\transparent[1]{}%
  }%
  \providecommand\rotatebox[2]{#2}%
  \ifx\svgwidth\undefined%
    \setlength{\unitlength}{49.54732089bp}%
    \ifx\svgscale\undefined%
      \relax%
    \else%
      \setlength{\unitlength}{\unitlength * \real{\svgscale}}%
    \fi%
  \else%
    \setlength{\unitlength}{\svgwidth}%
  \fi%
  \global\let\svgwidth\undefined%
  \global\let\svgscale\undefined%
  \makeatother%
  \begin{picture}(1,1.00059637)%
    \put(0,0){\includegraphics[width=\unitlength,page=1]{PushOff.pdf}}%
  \end{picture}%
\endgroup%
\,\,} K_{1,1}(\frac{-1}{-k-1}){\def\svgwidth{1,6ex}\,\,
\begingroup%
  \makeatletter%
  \providecommand\color[2][]{%
    \errmessage{(Inkscape) Color is used for the text in Inkscape, but the package 'color.sty' is not loaded}%
    \renewcommand\color[2][]{}%
  }%
  \providecommand\transparent[1]{%
    \errmessage{(Inkscape) Transparency is used (non-zero) for the text in Inkscape, but the package 'transparent.sty' is not loaded}%
    \renewcommand\transparent[1]{}%
  }%
  \providecommand\rotatebox[2]{#2}%
  \ifx\svgwidth\undefined%
    \setlength{\unitlength}{49.54732089bp}%
    \ifx\svgscale\undefined%
      \relax%
    \else%
      \setlength{\unitlength}{\unitlength * \real{\svgscale}}%
    \fi%
  \else%
    \setlength{\unitlength}{\svgwidth}%
  \fi%
  \global\let\svgwidth\undefined%
  \global\let\svgscale\undefined%
  \makeatother%
  \begin{picture}(1,1.00059637)%
    \put(0,0){\includegraphics[width=\unitlength,page=1]{PushOff.pdf}}%
  \end{picture}%
\endgroup%
\,\,} K_{1,1,m}(-1){\def\svgwidth{1,6ex}\,\,
\begingroup%
  \makeatletter%
  \providecommand\color[2][]{%
    \errmessage{(Inkscape) Color is used for the text in Inkscape, but the package 'color.sty' is not loaded}%
    \renewcommand\color[2][]{}%
  }%
  \providecommand\transparent[1]{%
    \errmessage{(Inkscape) Transparency is used (non-zero) for the text in Inkscape, but the package 'transparent.sty' is not loaded}%
    \renewcommand\transparent[1]{}%
  }%
  \providecommand\rotatebox[2]{#2}%
  \ifx\svgwidth\undefined%
    \setlength{\unitlength}{49.54732089bp}%
    \ifx\svgscale\undefined%
      \relax%
    \else%
      \setlength{\unitlength}{\unitlength * \real{\svgscale}}%
    \fi%
  \else%
    \setlength{\unitlength}{\svgwidth}%
  \fi%
  \global\let\svgwidth\undefined%
  \global\let\svgscale\undefined%
  \makeatother%
  \begin{picture}(1,1.00059637)%
    \put(0,0){\includegraphics[width=\unitlength,page=1]{PushOff.pdf}}%
  \end{picture}%
\endgroup%
\,\,} K_{1,1,m,2}(-1)$ &$t< -2$ &$k\leq -2$\\
  \midrule
  (8) &$K(r)=K(+1){\def\svgwidth{1,6ex}\,\,
\begingroup%
  \makeatletter%
  \providecommand\color[2][]{%
    \errmessage{(Inkscape) Color is used for the text in Inkscape, but the package 'color.sty' is not loaded}%
    \renewcommand\color[2][]{}%
  }%
  \providecommand\transparent[1]{%
    \errmessage{(Inkscape) Transparency is used (non-zero) for the text in Inkscape, but the package 'transparent.sty' is not loaded}%
    \renewcommand\transparent[1]{}%
  }%
  \providecommand\rotatebox[2]{#2}%
  \ifx\svgwidth\undefined%
    \setlength{\unitlength}{49.54732089bp}%
    \ifx\svgscale\undefined%
      \relax%
    \else%
      \setlength{\unitlength}{\unitlength * \real{\svgscale}}%
    \fi%
  \else%
    \setlength{\unitlength}{\svgwidth}%
  \fi%
  \global\let\svgwidth\undefined%
  \global\let\svgscale\undefined%
  \makeatother%
  \begin{picture}(1,1.00059637)%
    \put(0,0){\includegraphics[width=\unitlength,page=1]{PushOff.pdf}}%
  \end{picture}%
\endgroup%
\,\,} K_2(\frac{-1}{-k-1}){\def\svgwidth{1,6ex}\,\,
\begingroup%
  \makeatletter%
  \providecommand\color[2][]{%
    \errmessage{(Inkscape) Color is used for the text in Inkscape, but the package 'color.sty' is not loaded}%
    \renewcommand\color[2][]{}%
  }%
  \providecommand\transparent[1]{%
    \errmessage{(Inkscape) Transparency is used (non-zero) for the text in Inkscape, but the package 'transparent.sty' is not loaded}%
    \renewcommand\transparent[1]{}%
  }%
  \providecommand\rotatebox[2]{#2}%
  \ifx\svgwidth\undefined%
    \setlength{\unitlength}{49.54732089bp}%
    \ifx\svgscale\undefined%
      \relax%
    \else%
      \setlength{\unitlength}{\unitlength * \real{\svgscale}}%
    \fi%
  \else%
    \setlength{\unitlength}{\svgwidth}%
  \fi%
  \global\let\svgwidth\undefined%
  \global\let\svgscale\undefined%
  \makeatother%
  \begin{picture}(1,1.00059637)%
    \put(0,0){\includegraphics[width=\unitlength,page=1]{PushOff.pdf}}%
  \end{picture}%
\endgroup%
\,\,} K_{2,m}(-1){\def\svgwidth{1,6ex}\,\,
\begingroup%
  \makeatletter%
  \providecommand\color[2][]{%
    \errmessage{(Inkscape) Color is used for the text in Inkscape, but the package 'color.sty' is not loaded}%
    \renewcommand\color[2][]{}%
  }%
  \providecommand\transparent[1]{%
    \errmessage{(Inkscape) Transparency is used (non-zero) for the text in Inkscape, but the package 'transparent.sty' is not loaded}%
    \renewcommand\transparent[1]{}%
  }%
  \providecommand\rotatebox[2]{#2}%
  \ifx\svgwidth\undefined%
    \setlength{\unitlength}{49.54732089bp}%
    \ifx\svgscale\undefined%
      \relax%
    \else%
      \setlength{\unitlength}{\unitlength * \real{\svgscale}}%
    \fi%
  \else%
    \setlength{\unitlength}{\svgwidth}%
  \fi%
  \global\let\svgwidth\undefined%
  \global\let\svgscale\undefined%
  \makeatother%
  \begin{picture}(1,1.00059637)%
    \put(0,0){\includegraphics[width=\unitlength,page=1]{PushOff.pdf}}%
  \end{picture}%
\endgroup%
\,\,} K_{2,m,2}(-1)$ &$t=-2$ &$k\leq -2$\\
  \midrule
  (9) &$K(r)=K(+1){\def\svgwidth{1,6ex}\,\,
\begingroup%
  \makeatletter%
  \providecommand\color[2][]{%
    \errmessage{(Inkscape) Color is used for the text in Inkscape, but the package 'color.sty' is not loaded}%
    \renewcommand\color[2][]{}%
  }%
  \providecommand\transparent[1]{%
    \errmessage{(Inkscape) Transparency is used (non-zero) for the text in Inkscape, but the package 'transparent.sty' is not loaded}%
    \renewcommand\transparent[1]{}%
  }%
  \providecommand\rotatebox[2]{#2}%
  \ifx\svgwidth\undefined%
    \setlength{\unitlength}{49.54732089bp}%
    \ifx\svgscale\undefined%
      \relax%
    \else%
      \setlength{\unitlength}{\unitlength * \real{\svgscale}}%
    \fi%
  \else%
    \setlength{\unitlength}{\svgwidth}%
  \fi%
  \global\let\svgwidth\undefined%
  \global\let\svgscale\undefined%
  \makeatother%
  \begin{picture}(1,1.00059637)%
    \put(0,0){\includegraphics[width=\unitlength,page=1]{PushOff.pdf}}%
  \end{picture}%
\endgroup%
\,\,} K_1(\frac{-1}{-t-2}){\def\svgwidth{1,6ex}\,\,
\begingroup%
  \makeatletter%
  \providecommand\color[2][]{%
    \errmessage{(Inkscape) Color is used for the text in Inkscape, but the package 'color.sty' is not loaded}%
    \renewcommand\color[2][]{}%
  }%
  \providecommand\transparent[1]{%
    \errmessage{(Inkscape) Transparency is used (non-zero) for the text in Inkscape, but the package 'transparent.sty' is not loaded}%
    \renewcommand\transparent[1]{}%
  }%
  \providecommand\rotatebox[2]{#2}%
  \ifx\svgwidth\undefined%
    \setlength{\unitlength}{49.54732089bp}%
    \ifx\svgscale\undefined%
      \relax%
    \else%
      \setlength{\unitlength}{\unitlength * \real{\svgscale}}%
    \fi%
  \else%
    \setlength{\unitlength}{\svgwidth}%
  \fi%
  \global\let\svgwidth\undefined%
  \global\let\svgscale\undefined%
  \makeatother%
  \begin{picture}(1,1.00059637)%
    \put(0,0){\includegraphics[width=\unitlength,page=1]{PushOff.pdf}}%
  \end{picture}%
\endgroup%
\,\,} K_{1,m+1}(-1){\def\svgwidth{1,6ex}\,\,
\begingroup%
  \makeatletter%
  \providecommand\color[2][]{%
    \errmessage{(Inkscape) Color is used for the text in Inkscape, but the package 'color.sty' is not loaded}%
    \renewcommand\color[2][]{}%
  }%
  \providecommand\transparent[1]{%
    \errmessage{(Inkscape) Transparency is used (non-zero) for the text in Inkscape, but the package 'transparent.sty' is not loaded}%
    \renewcommand\transparent[1]{}%
  }%
  \providecommand\rotatebox[2]{#2}%
  \ifx\svgwidth\undefined%
    \setlength{\unitlength}{49.54732089bp}%
    \ifx\svgscale\undefined%
      \relax%
    \else%
      \setlength{\unitlength}{\unitlength * \real{\svgscale}}%
    \fi%
  \else%
    \setlength{\unitlength}{\svgwidth}%
  \fi%
  \global\let\svgwidth\undefined%
  \global\let\svgscale\undefined%
  \makeatother%
  \begin{picture}(1,1.00059637)%
    \put(0,0){\includegraphics[width=\unitlength,page=1]{PushOff.pdf}}%
  \end{picture}%
\endgroup%
\,\,} K_{1,m+1,2}(-1)$ &$t<-2$ &$k=-1$\\
  \midrule
  (10) &$K(r)=K(+1){\def\svgwidth{1,6ex}\,\,
\begingroup%
  \makeatletter%
  \providecommand\color[2][]{%
    \errmessage{(Inkscape) Color is used for the text in Inkscape, but the package 'color.sty' is not loaded}%
    \renewcommand\color[2][]{}%
  }%
  \providecommand\transparent[1]{%
    \errmessage{(Inkscape) Transparency is used (non-zero) for the text in Inkscape, but the package 'transparent.sty' is not loaded}%
    \renewcommand\transparent[1]{}%
  }%
  \providecommand\rotatebox[2]{#2}%
  \ifx\svgwidth\undefined%
    \setlength{\unitlength}{49.54732089bp}%
    \ifx\svgscale\undefined%
      \relax%
    \else%
      \setlength{\unitlength}{\unitlength * \real{\svgscale}}%
    \fi%
  \else%
    \setlength{\unitlength}{\svgwidth}%
  \fi%
  \global\let\svgwidth\undefined%
  \global\let\svgscale\undefined%
  \makeatother%
  \begin{picture}(1,1.00059637)%
    \put(0,0){\includegraphics[width=\unitlength,page=1]{PushOff.pdf}}%
  \end{picture}%
\endgroup%
\,\,} K_{m+2}(-1){\def\svgwidth{1,6ex}\,\,\,\,} K_{m+2,2}(-1)$ &$t=-1$ &$k<-2$\\
  \midrule
  (11) &$K(r)=K(\frac{1}{2}){\def\svgwidth{1,6ex}\,\,\,\,} K_1(\frac{-1}{-k-2}){\def\svgwidth{1,6ex}\,\,\,\,} K_{1,m}(-1){\def\svgwidth{1,6ex}\,\,\,\,} K_{1,m,2}(-1)$ &$t=-1$ &$k<-2$\\
  \midrule
  (12) &$K(r)=K(\frac{1}{2}){\def\svgwidth{1,6ex}\,\,\,\,} K_{m+1}(-1){\def\svgwidth{1,6ex}\,\,\,\,} K_{m+1,2}(-1)$ &$t=-1$ &$k=-2$\\
  \midrule
  (13) &$K(r)=K(\frac{1}{m+2}){\def\svgwidth{1,6ex}\,\,\,\,} K_3(-1)$ &$t=-1$ &$k=-1$\\
  \midrule
  (14) &$K(r)=K(+1){\def\svgwidth{1,6ex}\,\,\,\,} K_1(\frac{-1}{-t-1}){\def\svgwidth{1,6ex}\,\,\,\,} K_{1,k-1}(-1){\def\svgwidth{1,6ex}\,\,\,\,} K_{1,k-1,m-1}(-1){\def\svgwidth{1,6ex}\,\,\,\,} K_{1,k-1,m-1,2}(-1)$ &$t\leq -2$ &$k>0$\\
  \midrule
  (15)\\ (16) &$K(r)=K(\frac{1}{2}){\def\svgwidth{1,6ex}\,\,\,\,} K_1(\frac{-1}{-k-2}){\def\svgwidth{1,6ex}\,\,\,\,} K_{1,3}(-1){\def\svgwidth{1,6ex}\,\,\,\,} K_{1,3,m-1}(-1)$ &$t=-1$ &$k<-2$\\
  \midrule
  (17) &$K(r)=K(\frac{1}{2}){\def\svgwidth{1,6ex}\,\,\,\,} K_4(-1){\def\svgwidth{1,6ex}\,\,\,\,} K_{4,m-1}(-1)$ &$t=-1$ &$k=-2$\\
  \midrule
  (18)\\ (19)\\ (20) &$K(r)=K(+1){\def\svgwidth{1,6ex}\,\,\,\,} K_1(\frac{-1}{-t-2}){\def\svgwidth{1,6ex}\,\,\,\,} K_{1,1}(\frac{-1}{-k-1}{\def\svgwidth{1,6ex}\,\,\,\,} K_{1,1,3}(-1){\def\svgwidth{1,6ex}\,\,\,\,} K_{1,1,3,m-1}(-1)$ &$t<-2$ &$k\leq -2$\\
  \midrule
  (21)\\ (22) &$K(r)=K(+1){\def\svgwidth{1,6ex}\,\,\,\,} K_2(\frac{-1}{-k-1}){\def\svgwidth{1,6ex}\,\,\,\,} K_{2,3}(-1){\def\svgwidth{1,6ex}\,\,\,\,} K_{2,3,m-1}(-1)$ &$t=-2$ &$k\leq -2$\\
  \midrule
  (23) &$K(r)=K(+1){\def\svgwidth{1,6ex}\,\,\,\,} K_1(\frac{-1}{-t-2}){\def\svgwidth{1,6ex}\,\,\,\,} K_{1,4}(-1){\def\svgwidth{1,6ex}\,\,\,\,} K_{1,4,m-1}(-1)$ &$t<-2$ &$k=-1$\\
  \midrule
  (24)\\ (25) &$K(r)=K(+1){\def\svgwidth{1,6ex}\,\,\,\,} K_5(-1){\def\svgwidth{1,6ex}\,\,\,\,} K_{5,m-1}(-1)$ &$t=-2$ &$k=-1$\\
  \midrule
  (26) &$K(r)=K(\frac{1}{5}){\def\svgwidth{1,6ex}\,\,\,\,} K_{m}(-1)$ &$t=-1$ &$k=-1$\\
  \midrule
  (27) &$K(r)=K(+1){\def\svgwidth{1,6ex}\,\,\,\,} K_1(\frac{-1}{-t-1}){\def\svgwidth{1,6ex}\,\,\,\,} K_{1,k-1}(-1){\def\svgwidth{1,6ex}\,\,\,\,}_{1,k-1,2}(-1){\def\svgwidth{1,6ex}\,\,\,\,} K_{1,k-1,2,m-1}(-1)$ &$t<-1$ &$k>0$\\
  \midrule
  (28) &$K(r)=K(+1){\def\svgwidth{1,6ex}\,\,\,\,} K_1(\frac{-1}{-t-5}){\def\svgwidth{1,6ex}\,\,\,\,} K_{1,m}(-1)$ &$t<-5$ &$k=0$\\
  \midrule
  (29) &$K(r)=K(+1){\def\svgwidth{1,6ex}\,\,\,\,} K_{m+1}(-1)$ &$t=-5$ &$k=0$\\
  \midrule
  (30) &$K(r)=K(\frac{1}{m+1})$ &$t=-4$ &$k=0$\\
  \bottomrule
\end{tabular}  
\end{table}
}}

   \noindent\textbf{Families (6)-(13):} These families correspond to the different possibilities of $r_c=\frac{-(4m+3)(1-kt)-4t}{4-k(4m+3)}$ with $k<0$ that yield overtwisted contact structures.
 
   \noindent\textbf{Families (15)-(30):} These families correspond to the different possibilities of $r_c=\frac{-(4m+3)(1-kt)-t(m+1)}{m+1-k(4m+3)}$ that yield overtwisted contact structures.

In Table~\ref{tab:surgery_representation} we have listed the surgery diagrams of the above families.
\end{proof}

From Theorem~\ref{thm:lens_rational} we can deduce Corollary~\ref{cor:lens_rational} which was stated in the introduction. More precisely, we will prove the following.
\begin{corollary}
    For every $m\geq1$ and every $N\in\Z$ there exists a unique overtwisted contact structure $\xi_N^m$ on $L(4m+3,4)$ with homotopical invariants
    \begin{align*}
\big(\e(\xi_N^m),\de_3(\xi_N^m)\big)=\begin{cases}
            \left(0,N+\frac{1}{2}\right) \text{ if } m\in2\Z+1\\
            \left(1,N-\frac{1}{4m+3}+\frac{1}{2}\right) \text{ if } m\in2\Z.
        \end{cases}
    \end{align*}
    For any odd integer $N\geq\max\{7,m+3\}$ we have $\cs(L(4m+3,3),\xi_N^m)>1$.
\end{corollary}

\begin{proof}[Proof of Corollary~\ref{cor:lens_rational}]
First, we prove the existence of the contact structures $\xi_N^m$. We distinguish by the parity of $m$. If $m=2a+1$, we consider the tight contact structure $\xi^m$ on $L(4m+3,4)$ with surgery diagram from Figure~\ref{fig:tight_lens} with $\rot(K_1)=\rot(K_2)=0$. From that diagram we compute that $e_m=\e(\xi^m)=0$ and $d_m=\de_3(\xi^m)=\frac{1}{2}$. 

If $m=2a$, we consider the tight contact structure $\xi^m$ on $L(4m+3,4)$ with surgery diagram from Figure~\ref{fig:tight_lens} with $\rot(K_1)=0$ and $\rot(K_2)=1$. From that diagram we compute that $e_m=\e(\xi^m)=1$ and $d_m=\de_3(\xi^m)=\frac{1}{2}-\frac{1}{4m+3}$.

Now let $\xi^m$ be a contact structure on $L(4m+3,4)$ with Euler class $e_m$ and $d_3$-invariant $d_m\in\Q$. For every $N\in\Z$ we consider the overtwisted contact manifold 
$$\big(L(4m+3,4),\xi^m_N\big):=\big(L(4m+3,4),\xi\big)\#\big(S^3,\xi_N\big).$$ 
Since the $\de_3$-invariant behaves additive under connected sum and the Euler class does not change, we see that this contact structure has again Euler class $e_m$ and $d_3$-invariant $d_m+N$. This proves the existence of the overtwisted contact structures claimed in the corollary.

Since the first homology of $L(4m+3,4)$ has no $2$-torsion, an overtwisted contact structure is uniquely determined by $\e$ and $\de_3$ and thus $\xi_N^m$ is uniquely determined by its homotopical invariants.

To show that infinitely many of these contact structures have $\cs>1$ we check which of the $\xi_N^m$ appear in the lists from Theorem~\ref{thm:lens_overtwisted} and~\ref{thm:lens_rational}. For that we will check which of those have homotopical invariants of the form $(e,d_3)=(e_m,d_m+N)$.

\noindent\textbf{(A) Obstructions from the Euler class:}

\noindent\textbf{Family (1):} In this family the Euler class is $k+2m+2$, for $k=-1,0,1,\ldots,m$. Thus the Euler class takes values between $2m+1$ and $3m+2$ and in particular is not $0$ or $1$.

\noindent\textbf{Family (5):} The Euler class is $-m+2l$, for $l=0,1,\ldots,m$. If $m$ is odd this cannot be $0$ and if $m$ is even this cannot be $1$. 

\noindent\textbf{Family (4):} In this family the Euler class is $2l-(y+1)(m+1)$ for $y\in\{0,\pm2,\pm4\}$ and $l=0,1,\ldots,m+1$. We check that the for $m=2a+1$ ($m=2a$) the Euler class is $0$ ($1$) if and only if $y=0$ and $l=a+1$. Then we plug in these values in the $\de_3$-invariant to see that if $m$ is even we do not get any contact structure of the desired form. On the other hand, for every odd $m$ we get a contact structure with $e=0$ and $\de_3=1/2$.

\noindent{\textbf{(B) Universally bounded $\de_3$-invariants:}}

\noindent\textbf{Families (3), (6), (9), (20), (23), (29):} In these cases we observe that $d_3$ is always negative.

\noindent\textbf{Families (7), (8), (10)--(13),  (15), (17), (22), (26), (28), (30):} In these cases we can bound $d_3$ from above by $7$.

\noindent\textbf{Family (2):} Here we observe that if the Euler class is $0$ or $1$, then the $\de_3$-invariant is negative.

\noindent{\textbf{(C) Bounded $\de_3$-invariants:}}

\noindent\textbf{Families (16), (18), (19),  (21), (24), (25):} In these cases we estimate $d_3$ to be smaller than $m+3$.

\noindent{\textbf{(D) Even $\de_3$-invariants:}}

\noindent\textbf{Families (14), (27):} We will show that in these two cases it follows that if $m$ is odd (even) and $\de_3\in 1/2+\Z$ ($\de_3\in 1/2 -1/(4m+3)+\Z$) then it follows that $\de_3\in 1/2+2\Z$ ($\de_3\in 1/2 -1/(4m+3)+2\Z$). In other words, the $N$ in the statement of the corollary is always an even number. 

For that, we consider all possible parities of $m$, $k$, and $n$ and check the parities of the $\de_3$-invariants. We present the details for Family (14). (Family (27) works similarly.)

In Family (14) the formula for the $\de_3$ invariant consists of $5$-summands. The first summand is $1/2$ which we can ignore. The second two summands are integers and the last two summands are fractions with odd denominator $4m+3$.

The second term is $-(1+n)$ which is odd if and only if $n$ is even. The third term is $k(1+n)^2$ and is odd if and only if $n$ is even and $k$ is odd. The numerator of the last term is a multiple of $1+m$. And the third term has even denominator if and only if $(n+1)(k+1)$ is even which is the case if and only if $n$ or $k$ is odd. By adding up these parities we see that we get in all cases even values of $N$.
\end{proof}

\begin{proof}[Proof of Corollary~\ref{cor:integerlenscor}]
This corollary follows now directly from Corollary~\ref{cor:lens_rational} and Theorem~\ref{thm:lens_overtwisted}.
\end{proof}

\let\MRhref\undefined
\bibliographystyle{hamsalpha}
\bibliography{ref.bib}

\end{document}